\documentclass[a4paper, 11pt]{amsart}
\usepackage{amssymb,amsfonts,amsmath}
\usepackage{hyperref}
\usepackage{todonotes}
\usepackage{comment}
\usepackage{shuffle}

\usepackage[margin=25mm]{geometry}

\usepackage[all,cmtip]{xy}
\usepackage{tikz}
\usetikzlibrary{matrix,arrows,decorations.pathmorphing, decorations.pathreplacing, angles,quotes, automata, matrix, positioning, calc, shapes.multipart}
\usepackage{tikz-cd}

 \DeclareFontFamily{U}{wncy}{}
\DeclareFontShape{U}{wncy}{m}{n}{<->wncyr10}{}
\DeclareSymbolFont{mcy}{U}{wncy}{m}{n}
\DeclareMathSymbol{\Sh}{\mathord}{mcy}{"58}

\newcommand{\W}{\mathbb{W}}
\newcommand{\minor}{\Delta}

\newcommand{\Br}{\mathrm{Br}}

\newcommand{\borel}{\mathcal{B}}

\newcommand{\n}{\mathfrak{n}}
\newcommand{\m}{\mathfrak{m}}

\newcommand{\cC}{\mathcal{C}}
\newcommand{\A}{\mathcal{A}}
\newcommand{\U}{\mathcal{U}}
\newcommand{\cF}{\mathcal{F}}

\newcommand{\cD}{\mathcal{D}}
\newcommand{\cO}{\mathcal{O}}

\newcommand{\cS}{\mathcal{S}}

\newcommand{\per}{\mathcal{D}}

\DeclareMathOperator{\Stab}{Stab}
\newcommand{\seeds}[1]{\Sigma(#1)}

\newcommand{\B}{\widetilde{B}}

\DeclareMathOperator{\Aut}{Aut}

\usepackage{mathrsfs}

\newcommand{\C}{\mathbb{C}}

\newcommand{\Z}{\mathbb{Z}}

\newcommand{\be}{\mathbf{e}}
\newcommand{\bz}{\mathbf{z}}

\newcommand{\std}{\mathrm{std}}

\newcommand{\w}{\mathfrak{w}}

\newcommand{\rank}{\mathrm{rank}}
\newcommand{\BS}{\mathrm{BS}}

\newcommand{\GL}{\operatorname{GL}}

\newcommand{\wt}{\mathtt{w}}

\DeclareMathOperator{\Spec}{Spec}
\DeclareMathOperator{\Gr}{Gr}
\DeclareMathOperator{\diag}{diag}
\DeclareMathOperator{\mult}{mult}

\newcommand{\bx}{\mathbf{x}}
\newcommand{\Fl}{\mathscr{F}}

\newtheorem{theorem}{Theorem}[section]
\newtheorem{proposition}[theorem]{Proposition}
\newtheorem{corollary}[theorem]{Corollary}
\newtheorem{lemma}[theorem]{Lemma}

\newtheorem{conjecture}[theorem]{Conjecture}

\theoremstyle{definition}
\newtheorem{remark}[theorem]{Remark}
\newtheorem{definition}[theorem]{Definition}
\newtheorem{example}[theorem]{Example}

\makeatletter
\makeatletter
\def\@tocline#1#2#3#4#5#6#7{\relax
  \ifnum #1>\c@tocdepth 
  \else
    \par \addpenalty\@secpenalty\addvspace{#2}%
    \begingroup \hyphenpenalty\@M
    \@ifempty{#4}{%
      \@tempdima\csname r@tocindent\number#1\endcsname\relax
    }{%
      \@tempdima#4\relax
    }%
    \parindent\z@ \leftskip#3\relax \advance\leftskip\@tempdima\relax
    \rightskip\@pnumwidth plus4em \parfillskip-\@pnumwidth
    #5\leavevmode\hskip-\@tempdima
      \ifcase #1
       \or\or \hskip 1em \or \hskip 2em \else \hskip 3em \fi%
      #6\nobreak\relax
    \dotfill\hbox to\@pnumwidth{\@tocpagenum{#7}}\par
    \nobreak
    \endgroup
  \fi}
\makeatother

\def\CC{\mathbb{C}}
\def\Gr{\operatorname{Gr}}
\def\HH{\operatorname{HH}}
\def\QH{\operatorname{QH}}
\def\Fuk{\operatorname{Fuk}}
\def\Sing{\operatorname{Sing}}
\def\cO{\mathcal{O}}
\def\Crit{\operatorname{Crit}}
\def\Newt{\operatorname{Newt}}
\def\Der{\mathbf{D}}

\title{
Cluster deep loci and mirror symmetry 
}

\author{Marco Castronovo}
\address{Department of Mathematics, Columbia University, 2990 Broadway, New York, NY 10027}
\email{marco.castronovo@columbia.edu}
\author{Mikhail Gorsky}
\address{Faculty of Mathematics, University of Vienna, Oskar-Morgenstern-Platz 1, 1090 Vienna, Austria}
\email{mikhail.gorskii@univie.ac.at}
\author{Jos\'e Simental}
\address{Instituto de Matem\'aticas, Universidad Nacional Aut\'onoma de M\'exico. Ciudad Universitaria, CDMX, M\'exico}
\email{simental@im.unam.mx}
\author{David E Speyer}
\address{Department of Mathematics, University of Michigan, 2844 East Hall, 530 Church Street, Ann Arbor, MI 48109-1043, USA}
\email{speyer@umich.edu}

\begin{document}

\begin{abstract}
Affine cluster varieties are covered up to codimension 2 by open algebraic tori. We put forth a general
conjecture (based on earlier conversation between Vivek Shende and the last author) characterizing their deep locus, i.e. the complement of all cluster charts,
as the locus of points with non-trivial stabilizer under the action of cluster
automorphisms. We use the diagrammatics of Demazure weaves to verify the conjecture for skew-symmetric cluster varieties of finite cluster type with arbitrary choice of frozens and for the top open positroid strata of Grassmannians $\Gr(2,n)$ and $\Gr(3,n)$. We illustrate how this already has
applications in symplectic topology and mirror symmetry, by proving that the Fukaya category
of Grassmannians $\Gr(2,2n+1)$ is split-generated by finitely many Lagrangian tori, and
homological mirror symmetry holds with a Landau--Ginzburg model proposed by Rietsch.
Finally, we study the geometry of the deep locus, and find that it can
be singular and have several irreducible components of different dimensions, but they all are again cluster varieties in our examples in really full rank cases.
\end{abstract}

\maketitle

\tableofcontents

\section{Introduction}

\subsection{What clusters cannot see}

An affine cluster variety is the affine scheme $\A = \Spec A$ associated to a cluster algebra
$A$ in the sense of Fomin--Zelevinsky \cite{FZ}. Such algebras come with distinguished
generators, grouped into finite sets $\mathbf{x}(t)$ parametrized by a (typically
infinite) collection of seeds $t\in\Sigma(A)$. The size $|\mathbf{x}(t)|$ is independent
of $t$, and part of the structure is a rule that transforms $\mathbf{x}(t)$ into
$\mathbf{x}(t')$ for any pair $t,t'\in\Sigma(A)$, via a sequence of rational functions
called mutations. A key feature of cluster algebras is the Laurent phenomenon,
according to which any generator is a Laurent polynomial in $\mathbf{x}(t)$
for any $t\in\Sigma(A)$. The elements of $\mathbf{x}(t)$ can be thought of as coordinates
of an open torus $T(t)\subseteq\A$: the associated cluster chart. It is then tempting
to think of $\A$ as a special kind of manifold, with an atlas of cluster charts
whose transition maps are birational. However, this is not quite correct: for example
$\A$ may have singularities, and they are necessarily contained in $\cD(\A)=\A\setminus \bigcup_{t\in\Sigma(A)}T(t)$.
This closed subscheme was first named by Muller \cite{muller}, who called it the deep locus.

\subsection{Mirror symmetry interpretation}

Among the first examples of cluster varieties were the top positroid strata $\A(k,n)$
\cite{FZ, KLS13, Scott}. In mirror symmetry, these are widely believed to support
Landau--Ginzburg models for the complex Grassmannians $\Gr(k,n)$. For example the quantum
cohomology, which is defined using genus zero closed Gromov--Witten numbers \cite{FP,McS},
was recovered as coordinate ring $\QH(\Gr(k,n))\cong\cO(\Crit(W))$ on the
critical locus of a particular regular function $W\in\cO(\A(k,n))$ \cite{MR, R}.
More recently the Fukaya category $\Fuk(\Gr(k,n))$ was shown to contain Lagrangian tori
$L(t)$ corresponding to certain cluster charts $T(t)\subseteq\A(k,n)$, such that
the restriction $W|_{T(t)}$ is a Laurent polynomial in the variables of the seed $\mathbf{x}(t)$,
which can be interpreted as generating function of genus zero open Gromov--Witten numbers
with boundary condition at $L(t)$ in some cases. The Fukaya category \cite{Sei, Sh}
turned out to be generated by objects supported on the tori $L(t)$ in some cases but not
all \cite{Cas20, Cas23}, an obstruction being the existence of critical points of the
potential $W$ in the deep locus of the top positroid stratum $\A(k,n)$. In a different
direction, the deep locus of $\A(2,n)$ has been related to the space of Maurer--Cartan
deformations of certain immersed Lagrangians in $\Gr(2,n)$ by Hong-Kim-Lau \cite{HKL}.

\subsection{The ``no mysterious points" conjecture.}

The first goal of this article is to propose a conjecture characterizing the deep locus
of general cluster varieties in terms of intrinsic cluster symmetries. We consider a cluster algebra $A$ where all frozen variables are invertible and put $\A = \Spec A$.
We denote by $\Aut(A)$ the group of cluster automorphisms, i.e. those algebra automorphisms
which act on every cluster variable by rescaling. Geometrically, these are the automorphisms which act on every torus $T(t)$ by a translation.
We define the stabilizer locus $\cS(\A)$ to be the set of points of $\A$ with nontrivial stabilizer.
Since $\Aut(A)$ acts freely on the cluster charts $T(t)\subseteq \A$, we have a containment
$\cS(\A)\subseteq \cD(\A)$. 

\begin{conjecture}\label{conj:deep-stabilizer}
If $A$ is skew-symmetric and locally acyclic, then $\cD(\A)=\cS(\A)$.
\end{conjecture}

We will define a point of $\A$ to be ``mysterious" if it lies in the deep locus but not in the stabilizer locus. 
So Conjecture~\ref{conj:deep-stabilizer} says that there are no mysterious points.

This conjecture was first formulated in private conversation between Vivek Shende and the last author, concerning the particular case of open positroid varieties.
The assumptions on the cluster
algebra $A$ of the conjecture are meant to generalize some of the key features of the positroid examples. We establish some facts about the
general conjecture: the statement is independent of the choice
of frozen variables in the
cluster algebra (Proposition \ref{prop:reduce-to-full-rank-1}) and 
its validity only depends on the quasi-isomorphism
type of $A$ (Corollary \ref{cor:quasi-iso}). For simplicity, we did not examine examples where the cluster algebra $A$ is
only skew-symmetrizable and we make no claim or conjecture about those.

The local acyclicity in the statement of Conjecture~\ref{conj:deep-stabilizer} is a technical condition introduced by Muller \cite{muller} which guarantees, in particular, that both $\mathcal{A}$ and $\bigcup_{t\in\Sigma(A)}T(t)$ are finite type algebraic varieties, and that $\A$ is the affinization of $\bigcup_{t\in\Sigma(A)}T(t)$. In general, the affinization of $\bigcup_{t\in\Sigma(A)}T(t)$ is called the upper cluster variety, denoted $\mathcal{U}$, and the algebra of regular functions on $\bigcup_{t\in\Sigma(A)}T(t)$ is called the upper cluster algebra. Beyond the locally acyclic case, there might be a discrepancy between deep loci in $\mathcal{A}$ and $\mathcal{U}$. We do not investigate this issue in general, but in Section~\ref{subsec:markov-counter-example} we perform a detailed examination of the case of the Markov cluster algebra, related to the skein algebra of the once punctured torus, which is an archetypical example where $A \subsetneq U$. 
In the Markov example, we find that Conjecture~\ref{conj:deep-stabilizer} holds for the cluster algebra both without frozens and with principal coefficients, while its counterpart for the upper cluster algebra holds in the coefficient-free case but fails in the case of principal coefficients. In the latter case, interestingly, $\cS(\U)$ agrees with the complement in $\U$ to the union of cluster tori and certain toric charts corresponding to tagged triangulations, considered in \cite{Z20}. We see that an analogue of Conjecture~\ref{conj:deep-stabilizer} for {\bf upper} cluster varieties fails in general beyond the locally acyclic case, and its status depends on the choice of frozen variables.

\subsection{Evidence via braids and weaves}
Using the technology of braid varieties and weaves, we are able to prove that there are no mysterious points in several cases. Specifically, we prove:

\begin{theorem}\label{thm:evidence}
The cluster variety $\A$ has no mysterious points whenever the principal part of the quiver for $A$ is mutation equivalent to a quiver $Q_{k,l}$ of the form

\begin{center}
\begin{tikzpicture}
\node at (0,0) {$\circ$};
\draw[->] (0.2,0)--(0.8,0);
\node at (1,0) {$\circ$};
\draw [->] (1.2, 0) -- (1.6,0);
\node at (2,0) {$\ldots$};
\draw [->] (2.4, 0) -- (2.6,0);
\node at (3,0) {$\circ$};
\draw[->] (3.4, 0) -- (3.8, 0);
\node at (4,0) {$\circ$};
\draw[->] (4.2, 0) -- (4.6, 0);
\node at (5,0) {$\ldots$};
\draw[->] (5.4, 0) -- (5.8, 0);
\node at (6,0) {$\circ$};

\node at (0,-1) {$\circ$};
\draw[->] (0.2,-1)--(0.8,-1);
\node at (1,-1) {$\circ$};
\draw [->] (1.2, -1) -- (1.6,-1);
\node at (2,-1) {$\ldots$};
\draw [->] (2.4, -1) -- (2.8,-1);
\node at (3,-1) {$\circ$};

\draw[->] (0, -0.9) -- (0, -0.1);
\draw[->] (1, -0.9) -- (1, -0.1);
\draw[->] (3, -0.8) -- (3, -0.1);

\draw[->] (0.8, -0.2)--(0.2, -0.8);
\draw[->] (1.8, -0.2)--(1.2, -0.8);
\draw[->] (2.8, -0.2)--(2.2, -0.8);

\draw[decoration={brace,raise=5pt},decorate]
  (-0.2, 0) -- node[above=5pt] {$k$} (3.2, 0);
\draw[decoration={brace,raise=5pt},decorate]
  (3.8, 0) -- node[above=5pt] {$l$} (6.2, 0);
\end{tikzpicture}
\end{center}

for some $k, l \geq 0$. This includes
\begin{enumerate}
	\item 
    all cluster algebras of simply laced finite cluster type, and
	\item the cluster algebras of top positroid strata $\A(2,n)$ and $\A(3,n)$.
\end{enumerate}
\end{theorem}

Cluster algebras of finite cluster type have been classified by Fomin--Zelevinsky \cite{FZ03},
and are among the simplest examples where to test any conjecture in the field. Because they have
only finitely many cluster tori, the deep locus is in principle easy to compute. We restrict to simply
laced types due to the assumption that the cluster algebra is skew-symmetric in
Conjecture \ref{conj:deep-stabilizer}.

While cluster algebras of top positroid strata $\A(2,n)$ are of
finite type, the ones of $\A(3,n)$ for large $n$ are not, so Theorem~\ref{thm:evidence} gives
evidence for that there are no mysterious points in infinite type. The starting point for the proof of the theorem is to
realize these cluster varieties (with a suitable choice of frozens) as braid varieties $\A=X(\beta)$ in the sense of Casals, Gorsky, Gorsky and Simental
\cite{CGGS1}. A braid variety $X(\beta)$ is an affine variety associated to a positive braid
word $\beta$, whose points parameterize flags of linear spaces in relative positions
determined by the word. Their cluster structure has been described by Casals, Gorsky, Gorsky, Le, Shen, Simental \cite{CGGLSS}
using the formalism of weaves introduced by Casals and Zaslow \cite{CZ}, see also \cite{GLS, GLSS}. There are two key advantages in studying $\cD(\A)$ and $\cS(\A)$ via braid varieties. The first is that, at least for Bott--Samelson
braids $\beta$ (which include all of the examples in Theorem \ref{thm:evidence}),  the action of the
cluster automorphism group $\Aut(A)$ on $\A=X(\beta)$ agrees
with a torus action on $X(\beta)$ defined combinatorially; using the latter one can more easily
compute stabilizers of points and determine the stabilizer locus $\cS(\A)$.
The second advantage is that one can inductively construct cluster charts in $\A=X(\beta)$
using a special class of weave diagrams called Demazure weaves, which have a vertical
orientation; this often allows to show that certain points are not in the deep locus
$\cD(\A)$ by explicitly constructing a cluster chart that contains them. A subtle point
in the proof of Theorem \ref{thm:evidence} is that one may need to cyclically shift the
braid word $\beta$ giving the identification $\A=X(\beta)$ and words appearing as  horizontal cross-sections of weaves in order to prove that
$\cD(\A)=\cS(\A)$; this can be done because such shifts induce cluster quasi-isomorphisms
on $\A$, but Conjecture \ref{conj:deep-stabilizer} is independent of the
quasi-isomorphism type thanks to the general results of Section \ref{sec:deep-locus}. We expect that our method of proof can be expanded to more general braid varieties, including in particular all double Bott--Samelson varieties and all positroid strata in $\Gr(k,n)$. The combinatorial analysis, however, becomes very
complicated beyond the cases we study here.

We note that all the cluster varieties of type $Q_{k,l}$ with a suitable choice of frozens can also be realized as (possibly non-maximal) open positroid strata in $\Gr(3,n)$ for some $n$ depending on $k$ and $l$. For positroids, some cluster variables are Pl\"ucker coordinates, and cluster charts where all variables are Pl\"ucker coordinates can be described via plabic graphs, which is a more common description of  corresponding cluster structures. In Example~\ref{eg:NonPlabic}, we give an example of type $A_1$ inside $\Gr(3, 6)$ where there are points with non-free torus action in the complement to such charts. In other words, it is not enough to use charts corresponding to plabic graphs to cover all points with trivial stabilizer. This is yet another advantage of the weave approach.

\subsection{Back to mirror symmetry}

The results of Theorem~\ref{thm:evidence} already
have consequences for mirror symmetry. Denote $\Der\Fuk_\lambda(\Gr(2,n))$ the derived
Fukaya category of Lagrangian branes in $\Gr(2,n)$ with curvature $\lambda\in\CC$,  and
$\Der\Sing(W^{-1}(\lambda))$ the derived category of singularities of the fiber over
$\lambda\in\CC$ of the Landau--Ginzburg potential $W\in\cO(\A(n-2,n))$ defined by 
Rietsch \cite{MR, R}.

\begin{theorem}\label{thm:hms}(Theorem \ref{thm:fuk-gr(2,n)})
If $n$ is odd, then $\Der\Fuk_\lambda(\Gr(2,n))$ is generated by objects supported on
Lagrangian tori and $\Der\Fuk_\lambda(\Gr(2,n))\simeq \Der\Sing(W^{-1}(\lambda))$ for
all $\lambda\in\CC$.
\end{theorem}

This result was established for certain infinite classes of odd $n$ in \cite{Cas20, Cas23},
where Lagrangian tori $L(t)\subset\Gr(2,n)$ corresponding to the seeds
$t\in\Sigma(A(n-2,n))$ were constructed. The critical points $p$ of the
Landau--Ginzburg potential $W$ that belong to a cluster chart $T(t)\subseteq \A(n-2,n)$
were interpreted there as holonomies of local systems $\xi_p$ on $L(t)\subset\Gr(2,n)$ such that
$(L(t),\xi_p)$ is a nontrivial object of $\Der\Fuk_\lambda(\Gr(2,n))$. The interpretation
relies on the identification of $W|_{T(t)}$ with the Floer-theoretic disk potential
of the Lagrangian $L(t)$. The new input
consists in the observation that when $n$ is odd, $\cD(\A(n-2,n))=\emptyset$ by
Theorem \ref{thm:evidence}, and thus any critical point $p$ will have some (and
possibly many) seed $t\in\Sigma(A(n-2,n))$ such that $p\in T(t)$. This allows to prove
generation of the Fukaya category and mirror symmetry with the category of singularities
by invoking a standard generation criterion for the Fukaya category due to Sheridan \cite{Sh}.

We see that, for this sort of application to mirror symmetry, it is important to understand when the deep locus $\cD(\A)$    is empty.
A necessary condition, and according to Conjecture~\ref{conj:deep-stabilizer} a sufficient one, is that the stabilizer locus, $\cS(\A)$, is empty.
Thus, we will remark throughout 
the paper on combinatorial conditions which imply that various stabilizer loci are empty.
For example, for $\A(k,n)$, we will show (Corollary~\ref{cor:GCDCriterion}) that the stabilizer locus is empty if and only if $\gcd(k,n)=1$.
Combined with Theorem~\ref{thm:evidence}, this implies that the deep locus is empty for $\A(3,n)$ if $3$ does not divide $n$.

It is not clear, however, that we can derive the same consequences for  $\Der\Fuk_\lambda(\Gr(3,n))$ with $3$ not
dividing $n$. The main issue is that we don't know whether
$W_{|T(t)}$ matches the Floer-theoretic disk potential of the torus $L(t)$ when $k>2$.
A secondary issue is that $\Der\Fuk_\lambda(\Gr(3,n))$ could be impossible to generate
with a single object whenever $\lambda\in\CC$ is an eigenvalue of the operator of
multiplication by $c_1$ acting on quantum cohomology $\QH(\Gr(3,n))$ which has algebraic
multiplicity greater than $1$; such $\lambda$ exist whenever $n > 4$ is even \cite{C22}. In this case, a more sophisticated generation argument
is needed; see \cite{Cas24} for a different approach based on bulk-deformation,
which might eventually yield a generation result with coefficients over the Novikov field
instead of $\CC$.
Even after considering such issues, we believe that the general correspondence
between cluster charts $T(t)\subseteq\A(n-k,n)$ and Lagrangian tori $L(t)\subset\Gr(k,n)$
is still compelling enough to conjecture the following.

\begin{conjecture}
If $\gcd(k, n) = 1$, then $\Der\Fuk_\lambda(\Gr(k,n))$ is generated by objects
supported on tori for all $\lambda\in\CC$.
\end{conjecture}

The conjecture is based on the observation that the stabilizer locus of the Landau--Ginzburg
model $\cS(\A(n-k,n))=\emptyset$ if and only if 
$\gcd(k, n) = 1$. If Conjecture
\ref{conj:deep-stabilizer} holds then $\cD(\A(n-k,n))=\emptyset$ as well in such cases,
and thus one should again expect that any critical point $p$ of the potential $W$
will have some (and possibly many) seed $t\in\Sigma(A(n-k,n))$ such that $p\in T(t)$,
and together the objects $(L(t),\xi_p)$ with $t\in\Sigma(A(n-k,n)$ and $W(p)=\lambda$
will generate $\Der\Fuk_\lambda(\Gr(k,n))$. See Section \ref{conj:generation-by-tori} for a more
general conjecture regarding generation of the Fukaya category by objects supported on tori, which
follows the same logic.

\subsection{Beyond existence: geometry of the deep locus}

Finally, in Section \ref{sec:geometry-of-deep-loci} we explore the geometry of $\cD(\A)$ for the cluster algebras
$A$ studied in Theorem \ref{thm:evidence}. The main lesson we learn is that the deep locus can be quite nasty:
it can be singular, with several irreducible components, and these can have different
dimensions.

\begin{theorem}(Theorem \ref{thm:geometry-deep-locus})
Let $\A$ be a cluster variety of the type considered in Theorem~\ref{thm:evidence}.(1) or (2). Assume moreover that in the Dynkin cases, the variety $\A$ has really full rank. Then:
\begin{itemize}
    \item $\per(\A) = \emptyset$ if and only if $\A$ is of Dynkin type $A_{2n}$, $E_{6}$ or $E_{8}$; or one of the positroid strata $\A(2,2n+1), \A(3, 3n+1)$ or $\A(3,3n+2)$.
    \item $\per(\A)$ is nonempty, smooth and irreducible if and only if $\A$ is of Dynkin type $A_{2n+1}, D_{2n+1}$ or $E_{7}$, or the cluster algebra associated to the positroid stratum $\A(2,2n)$. In this case, $\per(\A)$ is itself a cluster variety of finite cluster type of Dynkin type $A$.
    \item $\per(\A)$ is not smooth and not irreducible if $\A$ is of Dynkin type $D_{2n}$. In this case, $\per(\A)$ has three components, each one of which is itself a cluster variety of Dynkin type $A$. If $2n > 4$, two of these components have the same dimension while the other does not.
    \item $\per(\A)$ is not smooth and not irreducible if $\A = \A(3, 3n)$. In this case, $\per(\A)$ has three irreducible components, each of which is isomorphic to the positroid stratum $\A(2,2n)$.
\end{itemize}
\end{theorem}

Albeit pathological from the point of view of algebraic geometry, we are surprised to
witness that in all cases each irreducible component of the deep locus $\cD(\A)$ is again
a cluster variety. 
While this is not the case for an arbitrary cluster variety  $\mathcal{A}$,  we show in Theorem \ref{thm:components-stabilizer-double-bs} that every irreducible component of the stabilizer locus $\mathcal{S}(X(\beta))$ for an arbitrary double Bott-Samelson variety is again a double Bott-Samelson variety, and so in particular is again a cluster variety. Further, all the intersections of the irreducible components are again double Bott-Samelson varieties and so in particular cluster varieties.  Thus, the same is true for the deep locus $\mathcal{D}(X(\beta))$ as long as Conjecture \ref{conj:deep-stabilizer} holds for $X(\beta)$. Similarly, upon the validity of Conjecture \ref{conj:deep-stabilizer}, every irreducible component and every nonempty intersection of irreducible components of the deep locus of a positroid variety \cite{KLS13} is a cluster variety. We do not know what is the largest possible generality where one can observe this phenomenon.

For an arbitrary braid variety $X(\beta)$, we show that all irreducible components of the analogue of the stabilizer locus for a certain natural torus action (which in general does not coincide with that of $\Aut(X(\beta))$) are again braid varieties and their intersections, when not empty, are also braid varieties. While this result does not immediately fit into the cluster algebra context, we find it interesting that the statement can be interpreted in terms of unlinking the components of the closure of the braid $\beta \Delta^{-1}$, where $\Delta$ is the half-twist braid.

\subsection{Related work}

While working on this project, we learned of work of James Beyer and Greg Muller (in progress) that 
studies the deep locus of cluster algebras arising from surfaces as in the works of
Fomin--Shapiro--Thurston \cite{FST} and Fomin--Thurston \cite{FT}. It would be interesting
to test the deep/stabilizer locus conjecture in such cases, as they are cluster algebras
of finite mutation type, the latter being larger than the class of algebras of finite cluster type but still somewhat tractable.

\section*{Acknowledgements}
This project started at the American Institute of Mathematics (AIM) during the workshop ``Cluster algebras and braid varieties'' in January 2023 co-organized by two of the authors. We are grateful to the other two organisers, Roger Casals and Melissa Sherman-Bennett, and to AIM. We are also grateful to Roger Casals, Eugene Gorsky, Greg Muller, and Daping Weng for interesting discussions related to this work. We thank James Beyer and Greg Muller for sharing their exciting results with us.\\

M.C. was partially supported by an AMS-Simons Travel Grant. The work of M.~G. is a part of a project that has received funding from the European Research Council (ERC) under the European Union's Horizon 2020 research and innovation programme (grant agreement No.~101001159). J.S. was partially supported by CONAHCyT project CF-2023-G-106 and UNAM’s PAPIIT Grant IA102124. D.E.S. was partially supported by NSF grants DMS-1855135, DMS-1854225 and DMS-2246570.

\section{Cluster varieties}
\label{sec:cluster-varieties}

This section reviews some basic definitions in cluster algebras, and establishes the notations and assumptions used throughout the paper. General results and conjectures in this section and in sections \ref{sec:deep-locus}, \ref{sec:braid varieties}, and \ref{sec:mirror-symmetry} make sense in the framework of skew-symmetrizable cluster algebras of geometric type, but to simplify the presentation and since we prove our conjectures only in this restricted setting, {\bf we will always work in the skew-symmetric case}.

\subsection{Cluster algebras}

Let $n,m$ be non-negative integers. Denote by $\cF$ the field of rational functions over $\CC$ in $n+m$ variables. An ice quiver $Q$ is an oriented graph with vertices labeled $\{1,\ldots ,n+m\}$. We will always assume that $Q$ has no loops and no oriented $2$-cycles. The first $n$ vertices are called mutable and the rest frozen. By assumption, there are no arrows between frozen vertices of $Q$.

\begin{definition}
\label{def:exchange-matrix}
The exchange matrix $\tilde{B}(Q)=(b_{ij})$ of an ice quiver $Q$ is the $(n+m)\times n$ integer matrix whose entry $(i,j)$ is $\pm$ the number of arrows connecting vertex $i$ and vertex $j$ in $Q$. The sign is determined by the direction of the arrows. The principal part $B(Q)$ is the $n\times n$ skew-symmetric submatrix corresponding to the mutable vertices of $Q$.
\end{definition}

\begin{definition}\label{def:seed}
A seed in $\cF$ is a pair $t=(\tilde{\mathbf{x}},\tilde{B})$ where:
\begin{itemize}
\item $\tilde{\mathbf{x}}=(x_1,\ldots ,x_{n+m})$ is a tuple of elements of $\cF$ that form a free generating set;
\item  $\tilde{B}=\tilde{B}(Q)$ is the exchange matrix of a quiver $Q$.
\end{itemize}
In the tuple $\tilde{\bf{x}}$, call $x_i\in\cF$ a mutable variable if $1\leq i\leq n$, and a frozen variable otherwise. Denote by $\mathbf{x}$ the tuple of mutable variables. The set $\tilde{\mathbf{x}}$ is called an extended cluster, and $\mathbf{x}$ is called a cluster. The elements of $\mathbf{x}$ are called cluster variables. 
\end{definition}

If $k$ is a mutable node of the quiver $Q$, denote $\mu_k(Q)$ its mutation at $k$, i.e. the quiver obtained as follows: (1) for each path $a\to k\to c$ where $a$ and $c$ are not both frozen, add an arrow $a\to c$; (2) invert all arrows incident to $k$; (3) delete a maximal collection of oriented $2$-cycles that may have been created in steps (1) and (2).

\begin{definition}\label{def:seed-mutation}
If $t=(\tilde{\mathbf{x}},\tilde{B})$ is a seed in $\cF$ and $1\leq k\leq n$, its mutation $\mu_k(t)=t'=(\tilde{\mathbf{x}}',\tilde{B}')$
is the seed where:
\begin{itemize}
    \item $\tilde{\mathbf{x}}'=(x_1',\ldots ,x_{n+m}')$ with $x_i'=x_i$ for $i\neq k$ and $x_k'=x_k^{-1}(\prod_{b_{ik}>0}x_i^{b_{ik}}+\prod_{b_{ik}<0}x_i^{-b_{ik}})$;
    \item $\tilde{B}'=\tilde{B}(\mu_k(Q))$.
\end{itemize}
\end{definition}

\begin{definition}\label{def:cluster-algebra}
The cluster algebra $A\subset\cF$ with initial seed $t=(\tilde{\mathbf{x}},\tilde{B})$ in $\cF$ is the $\CC[x_{n+1}^\pm,\ldots ,x_{n+m}^\pm]$-subalgebra generated by the cluster variables of $t$ and its iterated mutations. Denote by $\seeds{A}$ the set of seeds of $A$, that is, the set of seeds obtained from $t$ by iterated mutations.
\end{definition}

\begin{remark}
When $Q$ is the empty quiver, or equivalently when $\tilde{B}$ is the empty matrix, the cluster algebra $A$ is declared to simply be $\C$. 
\end{remark}

\begin{definition}\label{def:full-rank}
The cluster algebra $A$ with initial seed $t = (\tilde{\mathbf{x}}, \B)$  has \emph{full rank} if $\B$ has rank $n$, and \emph{really full rank} if moreover the $\Z$-span of the rows of $\tilde{B}$ is $\mathbb{Z}^n$.
\end{definition}

As explained in \cite[Corollary 5.4]{LS}, while the rank assumptions are
imposed on the initial seed, 
they hold for all seeds of $A$ simulatenously.
The following 
condition was introduced in \cite{BFZ}.

\begin{definition}\label{def:acyclic}
A cluster algebra $A$ is acyclic if there is some seed $t\in\Sigma(A)$
such that the principal part of its exchange matrix $B=B(Q)$ corresponds to a quiver $Q$ which has no oriented cycles after removing the frozen vertices.
\end{definition}

\subsection{Cluster varieties}

 If $t=(\tilde{\mathbf{x}},\tilde{B})\in\Sigma(A)$ is any seed, the Laurent phenomenon \cite[Theorem 3.1]{FZ} asserts that, in fact, any other element in any other cluster of $A$ can be expressed as a Laurent polynomial in the elements of $\tilde{\bx}$, so that
 \begin{equation}\label{eq: laurent phenomenon}
 A \subseteq \C[x_1^{\pm 1}, \dots, x_{n+m}^{\pm 1}].
 \end{equation}

 This motivates the following definition:

 \begin{definition}\label{def:upper-cluster-algebra}
     Let $t = (\tilde{\bx}, \tilde{B})$ be a seed and $A$ its cluster algebra. The upper cluster algebra with inital seed $t$ is
     \[
     U := \bigcap_{s = (\tilde{\bz}, \tilde{B'})\in \seeds{A}} \C[z_1^{\pm 1}, \dots, z_{n+m}^{\pm 1}] \subseteq \cF.
     \]
 \end{definition}

 By the Laurent phenomenon \eqref{eq: laurent phenomenon}, we have that in fact $A \subseteq U$. Note that $U$ is a $\C[x_{n+1}^{\pm 1}, \dots, x_{n+m}^{\pm 1}]$-algebra. Note that \eqref{eq: laurent phenomenon} can be restated as
 \begin{equation}\label{eq: laurent phenomenon 2}
A[(x_1\cdots x_{n+m})^{-1}] = \C[x_1^{\pm 1}, \dots, x_{n+m}^{\pm 1}],
 \end{equation}
and we get a similar equation with $A$ replaced by $U$.

\begin{definition}\label{def:cluster-variety}
The affine scheme $\A=\Spec A$ associated to a cluster algebra $A$ is
called a cluster variety. We will call the affine scheme $\U = \Spec U$ an upper cluster variety. 
\end{definition}

The cluster variety $\A$ (resp. $\U$) has algebra of regular functions $A$ (resp. $U$) and function field $\cF$, and the frozen variables have no zeros on $\A$ or $\U$. If $t=(\tilde{\mathbf{x}},\tilde{B})\in\Sigma(A)$ is any seed, by the Laurent phenomenon \eqref{eq: laurent phenomenon 2} the set $\mathbf{x}=(x_1,\ldots x_n)$  determines an
open embedding of the $(n+m)$-dimensional torus $T(t)=\Spec\CC[x_1^\pm,\ldots ,x_{n+m}^\pm]\subseteq \A$, on which the cluster variables of $\mathbf{x}$ have no zeros. Similarly, we have an embedding $T(t) \subseteq \U$.

\begin{definition}\label{def:cluster-chart}
The open tori $T(t)\subseteq\A$ with $t\in\Sigma(A)$ are called cluster tori.
\end{definition}

Cluster tori are a special case of the following construction. If $t=(\tilde{\mathbf{x}},\tilde{B})\in\Sigma(A)$ is a seed, choose a set $S\subseteq \{1,\ldots ,n\}$ of mutable nodes in the quiver $Q$ of the exchange matrix $\tilde{B}=\tilde{B}(Q)$, then form a new quiver $Q_S$ where the nodes in $S$ are frozen and the arrows between frozen vertices are removed. Denote $A_S$ the cluster algebra with initial seed $t_S=(\tilde{\mathbf{x}},\tilde{B}(Q_S))$. This is a subalgebra of the localization $A_S\subseteq A[x_i^{-1}:i\in S]$, and when the equality holds $\A_S=\Spec A_S\subseteq \A$ is an open embedding. In this case, we say that $A_S$ is a \emph{cluster localization} of $A$. The case $\A_S=T(t)$ corresponds to choosing $S=\{1,\ldots ,n\}$,
where all the vertices of $Q$ become frozen and $Q_S$ has no arrows.

\begin{definition}\label{def:locally-acyclic}
A cluster variety $\A$ is locally acyclic if it has a finite covering by open cluster varieties $\A_S\subset\A$ whose rings of regular functions $A_S$ are acyclic cluster algebras.
\end{definition}

The geometry of locally acyclic cluster varieties is particularly well-behaved. For example they are always normal, and even regular when the cluster algebra $A$ is full rank \cite{muller}. Moreover, for locally acyclic cluster algebras we always have that $A = U$, \cite[Theorem 4.1]{muller} so there is no difference between the cluster algebra (resp. variety) and the upper cluster algebra (resp. variety). Most of the cluster varieties discussed in this article will be assumed to be locally acyclic, so most of the time we will not need to distinguish between a cluster variety and an upper cluster variety.

\subsection{Cluster automorphisms} We review cluster automorphisms following \cite{LS}.

\begin{definition}\label{def:cluster-aut}
    Let $A$ be a cluster algebra. A \emph{cluster automorphism} of $A$ is an algebra automorphism $\varphi: A \to A$ such that, for every cluster variable $z$ of $A$ there exists $\zeta(z) \in \C^{\times}$ such that $\varphi(z) = \zeta(z)z$. We denote by $\Aut(A)$ the group of cluster automorphisms of $A$. 
\end{definition}

\begin{remark}
    Note that any cluster automorphism $\varphi \in \Aut(A)$ extends to an automorphism of the upper cluster algebra $U$. Indeed, let us fix a cluster $\tilde{\bx}$. By definition, a cluster automorphism $\varphi \in \Aut(A)$ extends to an automorphism $\varphi: \C[x_1^{\pm 1}, \dots, x_{n+m}^{\pm 1}] \to \C[x_1^{\pm 1}, \dots, x_{n+m}^{\pm 1}]$. Since this automorphism sends every cluster variable to a scalar multiple of itself, it sends any Laurent polynomial on a fixed cluster to a Laurent polynomial in the same cluster, so $\varphi$ preserves $U$. 
\end{remark}

Let $t = (\tilde{\bz}, \B) \in \seeds{A}$, with $z = (z_1, \dots, z_{n+m})$. Since every other cluster variable can be expressed as a rational function in $z_1, \dots, z_{n+m}$, every cluster automorphism is completely determined by the value $(\zeta(z_1), \zeta(z_2), \dots, \zeta(z_{n+m})) \in (\C^{\times})^{n+m}$. It follows that $\Aut(A)$ can be identified with a subgroup of $(\C^{\times})^{n+m}$. The group $\Aut(A)$ can be explicitly described. For this, for a $k\times \ell$-matrix $M$, let us denote by $\mult(M): (\C^{\times})^{\ell} \to (\C^{\times})^{k}$ the group homomorphism given by $\mult(M)(\zeta_1, \dots, \zeta_{\ell}) = (\prod_{j = 1}^{\ell}\zeta_{j}^{M_{ij}})_{i = 1}^{k}$. It is straightforward to check that $\mult(M_{1}M_{2}) = \mult(M_1)\circ\mult(M_2)$.

\begin{lemma}[Lemma 2.3, \cite{GSV_article}]\label{lem:aut-group}
Let $A$ be a cluster algebra and let $t = (\tilde{\bz}, \B) \in \seeds{A}$ be a seed. The group $\Aut(A)$ can be naturally identified with $\ker(\mult(\B^{T})) \subseteq (\C^{\times})^{n+m}$.
\end{lemma}

Note that, although the identification in Lemma \ref{lem:aut-group} depends on the choice of the seed $t$, mutation allows to identify $\ker(\mult(\B(t)^{T}))$ with $\ker(\mult(\B(t')^{T}))$ for different seeds $t, t' \in \seeds{A}$.

\begin{remark}\label{rmk:rk-in-full-rank-case}
    Note that the group $\Aut(A)$ does not need to be a torus, and it does not need to be connected. Nevertheless, if we are in the case when $A$ is a cluster algebra of really full rank, then $\Aut(A)$ is a torus of rank equal to $m$, the number of frozen variables of $A$.
\end{remark}

Since $\Aut(A)$ acts on $A$, it also acts on $\Spec(A)$. Fix a seed $t = (\bz, \B) \in \seeds{A}$, and let $T(t) = \C^{\times}_{z_1} \times \cdots \times \C^{\times}_{z_{n+m}} \subseteq \Spec(A)$ be the corresponding cluster torus. Note that, upon the identification $\Aut(A) \cong \ker(\mult(\B^{T})) \subseteq (\C^{\times})^{n+m}$, $\Aut(A)$ acts on $T(t)$ by left multiplication. Thus, we obtain the following result.

\begin{lemma}\label{lem:cluster-aut-free-on-charts}
Let $t \in \seeds{A}$ be a seed and $T(t) \subseteq \A =  \Spec(A)$ (resp. $T(t) \subseteq \U = \Spec(U)$) its cluster torus. The action of $\Aut(A)$ on $\A$ (resp. $\U$) preserves $T(t)$ and, moreover, $\Aut(A)$ acts freely on $T(t)$. 
\end{lemma}

\begin{example}
Let us take the maximal positroid stratum $\A(k,n)$ inside the Grassmannian $\Gr(k,n)$. The variety $\A(k,n)$ is an affine variety that may be identified with the set of $k \times n$-matrices (modulo row operations) whose cyclically consecutive $k$-minors are nonvanishing. It admits a cluster structure by \cite{Scott}. In this case, $\Aut(A) \cong (\C^{\times})^{n-1}$ acts on $\A(k,n)$ by rescaling the columns. 
\end{example}

\subsection{Cluster quasi-morphisms} Following \cite{LS}, see also \cite{Fraser_quasi}, we define the notion of a \emph{quasi-morphism} of cluster varieties, whose counterpart for cluster algebras is also known as a
\emph{quasi-cluster transformation}.

\begin{definition}\label{def:quasi-iso}
Let $\A_1 = \Spec(A_1)$ and $\A_2 = \Spec(A_2)$ be cluster varieties. Assume that they have the same rank, and moreover that they have the same number of mutable as well as frozen variables. A cluster quasi-morphism from $\A_2$ to $\A_1$ is a triple $(\Psi, \Phi, (R_{t})_{t \in \seeds{A_1}})$ where
\begin{enumerate}
    \item $\Psi: \A_2 \to \A_1$ is a morphism of algebraic varieties, with dual $\Psi^*: A_1 \to A_2$.
    \item $\Phi$ is a map from $\seeds{A_1}$ to $\seeds{A_2}$ commuting with mutation. 
    \item For every seed $t \in \seeds{A_1}$, $R_{t}$ is an $(n+m) \times (n+m)$-integer matrix with block triangular form $R_{t} = \left(\begin{matrix} I_{n} & 0 \\ P_{t} & Q_{t}\end{matrix}\right)$ such that, if $t = (\bz, \B_1)$ and $\Phi(t) = (\bx, \B_2)$ then:
    \begin{itemize}
    \item $\B_2 = R_{t}\B_1$. 
    \item $\Psi^{*}(z_j) = \prod_{i = 1}^{n+m}x_{i}^{R_{ij}}$. 
    \end{itemize}
\end{enumerate}
We say that a cluster quasi-morphism $(\Psi, \Phi, (R_{t}))$ is a cluster quasi-isomorphism if $\Psi$ is an isomorphism of algebraic varieties or, equivalently, if there exists $t \in \seeds{A_{1}}$ such that $R_{t} \in \GL(n+m, \Z)$.
\end{definition}

\begin{remark}
Note that by condition (3), the principal parts of exchange matrices for $t$ and $\Psi(t)$ coincide: $B_2 = B_1$. In particular, cluster quasi-morphisms preserve cluster mutation types. 
\end{remark}

\begin{remark}\label{rmk:mutable-times-monomial}
Again by condition (3), if $z_j \in \bz$ is a mutable variable of $A_1$, then $\Psi^{*}(z_j) = mx_j$, where $m$ is a Laurent monomial in the frozen variables of $A_2$. Similarly, if $z_j \in \bz$ is a frozen variable, then $\Psi^{*}(z_j)$ is a Laurent monomial in frozen variables. 
\end{remark}

\begin{remark}
    Let $(\Psi, \Phi, (R_t)_{t \in \seeds{A_1}})$ be a cluster quasi-morphism from $\A_2$ to $\A_1$. We claim that $\Psi^{*}: A_1 \to A_2$ can be extended to a morphism $\Psi^{*}: U_1 \to U_2$ (and thus we have a morphism of varieties $\Psi: \U_2 \to \U_1$.) Fix a seed $t = (\tilde{\bz}, \tilde{B}_1) \in \seeds{A_1}$, and let $s = \Phi(t) = (\tilde{\bx}, \tilde{B}_2) \in \seeds{A_2}$. By definition, cf. Remark \ref{rmk:mutable-times-monomial}, $\Psi^{*}$ extends to a map
    \[
    \Psi^{*}: \C[z_1^{\pm 1}, \dots, z_{n+m}^{\pm 1}] \to \C[x_1^{\pm 1}, \dots, x_{n+m}^{\pm 1}],
    \]
    and we need to show that $\Psi^{*}(U_1) \subseteq U_2$. Let $u \in U_1$, so that $u$ can be expressed as a Laurent polynomial in every cluster, and consider $\Psi^{*}(u)$. Since the principal parts of $\tilde{B}_1$ and $\tilde{B}_2$ coincide, \cite[Corollary 5.5]{CKLP} implies that for any seed $s' \in \seeds{A_2}$ we can find a seed $t' \in \seeds{A_1}$ with $\Phi(t') = s'$. Since $u$ can be expressed as a Laurent polynomial in the cluster variables of $t'$, this implies that $\Psi^{*}(u)$ is a Laurent polynomial in the cluster variables of $s'$, and we obtain the result. 
\end{remark}

\begin{lemma}\label{lem:equivariance}
    Let $(\Psi, \Phi, (R_{t}))$ be a cluster quasi-morphism from $\A_2$ to $\A_1$. Then, we have an induced morphism $\Aut(A_2) \to \Aut(A_1)$ in such a way that the map $\Psi: \A_2 \to \A_1$ is $\Aut(A_2)$-equivariant. 
\end{lemma}
\begin{proof}
    Let us fix a seed $t = (\bz, \B_1) \in \seeds{A_1}$, with $\Phi(t) = (\bx, \B_2) \in \seeds{A_2}$. Let us recall that $\Aut(A_i) = \ker(\mult(\B_{i}^{T})) \subseteq (\C^{\times})^{n+m}$, $i = 1, 2$. Now from the equation $\B_{2}^{T} = \B_{1}^{T}R_{t}^{T}$, it follows that $\mult(R_{t}^{T}): (\C^{\times})^{n+m} \to (\C^{\times})^{n+m}$ induces a map $R_{t}^{T}: \ker(\mult(\B_2^{T})) \to \ker(\mult(\B_1^{T}))$. Thus, we have an induced map $\Aut(A_2) \to \Aut(A_1)$. 

    For the equivariance statement, it is enough to show that for every $f \in A_1$ and $\zeta = (\zeta_1, \dots, \zeta_{n+m}) \in \Aut(A_2) = \ker(\mult(\B_2^{T}))$ we have $\Psi^{*}(\mult(R_{t}^{T})\zeta\cdot f) = \zeta \cdot \Psi^{*}(f)$. Since $A_1 \subseteq \C[z_{1}^{\pm 1}, \dots, z_{n+m}^{\pm 1}]$, it is enough to check this when $f$ is a cluster variable $z_1, \dots, z_{n+m}$. But this is clear. 
\end{proof}

\begin{remark}
  In the context of Lemma \ref{lem:equivariance} note that, with the same proof, we obtain that $\Psi: \U_2 \to \U_1$ is also $\Aut(A_2)$-equivariant.   
\end{remark}

\section{The deep locus}
\label{sec:deep-locus}
In this section, we introduce the deep locus of a (upper) cluster variety and discuss some of its basic properties. 
Several of the results from this section were previewed in the introduction; for clarity, we restate these ideas in this section before elaborating on them. 

\begin{definition}\label{def:deep-locus}
Let $\A = \Spec(A)$ be a cluster variety. The \emph{deep locus} of $\A$ is the complement to the union of the cluster tori in $\A$, that is:
\[
\per(\A) := \A \setminus \bigcup_{t \in \seeds{A}} T(t).
\]

Similarly, let $\U = \Spec(U)$ be an upper cluster variety. The deep locus of $\U$ is
\[
\per(\U) := \U \setminus \bigcup_{t \in \seeds{A}} T(t). 
\]
\end{definition}

Note that, by definition, $U$ is the algebra of regular functions on $\bigcup_{t \in \seeds{A}}T(t)$, i.e., 
\[
U = \Gamma\left(\bigcup_{t \in \seeds{A}}T(t), \mathcal{O}_{\bigcup_{t \in \seeds{A}}T(t)}\right)
\]
so that $\U$ is, in fact, the affinization of the union of cluster tori.

\begin{remark} \label{rmk:cluster-k2}
In fact, Fock and Goncharov \cite{FG09} define \emph{cluster $K_2$-varieties}, or \emph{cluster $\mathcal{A}$-varieties}, as schemes glued out of algebraic tori corresponding to seeds of $A$ along birational maps corresponding to mutations. Such a variety admits a canonical open immersion into $\mathcal{U}$ \cite{GHK15}, and the image is isomorphic to the union of all cluster tori (in $\mathcal{U}$ or, equivalently, in $\mathcal{A}$). The latter is sometimes called \emph{cluster manifold}.
\end{remark}

\subsection{Basic properties} For the sake of concreteness, from now on we will focus on ordinary cluster algebras. Note, however, that the results in this section have their analogues for upper cluster algebras as well. We give some basic properties of the deep locus $\per(\A)$. To start, note that $\per(\A)$ is a Zariski closed subset of $\A$. Also note that 
\[
\operatorname{Sing}(\A) \subseteq \per(\A).
\]

However, $\per(\A)$ may contain smooth points of $\A$, and it may not be empty when $\operatorname{Sing}(\A)$ is.

\begin{example}
    Consider the cluster variety $\A$ of type $A_{1}$ with a single frozen vertex, so that
    \[
    \A = \C^2 \setminus \{xy + 1 = 0\} \subseteq \C^2
    \]
    is a smooth variety. There are two cluster tori, $T_1 = \A \cap \{x \neq 0\}$ and $T_2 = \A \cap \{y \neq 0\}$. Thus, $\per(\A) = \{(0,0)\}$ is a single point. 
\end{example}

The following result is clear.

\begin{proposition}\label{prop:product}
Let $\B_1 = \left(\begin{matrix} B_1 \\ \hline C_1 \end{matrix}\right)$ and $\B_2 = \left(\begin{matrix} B_2 \\ \hline C_2 \end{matrix}\right)$ be extended exchange matrices, and form the extended exchange matrix \[\B = \left(\begin{matrix} B_1 & 0 \\ 0 & B_2 \\ \hline C_1 & 0 \\ 0 & C_2\end{matrix}\right).\] Let $\A_1, \A_2$ and $\A$ be the corresponding cluster varieties, so that $\A = \A_1 \times \A_2$. Then, $\per(\A) = (\per(\A_1) \times \A_2) \cup (\A_1 \times \per(\A_2))$.
\end{proposition}

In particular, if $\A_1 = \A\left(\begin{matrix} B\\ \hline C \end{matrix}\right)$ and $\A_2 = \A\left(\begin{matrix} B \\ \hline C \\ 0_{s \times n} \end{matrix}\right)$ then $\per(\A_2) = \per(A_1) \times (\C^{\times})^{s}$.

We now examine the behaviour of the deep locus under cluster quasi-morphisms. In particular, cluster quasi-isomorphisms do not change the geometry of the cluster variety $\A$, so they intuitively should preserve the deep locus. This is confirmed in Corollary \ref{cor:quasi-iso} below. 
\begin{lemma}\label{lem:quasi-h}
Let $\A_1 = \Spec(A_1)$ and $\A_2 = \Spec(A_2)$ be cluster varieties, and let $(\Psi, \Phi, (R_{t}))$ be a quasi-morphism from $\A_2$ to $\A_1$. Then, $\Psi(\per(\A_2)) \subseteq  \per(\A_1)$. 
\end{lemma}
\begin{proof}
We will show that $\Psi^{-1}(\A_1\setminus\per(\A_1)) \subseteq \A_2\setminus\per(\A_2)$, which is equivalent to the conclusion of the lemma. Assume $\m \in \A_1\setminus\per(\A_1)$. Thus, there is a seed $t = (\bz, \B_1)$ of $A_1$ such that $z_{j}(\m) \neq 0$ for every $j = 1, \dots, n+m$. Now let $\n$ be such that $\Psi(\n) = \m$. Then, $\Psi^{*}(z_{j})(\n) = z_{j}(\m) \neq 0$ for every $j$. In particular, for $j = 1, \dots, n$, $\Psi^{*}(z_j) = x_{j}\prod_{i>n}x_{i}^{R_{ij}}$, where $\Phi(\bz, \B_1) = (\bx, \B_2)$ and $\prod_{i > n}x_i^{R_{ij}}$ is a Laurent monomial in the frozen variables of $A_1$, cf. Remark \ref{rmk:mutable-times-monomial}. It follows that the cluster variables $x_1, \dots, x_n$ do not vanish at $\n$. Since the cluster variables $x_{n+1}, \dots, x_{n+m}$ are nowhere vanishing, it follows that $\n$ is in the cluster torus corresponding to $\Phi(t)$ and thus is not in $\per(\A_2)$. 
\end{proof}

\begin{corollary}\label{cor:quasi-iso}
Let $\A_1 = \Spec(A_1)$ and $\A_2 = \Spec(A_2)$ be cluster varieties, and let $(\Psi, \Phi, (R_{t}))$ be a cluster quasi-isomorphism from $\A_2$ to $\A_1$. Then, $\Psi$ induces an isomorphism $\Psi: \per(\A_2) \to \per(\A_1)$. 
\end{corollary}

\begin{corollary}\label{cor:coincidence-up-to-tori}
Let $\A_1 = \A(\B_1)$ and $\A_2 = \A(\B_2)$ be locally acyclic cluster varieties, where 
\[\B_1 = \left(\begin{matrix} B \\ \hline C_1 \end{matrix}\right) \qquad \text{and} \qquad \B_2 = \left(\begin{matrix} B \\ \hline C_2\end{matrix}\right)\] are extended exchange matrices of sizes $(n+m_1) \times n$ and $(n+m_2) \times n$, respectively, with the same principal part. Assume moreover that the integer spans of the rows of $\B_1$ and $\B_2$ coincide. Then
\[
\per(\A_1) \times (\C^{\times})^{m_{2}} \cong \per(\A_2) \times (\C^{\times})^{m_{1}}.
\]
\end{corollary}
\begin{proof}
Thanks to Proposition \ref{prop:product} and Lemma \ref{lem:quasi-h}, together with \cite[Proposition 5.8]{LS}, both $\per(\A_1) \times (\C^{\times})^{m_{2}}$ and $\per(\A_2) \times (\C^{\times})^{m_{1}}$ are isomorphic to
\[
\per\left(\A\left(\begin{matrix} B \\ \hline C_1 \\ C_2\end{matrix}\right)\right),
\]
see \cite[Proposition 5.11]{LS} for more details. 
\end{proof}

Thanks to Corollary \ref{cor:coincidence-up-to-tori}, many properties of the deep locus of a locally acyclic cluster variety do not depend on an explicit choice of coefficients, i.e., frozen variables. More precisely, we have the following result. 

\begin{corollary} \label{cor:properties-of-deep-loci-up-to-tori}
Under the same assumptions as Corollary \ref{cor:coincidence-up-to-tori}, we have that
\begin{enumerate}
    \item The deep locus $\per(\A_1)$ is nonempty if and only if so is $\per(\A_2)$. 
    \item The number of irreducible components of $\per(\A_1)$ and $\per(\A_2)$ coincide. In particular, $\per(\A_1)$ is irreducible if and only if so is $\per(\A_2)$.
    \item $\per(\A_1)$ is equidimensional if and only if so is $\per(\A_2)$.
\end{enumerate}
\end{corollary}

\subsection{\texorpdfstring{$\operatorname{Aut}(A)$}{Aut(A)} and the deep locus} 
The automorphism group $\Aut(A)$ acts freely on any cluster torus $T(t) \subseteq \A$. Thus, any point $\m \in \A$ such that $\Stab_{\Aut(A)}(\m) \neq \{1\}$ must necessarily belong to the deep locus $\per(\A)$. This obvious result underscores the importance of the following definition, seen earlier in the introduction:

\begin{definition}\label{def:stabilizer-locus}
Let $A$ be a cluster algebra and $\A$ the corresponding cluster variety. We define the \emph{stabilizer locus}: 
\[
\cS(\A) := \{\m \in \A \mid \Stab_{\Aut(A)}(\m) \neq \{1\}\}. 
\]
\end{definition}

The observation above can then be stated as $\cS(\A) \subseteq \per(\A)$. As described in the introduction, a point in $\per(\A) \setminus \cS(\A)$ will be called \lq\lq mysterious\rq\rq. 

\begin{definition}\label{def:good}
Let $A$ be a cluster algebra and $\A$ the corresponding cluster variety. We say that $A$ (or equivalently, $\A$) has no mysterious points if  
\[
\cS(\A) = \per(\A). 
\]
\end{definition}

\begin{remark}
    Note that a similar definition can be made for the upper cluster algebra $U$. We say that $U$ (or equivalently, $\U$) has no mysterious points if
    \[
    \cS(\U) = \per(\U). 
    \]
\end{remark}

We restate the ``no mysterious points conjecture" from the introduction:
\begin{conjecture}[Shende-DES]\label{conj:main}
    Let $A$ be a locally acyclic cluster algebra. Then, $A$ has no mysterious points.
\end{conjecture}

\begin{remark}
    It follows from Corollary \ref{cor:reduction-to-full-rank} below that to prove Conjecture \ref{conj:main} it is enough to do it under the extra assumption that $A$ has really full rank. 
\end{remark}

The following is the main result of this paper. It may be seen as the first evidence for Conjecture \ref{conj:main}.

\begin{theorem}\label{thm:main:Section3}
The following cluster algebras have no mysterious points.
\begin{enumerate}
\item Finite type simply-laced cluster algebras (with an arbitrary choice of frozen variables). 
\item Cluster algebras associated to the maximal open positroid strata in $\Gr(2,n)$ and $\Gr(3,n)$, as well as those associated to the affine cones over the maximal positroid strata. 
\end{enumerate}
\end{theorem}

Let us remark that the cluster algebra associated to the maximal positroid stratum in $\Gr(3,n)$ is not of finite cluster type for $n \geq 9$. The proof of Theorem \ref{thm:main:Section3} will make heavy use of cluster structures on double Bott-Samelson varieties and, more generally, braid varieties, see \cite{CGGLSS, GLS, GLSS, shen2021cluster}. In particular, we will use the weave calculus from \cite{CGGLSS, CGGS1}. Theorem \ref{thm:main:Section3} is proved in Section~\ref{sec:empty-deep-loci}. For the rest of this section, we address the preservation of the property of having no mysterious points under several constructions. To start we have the following result, which follows at once from Lemma \ref{lem:equivariance} and Corollary \ref{cor:quasi-iso}.

\begin{proposition} \label{prop:validity-of-conjecture-under-qiso}
    Let $A_{1}$ and $A_{2}$ be quasi-isomorphic cluster algebras. Then, $A_1$ has no mysterious points if and only if neither does $A_2$. 
\end{proposition}

\begin{proposition}\label{prop:reduce-to-full-rank-1}
Let $\A_1 = \A(\B_1)$ and $\A_2 = \A(\B_2)$ be locally acyclic cluster varieties, where $\B_1 = \left(\begin{matrix} B \\ \hline C_1 \end{matrix}\right)$ and $\B_2 = \left(\begin{matrix} B \\ \hline C_2\end{matrix}\right)$ are extended exchange matrices of sizes $(n+m_1) \times n$ and $(n+m_2) \times n$, respectively, with the same principal part. Assume moreover that the integer spans of the rows of $\B_1$ and $\B_2$ coincide. Then, $\A_1$ has no mysterious points if and only if neither does $\A_2$.
\end{proposition}
\begin{proof}
    Following the proof of Corollary \ref{cor:coincidence-up-to-tori}, let us consider the cluster variety
    \[
    \A\left(\begin{matrix} B \\ \hline C_1 \\ C_2\end{matrix}\right).
    \]
    It is enough to show that $\A_1$ has no mysterious points if and only neither does $\A$. We have a quasi-isomorphism
    \[
    \A\left(\begin{matrix} B \\ \hline C_1 \\ C_2\end{matrix}\right) \cong \A\left(\begin{matrix} B \\ \hline C_1 \\ 0\end{matrix}\right) = \A_1 \times (\C^{\times})^{m_{2}}.
    \]
    Note that $\Aut\left(A\left(\begin{matrix} B \\ \hline C_1 \\ 0 \end{matrix}\right)\right) = \Aut(A_1) \times (\C^{\times})^{m_{2}}$, and the action of $\Aut(A_1) \times (\C^{\times})^{m_{2}}$ on $\A_1 \times (\C^{\times})^{m_{2}}$ is done component-wise. Thus, $\cS(\A_1 \times (\C^{\times})^{m_{2}}) = \cS(\A_1) \times (\C^{\times})^{m_{2}}$. On the other hand, we know from Proposition \ref{prop:product} that $\per(\A_1 \times (\C^{\times})^{m_{2}}) = \per(\A_1) \times (\C^{\times})^{m_{2}}$. The result follows. 
\end{proof}

\begin{corollary}\label{cor:independence-of-coefficients}
    Let $\A_1 = \A(\tilde{B}_{1})$ and $\A_2 = \A(\tilde{B}_{2})$ be locally acyclic cluster varieties with really full rank, and assume that the principal parts of $\tilde{B}_{1}$ and $\tilde{B}_{2}$ coincide. Then, $\A_1$ has no mysterious points if and only if neither does $\A_2$. 
\end{corollary}

Next, we study the preservation of the property of having no mysterious points upon removing a frozen row to the extended exchange matrix.

\begin{proposition}\label{prop:reduce-to-full-rank-2}
Let $\tilde{B}_1 = \left( \begin{matrix} B \\ \hline  C \end{matrix}\right)$ be an extended exchange matrix for the locally acyclic cluster algebra $A_1$, and let $\tilde{B}_2 = \left(\begin{matrix} B \\ \hline C \\ d\end{matrix}\right)$ be an extended exchange matrix obtained from $\tilde{B}_1$ by adding an extra frozen row $d$. Let $\A_1$, $\A_2$ be the corresponding cluster varieties. If $\A_2$ has no mysterious points then neither does $\A_1$.  
\end{proposition}
\begin{proof}
We show that if $\A_1$ \emph{has} mysterious points then so does $\A_2$. First, let $x_{n+m+1} \in A_2$ be the frozen variable corresponding to the last row of $\tilde{B}_2$. According to \cite[Proposition 3.7]{muller}, we have an isomorphism
\[
p: A_2/(x_{n+m+1}-1)A_2 \to A_1
\]
and thus a closed embedding $\A_1 \hookrightarrow \A_2$, whose image consists of those points where the frozen variable $x_{n+m+1}$ evaluates to $1$. Now we separate into three cases, depending on whether the row $d$ is in the row-span of $\tilde{B}_1$. \\

{\it Case 1. $d$ belongs to the row-span of $\tilde{B}_1$}. In this case, thanks to Proposition \ref{prop:reduce-to-full-rank-1} we have that $\A_1$ has no mysterious points if and only neither does $\A_2$, so we are done. \\

{\it Case 2. $d$ does not belong to the row-span of $\tilde{B}_1$, but a nonzero multiple of it does.} In this case, let $r$ be the minimal positive integer so that $rd$ belongs to the row-span of $\tilde{B}_1$. Note that we obtain
\[
\Aut(A_2) = \Aut(A_1) \times \mu_{r}
\]
\noindent where $\mu_r$ denotes the group of $r$ roots of unity, and the $\mu_r$ component acts only on the frozen variable $x_{n+m+1}$ by multiplication by an $r$-th root of unity, while $\Aut(A_1)$ stabilizes $x_{n+m+1}$. Now assume $\A_1$ has a mysterious point $\mathfrak{m} \in \A_1$, i.e.,  $\Stab_{\Aut(A_1)}(\mathfrak{m})$ is trivial but $\mathfrak{m}$ does not belong to any cluster torus. Consider the image $\mathfrak{n} = p^{*}(\mathfrak{m})$ inside $\A_2$, and note that it does not belong to any cluster torus. So, to prove that $\A_2$ has mysterious points it suffices to show that $\Stab_{\Aut(A_2)}(\mathfrak{n})$ is trivial. Since $x_{n+m+1}(\mathfrak{n}) = 1$, an element of $\Stab_{\Aut(A_2)}(\mathfrak{n})$ must have trivial $\mu_r$-component. But then such an element belongs to $\Stab_{\Aut(1)}(\mathfrak{m}) \times \{1\}$, which is trivial by assumption. \\

{\it Case 3. No nonzero multiple of $d$ belongs to the row-span of $\tilde{B}_1$.} This is similar to Case 2. Here we have $\Aut(A_2) = \Aut(A_1)$, and we proceed as in Case 2.
\end{proof}

\begin{corollary}\label{cor:reduction-to-full-rank}
Let $B$ be an $n \times n$ skew-symmetric matrix. Assume that there exists an $m \times n$ matrix $C$ such that for the extended exchange matrix
\[
\tilde{B} = \left(\begin{matrix} B \\ \hline C \end{matrix}\right)
\]
we have:
\begin{enumerate}
\item $A(\tilde{B})$ is locally acyclic,
\item $\tilde{B}$ has really full rank, and
\item $A(\tilde{B})$ has no mysterious points. 
\end{enumerate}

Then, for \emph{any} $k \geq 0$ and any $k \times n$-matrix $D$, the cluster algebra $A\!\left(\begin{matrix} B \\ \hline  D \end{matrix}\right)$ has no mysterious points. 
\end{corollary}
\begin{proof}
If $A(\tilde{B})$ has no mysterious points, then by Proposition \ref{prop:reduce-to-full-rank-1} the cluster algebra
\[
A\!\left(\begin{matrix} B \\ \hline  D \\ C\end{matrix}\right)
\]
has no mysterious points either. Then, using Proposition \ref{prop:reduce-to-full-rank-2} repeatedly to delete the rows of $C$, we obtain the result. 
\end{proof}

Thanks to Corollary \ref{cor:independence-of-coefficients}, to show Conjecture \ref{conj:main} for an extended exchange matrix $\B$ we are allowed to change the non-principal part of $\B$ as long as the resulting exchange matrix has really full rank. We will leverage this by proving Theorem \ref{thm:main:Section3} for cluster varieties coming from braid varieties, which are known to be locally acyclic and really full rank thanks to \cite[Theorem 7.13 and Lemma 8.1]{CGGLSS}, see also \cite[Section 5]{GLSS}. 

\subsection{When does \texorpdfstring{$\operatorname{Aut}(A)$}{Aut(A)} act freely on \texorpdfstring{$\mathcal{A}$}{A}?} According to Conjecture \ref{conj:main}, if $A$ is a locally acyclic cluster algebra, then $\per(\A) = \emptyset$ if and only if the automorphism group $\Aut(A)$ acts freely on $\A$. In any case, $\Aut(A)$ acting freely on $\A$ is certainly a necessary condition for $\per(\A) = \emptyset$. Thus, it becomes an interesting question to characterize when the action of $\Aut(A)$ on $\A$ is free. We do not know an answer to this question. However, it is straightforward to see that if $\Aut(A)$ acts freely on the set of frozen variables of $\A$, then $\Aut(A)$ acts freely on $\A$.

\begin{proposition}
Let $\tilde{B} = \left(\begin{matrix} B \\ \hline C \end{matrix}\right)$ be an extended exchange matrix, and let $A = A(\tilde{B})$. Assume that $C$ is not the empty matrix. The following conditions are equivalent, and imply that $\Aut(A)$ acts freely on $\A$. 
\begin{enumerate}
\item $\det(B) = \pm 1$, that is, the matrix $B$ has really full rank.
\item $\Aut(A)$ acts freely on the set of frozen variables of $A$.
\end{enumerate}
\end{proposition}
\begin{proof}
Assume first that $\det(B) = \pm 1$. Since $B$ has really full rank, the row-span of $C$ is already contained in the row-span of $B$, so up to a cluster quasi-isomorphism we may replace $\tilde{B}$ by $\left(\begin{matrix}B \\ \hline 0 \end{matrix} \right)$, and for this extended exchange matrix it is clear that $\Aut(A)$ acts freely on the set of frozen variables.

Now assume that $\Aut(A)$ acts freely on the set of frozen variables of $A$. Recall that $\Aut(A)$ is identified with the kernel of $\mult(\tilde{B}^{T}): \C^{n+m} \to \C^{n}$. Consider $(\C^{\times})^{n} \cong V \subseteq (\C^{\times})^{n+m}$, the set of all points whose last $m$ coordinates are equal to $1$, and note that $\mult(\tilde{B}^{T})|_{V}$ can be identified with $\mult(B^{T}): (\C^{\times})^{n} \to (\C^{\times})^{n}$. We must check that $\mult(B^{T})$ is injective, that is, that $\ker(\tilde{B}^{T})\cap V$ is trivial. But any element in $\ker(\tilde{B}^{T})\cap V$ is an element of $\Aut(A)$ acting trivially on the frozen variables, and the result follows. 
\end{proof}

Thus, assuming Conjecture~\ref{conj:main} and thanks to Corollary~\ref{cor:reduction-to-full-rank}, any locally acyclic cluster algebra where the principal part of the exchange matrix has really full rank has empty deep locus. The converse is not true: if we consider the cluster algebra of type $A_1$ without frozens, whose exchange matrix is $\tilde{B} = B = (0)$, then $B$ has kernel of dimension $1$ and $\Aut(A)$ is a $1$-dimensional torus, but the deep(= stabilizer) locus is empty. Note, however, that if we add a frozen row to obtain a really full rank matrix
\[
\tilde{B} = \left( \begin{matrix} 0 \\ \hline 1\end{matrix}\right)
\]
then $\Aut(A)$ does not act freely on the frozen variables and the deep locus is non-empty, cf. Theorem \ref{thm:finite-type-really-full-rank}.

\subsection{The Markov cluster algebra}
\label{subsec:markov-counter-example}

In this section, we analyze the deep locus of the Markov cluster algebra and, in particular, whether this algebra has mysterious points. This cluster algebra is \emph{not} locally acyclic and it is not equal to its own upper cluster algebra, so we will need to treat the deep loci $\per(\A)$ and $\per(\U)$ separately. In this section, we make heavy use of the results from \cite[Sections 6.2 and 7.1]{MM}.

\subsubsection{The coefficient-free case} Let us first analyze the Markov cluster algebra without frozen variables. The extended exchange matrix is
\[
\tilde{B} = \left(\begin{matrix}0 & -2 & 2 \\ 2 & 0 & -2 \\ -2 & 2 & 0\end{matrix}\right).
\]
We see that $\Aut(A) = \{(t_1, t_2, t_3) \in (\C^{\times})^{3} \mid t_1^{2} = t_2^{2} = t_3^{2}\} \cong \C^{\times} \times (\Z/2\Z)^{2}$. Note that mutating $\tilde{B}$ in any direction results in $-\tilde{B}$.

\begin{proposition}
The cluster algebra $A$ and the upper cluster algebra $U$ for the matrix $\tilde{B}$ have no mysterious points.
\end{proposition}

\begin{proof}
Let $\mathfrak{m} \in \per(\A)$ be a deep point, so for every cluster $\{x_1, x_2, x_3\}$ of $A$, we have $(x_1x_2x_3)(\mathfrak{m}) = 0$. Assume that $x_i(\mathfrak{m}) = x_j(\mathfrak{m}) = 0$ for $i \neq j$. Then we claim that, in fact, $\mathfrak{m}$ is annihilated by \emph{all} cluster variables of $A$. In fact, mutating at $i$ we have:
\[
0 = x_i(\mathfrak{m})x_i'(\mathfrak{m}) = x_j^2(\mathfrak{m}) + x_k^2(\mathfrak{m}) = x_k^2(\mathfrak{m})
\]
so that $x_k(\mathfrak{m}) = 0$ and in fact $x_1(\mathfrak{m}) = x_2(\mathfrak{m}) = x_3(\mathfrak{m}) = 0$. Now, in every adjacent cluster at least two cluster variables vanish, so all three cluster variables vanish, and so on. It follows that $\Stab_{\Aut(A)}(\mathfrak{m}) = \Aut(A)$.

Now assume that for every cluster $\{x_1, x_2, x_3\}$ of $A$ there exists a unique $j$ such that $x_j(\mathfrak{m}) = 0$. Note that this $j$ has to be constant among all clusters. So, if for example $j = 1$, we have $(-1, 1, 1) \in \Stab_{\Aut(A)}(\mathfrak{m})$. In any case, we obtain that any deep point must have nontrivial stabilizer under the $\Aut(A)$-action, and we conclude that the Markov cluster algebra $A$ is has no mysterious points.

Let us now examine the Markov upper cluster algebra $U$. Here, we use the description of the upper cluster algebra and the deep locus obtained in \cite[Proposition 6.2.2]{MM}. We have that $U$ is generated by a single cluster $x_1, x_2, x_3$ and 
\[
M := \frac{x_1^{2} + x_2^{2} + x_3^{2}}{x_1x_2x_3}
\]
so that, in fact:
\[
U = \C[x_1, x_2, x_3, M]/(x_1x_2x_3M - x_1^{2} - x_2^{2} - x_3^{2}). 
\]
The action of $\Aut(A)$ on $M$ is given by
\[
(t_1, t_2, t_3)\cdot M = \frac{t_1}{t_2t_3}M = \frac{t_2}{t_1t_3}M = \frac{t_3}{t_1t_2}M. 
\]
Now let $\mathfrak{m} \in \U$ be a deep point. If all cluster variables vanish at $\mathfrak{m}$, then $\mathfrak{m}$ is stabilized by the point $(-1, -1, 1) \in \Aut(A)$, since this stabilizes $M$. According to \cite[Remark 6.2.3]{MM}, other components of the deep locus have the form $x_i = x_j^{2} + x_k^{2} = M = 0$, where $\{i, j, k\} = \{1, 2, 3\}$. But then such a point is stabilized by $t_{i} = -1, t_j = t_k = 1$. In any case, any point in the deep locus must have non-trivial $\Aut(A)$-stabilizer, and we conclude that the upper Markov cluster algebra has no mysterious points.
\end{proof}

\subsubsection{Principal coefficients} Now we analyze the Markov cluster algebra with principal coefficients, whose extended exchange matrix is 
\[
\tilde{B} = \left(\begin{matrix}0 & -2 & 2 \\ 2 & 0 & -2 \\ -2 & 2 & 0 \\ \hline 1 & 0 & 0 \\ 0 & 1 & 0 \\ 0 & 0 & 1 \end{matrix}\right).
\]

\begin{proposition}
The cluster algebra $A_{\mbox{\small{prin}}}$ for the matrix $\tilde{B}$ has no mysterious points, while the upper cluster algebra $U_{\mbox{\small{prin}}}$ has mysterious points.
\end{proposition}

\begin{proof}
We have $\Aut(A_{\mbox{\small{prin}}}) \cong (\C^{\times})^{3}$. Let us fix an extended cluster $\{x_1, x_2, x_3, f_1, f_2, f_3\}$, where $f_i$ is the frozen variable with an arrow to $x_i$. If we normalize the action of $(\C^{\times})^{3}$ so that the weights are 
\[
\wt(x_1) = (1,0,0), \qquad \wt(x_2) = (0,1,0), \qquad \wt(x_3) = (0,0,1),
\]
then 
\[\wt(f_1) = (0, -2, 2), \qquad \wt(f_2) = (2, 0, -2), \qquad \wt(f_3) = (-2, 2, 0).\]
Now a similar analysis to that as in the coefficient-free case shows that the Markov cluster algebra with principal coefficients has no mysterious points.

We now prove that the Markov \emph{upper} cluster algebra with principal coefficients has mysterious points by exhibiting some of them. By \cite[Proposition 7.1.1]{MM}, the upper cluster algebra in this case is generated by $x_1, x_2, x_3, f_1, f_2, f_3$ and $L_1, L_2, L_3, y_1, y_2, y_3$, where we have
\[
\wt(y_1) = (-1, 0, 0), \wt(y_2) = (0, -1, 0), \wt(y_3) = (0, 0, -1). 
\]
Thus, if we can find a point $\mathfrak{m} \in \U_{\mbox{\small{prin}}}$ with $x_1(\mathfrak{m}) = x_2(\mathfrak{m}) = x_3(\mathfrak{m}) = 0$ but $(y_1y_2y_3)(\mathfrak{m}) \neq 0$, then this point will be in $\per(\U_{\mbox{\small{prin}}})$ but will have trivial $\Aut(A_{\mbox{\small{prin}}})$ stabilizer. A point with
\[
x_1(\mathfrak{m}) = x_2(\mathfrak{m}) = x_3(\mathfrak{m}) = L_1(\mathfrak{m}) = L_2(\mathfrak{m}) = L_3(\mathfrak{m}) = 0, f_1(\mathfrak{m}) = f_2(\mathfrak{m}) = f_3(\mathfrak{m}) = 1
\]
and $(y_1y_2y_3)(\mathfrak{m}) \neq 0$, $(y_1^2 + y_2^2 + y_3^2)(\mathfrak{m}) = 0$ exists thanks to the relations given in \cite[Proposition 7.1.1]{MM}. Thus, the Markov upper cluster algebra with principal coefficients has mysterious points.
\end{proof}

\begin{remark}
Note, in particular, that the property of having no mysterious points is not preserved for upper cluster algebras upon adding frozen vertices, see also \cite[Corollary 7.1.2]{MM}. We do not know if the corresponding result is true for cluster algebras. 
\end{remark}

\begin{remark}
The elements $y_1, y_2, y_3 \in U_{\mbox{\small{prin}}}$ satisfy many properties similar to those of the cluster variables $x_1, x_2, x_3$. In particular, one can mutate (using the same mutation rule given by the matrix $\tilde{B}$) indefinitely to get elements of the upper cluster algebra, which satisfy the Laurent phenomenon. In the setting of Fomin-Shapiro-Thurston \cite{FST, FT}, these elements correspond to \lq\lq notched\rq\rq \, arcs. If one modifies the definition of the deep locus to include the complement to the union of the cluster tori \emph{and} the tori determined by the \lq\lq clusters\rq\rq \, obtained by mutating the $y$-variables, then it is possible to see, similarly to the analysis done above, that the deep locus indeed coincides with the stabilizer locus. This larger union of toric charts has been considered in \cite{Z20} and denoted by $\widetilde{\A}_{\mbox{\small{prin}}, \mathbb{T}^2_1}$. A similar construction is considered in loc.cit. in the coefficient-free case under the name $\widetilde{\A}_{\mathbb{T}^2_1}$ and it is verified that the union of cluster tori has codimension 1 in $\widetilde{\A}_{\mathbb{T}^2_1}$. The algebra of regular functions on $\widetilde{\A}_{\mathbb{T}^2_1}$ is therefore strictly contained in the upper cluster algebra $U$ of regular functions on the the union of cluster tori. As verified in \cite[Section 5.3]{Z20}, it agrees with the image of $U_{\mbox{\small{prin}}}$ under the specialization of all frozen variables to $1$. The connection between this algebra sitting between $A$ and $U$ which corresponds to $\widetilde{\A}_{\mathbb{T}^2_1}$ and the tagged skein algebra of the once punctured torus is explained in \cite[Remark 3.19]{MQ23}. The locus of points with the non-trivial stabilizer under the action of $\Aut(A)$ in the affinization of $\widetilde{\A}_{\mathbb{T}^2_1}$ does coincide with the complement to $\widetilde{\A}_{\mathbb{T}^2_1}$. 
We thank Greg Muller for comments on this remark.
\end{remark}

\section{Braid varieties}
\label{sec:braid varieties}

In this section, we introduce braid varieties that will provide explicit geometric models for the cluster varieties we will work with. We will work only in the Coxeter type $A$, that is, with usual braids on $n$ strands.
We will review the cluster structure on braid varieties via weaves from \cite{CGGLSS}, and the torus action that appeared in \cite{CGGS1}. As we will see, in some cases this torus action is precisely the action by cluster automorphisms.

\subsection{Definition}\label{subsec:braid-varieties} Let us fix $n > 0$. We will denote by $\Br_n^+$ the positive braid monoid on $n$ strands. This is the monoid with generators $\sigma_1, \dots, \sigma_{n-1}$ and relations
\begin{equation}\label{eq:braid relations}
\sigma_{i}\sigma_{i+1}\sigma_{i} = \sigma_{i+1}\sigma_i\sigma_{i+1} \; (i \leq n-2), \; \qquad \; \sigma_i\sigma_j = \sigma_j\sigma_i \; (|i - j| > 1).
\end{equation}
Note that there is an obvious projection $\pi: \Br^{+}_n \to S_n$, sending the generator $\sigma_i$ to the Coxeter generator $s_i = (i, i+1)$ of the symmetric group. We will refer to $\sigma_1, \dots, \sigma_{n-1}$ as the \emph{braid generators} of $\Br_n^{+}$. 

\begin{definition}\label{def:permutation-length}
Let $\beta \in \Br_n^+$. If $\beta = \sigma_{i_1}\cdots \sigma_{i_{\ell}}$ is an expression of $\beta$ as a product of the generators $\sigma_i$, then we call $\ell$ the \emph{length of $\beta$} and write $\ell =: 
|\beta|$. This is well-defined by Matsumoto's Theorem \cite{Matsumoto}.

Let $w \in S_n$. We define the \emph{length} $\ell(w)$ of $w$ to be the minimal length of an element $\beta \in \Br_n^+$ with $\pi(\beta) = w$. 
\end{definition}

If $\beta \in \Br_n^{+}$ is such that $|\beta|$ is minimal amongst all elements in $\Br_n^{+}$ with $\pi(\beta) = w$, then we say that $\beta$ is a \emph{minimal lift} of $w$ to $\Br_n^{+}$. Every element of $S_n$ has a lift: simply express $w$ as a product of the Coxeter generators of $S_n$, and replace each Coxeter generator $s$ by the corresponding braid generator $\sigma$. 

\begin{definition}\label{def:braid-block}
For $i = 1, \dots, n$ we define the \emph{braid matrix} $B_{i}(z) \in \GL_n(\C[z])$ to be the matrix given by
\begin{equation}\label{eq:braid matrix}
B_i(z)_{jk} = \begin{cases} z & j = k = i \\ 1 & j = k \neq i, i+1 \\ -1 & j = i, k = i+1 \\ 1 & j = i+1, k = i \\ 0 & \text{else} \end{cases}, \qquad B_{i}(z) = \left(\begin{matrix} 
1 & \ldots &     &     &   \ldots & 0  \\
\vdots & \ddots &  &  &   & \vdots \\ 
0 & \ldots & z & -1 & \ldots & 0 \\
0 & \ldots & 1 & 0  & \ldots & 0 \\
\vdots &   &  &   &  \ddots & 0 \\
0 & \ldots & & & \ldots & 1 
\end{matrix}\right),
\end{equation}
where the only nontrivial block of $B_i(z)$ is in the $i$ and $i+1$-st row and column.
\end{definition}

Up to a change of variables, the matrices $B_{i}(z)$ satisfy the braid relations \eqref{eq:braid relations}. More precisely, we have:
\begin{equation}\label{eq:matrix braid relations}
\begin{array}{c} B_i(z_1)B_{i+1}(z_2)B_i(z_3) = B_{i+1}(z_3)B_{i}(z_1z_3 - z_2)B_{i+1}(z_1), \\ B_i(z)B_j(w) = B_j(w)B_i(z) \; (|i-j| > 1),
\end{array}
\end{equation}
that can be easily verified. 

\begin{definition}\label{def:braid-matrix}
    For a braid word $\beta = \sigma_{i_1}\cdots \sigma_{i_{\ell}} \in \Br^{+}_{n}$ we define the \emph{braid matrix}
\[
B_{\beta}(z_1, \dots, z_{\ell}) := B_{i_1}(z_1)\cdots B_{i_{\ell}}(z_{\ell}) \in \GL_n(\C[z_1, \dots, z_{\ell}]). 
\]
\end{definition}

The last ingredient we need to define the braid variety is that of the Demazure product $\delta: \Br^{+}_{n} \to S_n$. It can be defined inductively as follows: $\delta(1) = 1$ and
\[
\delta(\beta\sigma_i) = \begin{cases} \delta(\beta)s_i & \text{if} \; \ell(\delta(\beta)s_i) > \ell(\delta(\beta)) \\ \delta(\beta) & \text{else}.\end{cases}
\]
equivalently, $\delta(\beta)$ is the permutation associated to the longest reduced subexpression of $\beta$.

\begin{definition}\label{def:braid-variety}
    Let $\beta = \sigma_{i_1}\cdots \sigma_{i_{\ell}} \in \Br^{+}_{n}$ be a positive braid with Demazure product $\delta$. The \emph{braid variety} $X(\beta)$ is defined to be
    \[
    X(\beta) = \{(z_1, \dots, z_{\ell}) \in \C^\ell \mid \delta(\beta)^{-1}B_{\beta}(z_1, \dots, z_\ell) \; \text{is upper-triangular}\},
    \]
    where $\delta(\beta)^{-1} \in S_n$ is identified with its corresponding permutation matrix in $\GL_n$. 
\end{definition}
Note that, while the braid matrix $B_{\beta}(z_1, \dots, z_{\ell})$ depends on a particular expression chosen for $\beta$, it follows from \eqref{eq:matrix braid relations} that two expressions for $\beta$ yield isomorphic braid varieties. 

\begin{theorem}(\cite[Theorem 20]{Escobar}, \cite[Theorem 4.8]{CGGS_positroid})
For any braid $\beta$, the braid variety $X(\beta)$ is a smooth, affine algebraic variety of dimension $|\beta| - \ell(\delta(\beta))$. 
\end{theorem}

Most of the time, we will look at braids $\beta$ whose Demazure product is $\delta(\beta) = w_0$, the longest element in the symmetric group $S_n$, whose length is $\binom{n}{2}$.

\subsection{Braid varieties via flags} Let us give a coordinate-free definition of the braid variety $X(\beta)$ that makes many constructions (such as the torus action introduced in \cite{CGGS1}) more natural. To do this, we work in the flag variety $\Fl(n)$, the moduli space of flags of subspaces
\[
F^{\bullet} = (F^{0} = \{0\} \subseteq F^{1} \subseteq \cdots \subseteq F^{n} = \C^n)
\]
with $\dim F^{i} = i$. Choosing a  basis $\{e_1, \dots, e_n\}$ of $\C^n$, we have an identification
\[
\Fl(n) \cong \GL(n)/\borel(n)
\]
where $\borel(n) \subseteq \GL(n)$ is the subgroup of upper triangular matrices. Under this identification, a matrix $A \in \GL(n)$ is associated to a flag $F_{A}^{\bullet}$, where $F_{A}^{i}$ is the span of the first $i$ columns of the matrix $A$. The \emph{standard flag} is the flag
\[
F_{\std}^{\bullet} = (\{0\} \subseteq \langle e_1\rangle \subseteq \langle e_1, e_2\rangle \subseteq \cdots \subseteq \langle e_1, \dots, e_{n-1}\rangle \subseteq \C^n).
\]
Moreover, for each permutation $w \in S_{n}$ we have the \emph{coordinate flag}
\[
F_{w}^{\bullet} = ( \{0\} \subseteq \langle e_{w(1)}\rangle \subseteq \langle e_{w(1)}, e_{w(2)}\rangle \subseteq \cdots \subseteq \langle e_{w(1)}, \dots, e_{w(n-1)}\rangle \subseteq \C^n).
\]

We say that two flags $F^{\bullet}$ and $G^{\bullet}$ are in position $s_{i}$ ($i =1, \dots, n-1$) if $F^{j} = G^{j}$ for $j \neq i$, but $F^{i} \neq G^{i}$. We denote this relationship between flags by $F^{\bullet} \buildrel s_i \over \rightarrow G^{\bullet}$. It is easy to see that, fixing a matrix $A \in \GL(n)$:
\[
\{G^{\bullet} \mid F_{A}^\bullet \xrightarrow{s_i} G^{\bullet}\} = \{F^{\bullet}_{AB_i(z)} \mid z \in \C\}.
\]

It follows that, if $\beta = \sigma_{i_{1}}\cdots \sigma_{i_{\ell}}\in \Br_{n}^{+}$ is a positive braid we obtain:
\begin{equation}\label{eqn:braid flags}
X(\beta) \cong \{(F_{0}^{\bullet}, F_{1}^{\bullet}, \dots, F_{\ell}^{\bullet}) \in \Fl(n)^{\ell + 1} \mid F_{0}^{\bullet} = F_{\std}^{\bullet}, F_{\ell}^{\bullet} = F_{\delta(\beta)}^{\bullet}, F_{j-1}^{\bullet} \buildrel s_{i_{j}} \over \rightarrow F_{j}^{\bullet}, j = 1, \dots, \ell\}.
\end{equation}

\subsection{Cluster structure} For any braid $\beta$, the braid variety $X(\beta)$ admits a cluster structure that we describe in this section, following \cite{CGGLSS}, see also \cite{GLS, GLSS}.

The main ingredient when constructing a cluster structure on $X(\beta)$ is that of an \emph{algebraic weave}, introduced in \cite{CGGS1, CZ}. We will not describe these in general, but we will restrict to a class of weaves called \emph{Demazure weaves}. By abuse of notation, we denote by $\delta(\beta)$ the minimal braid lift of a fixed reduced expression of $\delta(\beta)$. Thus, we think of $\delta(\beta)$ as a positive braid word.

\begin{definition}\label{def:braid-graph}
We define the (directed) graph of positive braid words, $\W_n$ as follows. The vertices correspond to all words in the generators $\sigma_{1}, \dots, \sigma_{n-1}$ of $\Br^{+}_{n}$. The arrows are defined as follows:
\begin{enumerate}
\item There is an arrow $\beta'\sigma_i\sigma_i\beta'' \to \beta'\sigma_i\beta''$. Crucially, there is no arrow going in the opposite direction.
\item If $|i-j| > 1$, there is an arrow $\beta'\sigma_i\sigma_j\beta'' \to \beta'\sigma_j\sigma_i\beta''$, as well as an arrow going in the opposite direction.
\item There is an arrow $\beta' \sigma_{i}\sigma_{i+1}\sigma_{i}\beta'' \to \beta'\sigma_{i+1}\sigma_{i}\sigma_{i+1}\beta'$ as well as an arrow going in the opposite direction.
\end{enumerate}
\end{definition}

Note that the Demazure product is constant along each connected component of $\W_n$. In fact, it is not hard to show that the Demazure product gives a bijection between connected components of $\W_n$ and the symmetric group $S_n$.

\begin{definition}\label{def:demazure-weave}
Let $\beta_1, \beta_2 \in \W_n$.
\begin{itemize}
    \item A \emph{Demazure weave} from $\beta_1$ to $\beta_2$ is a directed path from $\beta_1$ to $\beta_2$ in $\W_n$.
    \item A \emph{complete Demazure weave} on $\beta_1$ is a Demazure weave from $\beta_1$ to $\delta(\beta_1)$. 
\end{itemize} 
\end{definition}

We can think of a Demazure weave as a diagram, as follows.
A Demazure weave $\w: \beta_1 \to \beta_2$ is a graph on a rectangle $R$, whose edges are colored by $1, \dots, n-1$ and whose vertices are of the following type:
\begin{enumerate}
    \item[(*)] Univalent vertices, which are located only on the top and bottom sides of the rectangle. Moreover, we require that the colors of the edges adjacent to the vertices on the top side spell $\beta$ when read from left-to-right, and that the colors of the edges adjacent to vertices on the bottom side spell $\beta_2$.
    \item Trivalent vertices, located in the interior of $R$. We require that all edges adjacent to a trivalent vertex have the same color, say $i$. In this case, we will say that it is an $i$-trivalent vertex. We also require that the three edges adjacent to a vertex are located to the northwest of the vertex (the \lq\lq left arm"), to the northeast of the vertex (the \lq\lq right arm") and to the south of the vertex (the \lq\lq leg"). 
    \item Tetravalent vertices, located in the interior of $R$. We require that two of the edges adjacent to a tetravalent vertex have the same color $i$, and the other two have the same color $j$, with $|i-j| > 1$. Moreover, the colors of the edges should be alternating.
    \item Hexavalent vertices, located in the interior of $R$. We require that three of the edges adjacent to a hexavalent vertex have the same color $i$, and the other three have the same color $i+1$. Moreover, the colors of the edges should be alternating. 
\end{enumerate}

The univalent vertices (*) give us the domain and target of the weave, while the vertices of type (1), (2), (3) correspond to arrows in the graph $\W_n$. We also require that the edges of $\w$ have isolated horizontal tangencies. See Figure \ref{fig:11121212weave}.

\begin{figure}[h!]
    \centering
    \includegraphics{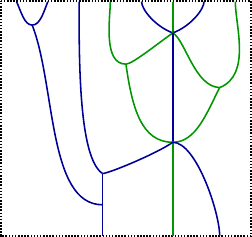}
    \caption{A complete weave $\mathfrak{w}: \beta \to \delta(\beta)$. This weave depicts the path $\sigma_1\sigma_1\sigma_1\sigma_2\sigma_1\sigma_2\sigma_1\sigma_2 \to \sigma_1\sigma_1\sigma_2\sigma_1\sigma_2\sigma_1\sigma_2 \to \sigma_1\sigma_1\sigma_2\sigma_2\sigma_1\sigma_2\sigma_2 \to \sigma_1\sigma_1\sigma_2\sigma_1\sigma_2\sigma_2 \to \sigma_1\sigma_1\sigma_2\sigma_1\sigma_2 \to  \sigma_1\sigma_1\sigma_1\sigma_2\sigma_1 \to  \sigma_1\sigma_1\sigma_2\sigma_1 \to  \sigma_1\sigma_2\sigma_1$. }
    \label{fig:11121212weave}
\end{figure}

\begin{remark}
Note that hexavalent vertices only appear when $n > 2$, and tetravalent vertices only appear when $n > 3$. In particular, since we restrict to $n \leq 3$, we will not see tetravalent vertices in this paper. 
\end{remark}

A weave also encodes morphisms between distinct braid varieties. Tetra and hexavalent vertices encode isomorphisms given by the equations \eqref{eq:matrix braid relations}. On the other hand, trivalent vertices encode a birational map
\begin{equation}\label{eqn:opening crossings}
X(\beta_1\sigma_{i}\sigma_{i}\beta_2) \dashrightarrow X(\beta_1\sigma_{i}\beta_2).
\end{equation}
This map is very natural in terms of the flag realization of the braid variety \eqref{eqn:braid flags}. Indeed, let us say that around the subword $\sigma_{i}\sigma_{i}$ of $\beta_1\sigma_{i}\sigma_{i}\beta_{2}$ we have the flags
\[
F^{\bullet} \buildrel s_{i} \over \rightarrow G^{\bullet} \buildrel s_{i} \over \rightarrow H^{\bullet}.
\]
Note that, if $F^{\bullet} \neq H^{\bullet}$ then $F^{\bullet} \buildrel s_{i} \over \rightarrow H^{\bullet}$. The birational map \eqref{eqn:opening crossings} is then defined on the open set $F^{\bullet} \neq H^{\bullet}$, and it simply forgets the flag $G^{\bullet}$. 

We will need an expression of the map \eqref{eqn:opening crossings} in terms of coordinates, which stems from the following two results. For their proof, see e.g. \cite[Section 2]{CGGS1}.

\begin{lemma}
Let $z, w \in \C$. Then, $B_{i}(z)B_{i}(w)$ is upper triangular if and only if $w = 0$. If $w \neq 0$ then we have a decomposition
\[
B_i(z)B_i(w) = B_i(z - w^{-1})U_i(w),
\]
where $U_i(w)$ is the upper triangular matrix that is the identity except for the $i$-th and $i+1$-st row and columns, where it is
\[\left(\begin{matrix} w & -1 \\ 0 & w^{-1}\end{matrix}\right).
\]
\end{lemma}

\begin{lemma}\label{lem:sliding}
Let $U$ be an invertible upper triangular matrix, and $z \in \C$. Then, there exist a unique upper triangular matrix $U'$ and $z' \in \C$ such that $UB_i(z) = B_i(z')U'$.
\end{lemma}

Thus, in coordinates the morphism \eqref{eqn:opening crossings} is given as follows,
\begin{equation}\label{eqn:opening crossings variables}
(\dots, z, w, z_{1}, z_{2}, \dots, z_{\ell(\beta_2)}) = (\dots, z - w^{-1}, z'_{1}, \dots, z_{\ell(\beta_2)}^{'})
\end{equation}
where $z, w$ are the variables corresponding to the consecutive $\sigma_{i}$, and $z'_{1}, \dots, z'_{\ell(\beta_{2})}$ are polynomial functions of $z_{1}, \dots, z_{\ell(\beta_2)}$ and $w^{\pm 1}$ which are obtained by ``sliding'' the upper triangular matrix $\left(\begin{matrix} w & -1 \\ 0 & w^{-1}\end{matrix}\right)$ to the right using Lemma \ref{lem:sliding}.

\begin{example}\label{ex:birational-map}
Consider the words $\beta_1 = \sigma_1\sigma_1\sigma_2\sigma_1\sigma_2$ and $\beta_2 = \sigma_1\sigma_2\sigma_1\sigma_2$. If $z_2 \neq 0$ then we have the following matrix identity:
\[
\begin{array}{rl}
B_{\beta_1}(z_1, z_2, z_3, z_4, z_5) =  & B_1(z_1 - z_2^{-1})U_1(z_2)B_2(z_3)B_1(z_4)B_2(z_5) \\
= & B_1(z_1-z_2^{-1})B_2(z_3z_2^{-1})\left(\begin{matrix} z_2 & -z_3 & 1 \\ 0 & 1 & 0 \\ 0 & 0 & z_2^{-1}\end{matrix}\right)B_1(z_4)B_2(z_5) \\
= & B_1(z_1-z_2^{-1})B_2(z_3z_2^{-1})B_1(z_2z_4 - z_3)\left(\begin{matrix}1 & 0 & 0 \\ 0 & z_2 & -1 \\ 0 & 0 & z_2^{-1} \end{matrix}\right) B_2(z_5) \\
= & B_1(z_1-z_2^{-1})B_2(z_3z_2^{-1})B_1(z_2z_4 - z_3)B_2(z_5z_2^2 - z_2)\left(\begin{matrix} 1 & 0 & 0 \\ 0 & z_2^{-1} & 0 \\ 0 & 0 & z_2  \end{matrix}\right) \\
= & B_{\beta_2}(z_1-z_2^{-1}, z_3z_2^{-1}, z_2z_4 - z_3, z_5z_2^2 - z_2)U
\end{array}
\]
where the matrix $U$ is upper triangular. Thus, the map $X(\beta_1) \dashrightarrow X(\beta_2)$ is given by
\[
(z_1, z_2, z_3, z_4, z_5) \mapsto (z_1 - z_2^{-1}, z_3z_2^{-1}, z_2z_4 - z_3, z_5z_2^{2} - z_2).
\]
\end{example}

We remark that the birational map \eqref{eqn:opening crossings} identifies the open locus $\{w \neq 0\}\subseteq X(\beta_1\sigma_i\sigma_i\beta_2)$ with the product $X(\beta_1\sigma_{i}\beta_2) \times \C^{\times}$, see e.g. \cite[Section 2.3]{CGGS1}. In Section \ref{subsec:cluster-localizations} below, we will strengthen this result by taking into account the cluster structure on these varieties, see Lemma \ref{lem:cluster-localization}. 

It is convenient to decorate the edges of a weave with variables: on the top we have the coordinates $z_1, \dots, z_{\ell}$ of the braid variety $X(\beta)$, and we propagate the variables down according to the rules \eqref{eq:matrix braid relations} and \eqref{eqn:opening crossings variables} as in Example \ref{ex:birational-map}. We will be interested in the following two results, that illustrate this procedure.

\begin{lemma}\label{lem:easy3valent}
For $w, z_1 \in \C$, $w \neq 0$ we have:
\[
\left(\begin{matrix} w & -1 \\ 0 & w^{-1} \end{matrix}\right)\left(\begin{matrix} z_1 & -1 \\ 1 & 0 \end{matrix}\right) = \left( \begin{matrix} w^{2}z_1 - w & -1 \\ 1 & 0\end{matrix}\right)\left(\begin{matrix} w^{-1} & 0 \\ 0 & w\end{matrix}\right).
\]
\end{lemma}

In other words, $U_{i}(w)B_{i}(z) = B_{i}(w^{2}z - w)D$, where the matrix $D$ is diagonal. See Figure \ref{fig:easy3valent} for a pictorial interpretation of Lemma \ref{lem:easy3valent}. 
 
\begin{figure}[h!]
    \centering
    \includegraphics[scale=1.5]{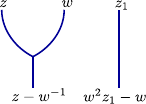}
    \caption{Pictorial interpretation of Lemma \ref{lem:easy3valent}. Note that a variable to the right of $z_1$, if any, will simply be multiplied by a power of $w$.}
    \label{fig:easy3valent}
\end{figure}

The other case of interest for us is the following.

\begin{lemma}\label{lem:hard3valent}
Let $w \in \C^{\times}$, $z_1, z_2, z_3 \in \C$. Then we have:
\begin{align}
\label{eq:hard3valent1} \left(\begin{matrix} w & -1 & 0 \\ 0 & w^{-1} & 0 \\ 0 & 0 & 1 \end{matrix}\right)\left(\begin{matrix} 1 & 0 & 0 \\ 0 & z_1 & -1 \\ 0 & 1 & 0 \end{matrix}\right) & = \left(\begin{matrix} 1 & 0 & 0 \\ 0 & z_1w^{-1} & -1 \\ 0 & 1 & 0 \end{matrix}\right)\left(\begin{matrix} w & -z_1 & 1 \\ 0 & 1 & 0 \\ 0 & 0 & w^{-1} \end{matrix}\right); \\
\label{eq:hard3valent2} \left(\begin{matrix} w & -z_1 & 1 \\ 0 & 1 & 0 \\ 0 & 0 & w^{-1} \end{matrix}\right)\left(\begin{matrix} z_2 & -1 & 0 \\ 1 & 0 & 0 \\ 0 & 0 & 1 \end{matrix}\right) & = \left(\begin{matrix} wz_2 - z_1 & -1 & 0 \\ 1 & 0 & 0 \\ 0 & 0 & 1 \end{matrix}\right)\left(\begin{matrix} 1 & 0 & 0 \\ 0 & w & -1 \\ 0 & 0 & w^{-1} \end{matrix}\right); \\
\label{eq:hard3valent3} \left(\begin{matrix} 1 & 0 & 0 \\ 0 & w & -1 \\ 0 & 0 & w^{-1} \end{matrix}\right)\left(\begin{matrix} 1 & 0 & 0 \\ 0 & z_3 & -1 \\ 0 & 1 & 0 \end{matrix}\right) & = \left(\begin{matrix} 1 & 0 & 0 \\ 0 & z_3w^2 - w & -1 \\ 0 & 1 & 0 \end{matrix}\right)\left(\begin{matrix} 1 & 0 & 0 \\ 0 & w^{-1} & 0 \\ 0 & 0 & w \end{matrix}\right).
\end{align}
\end{lemma}

Pictorially, we interpret Lemma \ref{lem:hard3valent} as in Figure \ref{fig:hard3valent}. 

\begin{figure}[h!]
    \centering
    \includegraphics[scale=1.5]{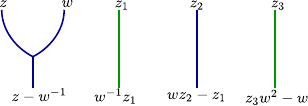}
    \caption{Pictorial interpretation of Lemma \ref{lem:hard3valent}. Note that a variable to the right of $z_3$, if any, will simply be multiplied by a power of $w$. Also note that we have a similar diagram with the colors blue and green interchanged.}
    \label{fig:hard3valent}
\end{figure}

\begin{definition}\label{def:s-variable}
Let $v$ be a trivalent vertex on a complete Demazure weave $\mathfrak{w}$. The $s$-variable associated to $v$, $s_{v} \in \C(X(\beta))$, is the rational function on $X(\beta)$ decorating the right arm of the trivalent vertex $v$.
\end{definition}

If $\mathfrak{w}$ is a complete Demazure weave on $\beta$, then the non-vanishing of the $s$-variables defines an open torus $T_{\mathfrak{w}}$ in $X(\beta)$. This is, in fact, a cluster torus, see e.g. \cite[Remark 5.22]{CGGLSS}. Moreover, we have the following result. 

\begin{theorem}\label{thm:cluster-structure}
For any braid $\beta$, the braid variety $X(\beta)$ admits a cluster structure. Moreover,
\begin{itemize}
\item[(a)] For any complete Demazure weave $\mathfrak{w}$ on $\beta$, $T_{\mathfrak{w}}$ is a cluster torus.
\item[(b)] For any complete Demazure weave $\mathfrak{w}$, the corresponding cluster variables are monomials in the $s$-variables.
\item[(c)] For each $i$, the coordinate $z_i$ on $X(\beta)$ introduced in Definition \ref{def:braid-variety} is a cluster monomial. 
\item[(d)] The cluster structure on $X(\beta)$ is locally acylic and really full rank. 
\item[(e)] If $\beta_1$ and $\beta_2$ are related by a sequence of braid moves, then $X(\beta_1)$ and $X(\beta_2)$ are isomorphic as cluster varieties. 
\end{itemize}
\end{theorem}
\begin{proof}
The claim about the cluster structure is \cite[Theorem 1.1]{CGGLSS}. Part (a) is  \cite[Remark 5.22]{CGGLSS}. Thanks to the same remark, both the cluster variables and the $s$-variables form a system of coordinates for the same open torus $T_{\mathfrak{w}}$, so (b) follows. Part (c) is Corollary 5.29 in \cite{CGGLSS} and the locally acyclic part of (d) is Theorem 7.13 in \emph{loc. cit.} 

Let us prove that the cluster structure on $X(\beta)$ is really full rank. By \cite[Lemmas 8.2 and 8.4]{CGGLSS}, the exchange matrix associated to a complete Demazure weave $\mathfrak{w}$ on $\beta$ can be extended, using integer entries, to a unimodular $(n+m)\times(n+m)$-matrix. It is easy to see that this implies that the cluster structure has really full rank (see also \cite[Theorem 5.9]{GLSS}.) Finally, (e) is a consequence of \cite[Proposition 4.28]{CGGLSS}.
\end{proof}

We have not explained how to obtain an ice quiver (eq. an extended exchange matrix) giving the cluster structure on $X(\beta)$, as this is quite technical and will not be needed in full generality. See \cite[Section 4]{CGGLSS} for a construction of the quiver. In Section \ref{subsec:double-bs} below, we will explain how to construct a quiver in the so-called double Bott-Samelson case, as this is all that will be needed. 

\begin{lemma} [Lemma 3.10, Theorem 5.17, \cite{CGGLSS}]\label{lem:rotation-is-qiso}
Let $\beta = \sigma_{i_1}\cdots \sigma_{i_l} \in \mathrm{Br}^{+}_{n}$, and $\beta' = \sigma_{i_2}\cdots \sigma_{i_l}\sigma_{n - i_1}$. Then, the braid varieties $X(\beta)$ and $X(\beta')$ are cluster quasi-isomorphic in the sense of Definition \ref{def:quasi-iso}. 
\end{lemma}

\begin{definition}\label{def:cyclic-rotation}
We will refer to the cluster quasi-isomorphism $X(\beta) \to X(\beta')$ in Lemma \ref{lem:rotation-is-qiso} as \emph{cyclic rotation}, and we will say that the braid words $\beta$ and $\beta'$ are cyclically equivalent, denoted by $\beta \simeq \beta'$. 
\end{definition}

\subsection{Cluster localizations}\label{subsec:cluster-localizations} Let $A$ be a cluster algebra, and $t \in \seeds{A}$ a seed. Let $x_1, \dots, x_k$ be a collection of mutable variables from $t$. Recall that the algebra $A[x_1^{-1}, \dots, x_k^{-1}]$ is said to be a cluster localization of $A$ if $A[x_1^{-1}, \dots, x_k^{-1}]$ is the cluster algebra of the exchange matrix obtained from  the extended exchange matrix $\B(t)$ by deleting the columns corresponding to $x_1, \dots, x_k$ and considering the corresponding rows as frozen. If we denote by $A'$ the latter cluster algebra, it is always the case that $A' \subseteq A[x_1^{-1}, \dots, x_k^{-1}]$. According to \cite[Lemma 3.4]{muller}, a sufficient condition for $A' = A[x_1^{-1}, \dots, x_k^{-1}]$ is that $A'$ equals its own upper cluster algebra, which happens if, for example, $A'$ is locally acyclic, cf. \cite[Theorem 4.1]{muller}.

\begin{lemma}\label{lem:cluster-localization}
Consider the braid words $\beta = \beta_1\sigma_i\sigma_i\beta_2$ and $\beta' = \beta_1\sigma_i\beta_2$. Let $z_k$ be the $z$-coordinate corresponding to the rightmost middle $\sigma_i$ of $\beta$. Then, $\C[X(\beta)][z_k^{-1}]$ is a cluster localization of $\C[X(\beta)]$, and as cluster varieties $\Spec(\C[X(\beta)][z_k^{-1}])$ is cluster quasi-isomorphic to $X(\beta') \times \C^{\times}$, where the cluster structure on $X(\beta') \times \C^{\times}$ is obtained from the cluster structure on $X(\beta')$ by adding a disjoint frozen vertex.
\end{lemma}
\begin{proof}
Start a Demazure weave for $\beta$ by having a trivalent vertex $\sigma_i\sigma_i \to \sigma_i$. Completing it to a complete Demazure weave $\mathfrak{w}$, this witnesses that $z_k$ is indeed a cluster variable of $\C[X(\beta)]$. Moreover, after freezing the vertex corresponding to $z_k$ in $Q_{\mathfrak{w}}$, we obtain a quiver for the cluster structure on $X(\beta')$, with an extra frozen vertex corresponding to $z_k$. But the cluster structure on $X(\beta')$ is locally acyclic and really full rank, so the result follows from \cite[Theorem 4.1]{muller} and \cite[Proposition 5.11]{LS}. 
\end{proof}

By the previous lemma, $X' := \{z_k \neq 0\} \subseteq X(\beta)$ is an open subset with a projection $\pi: X' \twoheadrightarrow X(\beta')$. Now let $t \in \seeds{\C[X(\beta')]}$ and $T(t) \subseteq X(\beta')$ be the corresponding cluster torus, so that $\pi^{-1}(T(t)) \cong \C^{\times} \times T(t) \subseteq X' \subseteq X(\beta)$ is an open torus. 

\begin{corollary}\label{cor:necessary-for-induction}
The torus $\pi^{-1}(T(t)) \subseteq X(\beta)$ is a cluster torus. 
\end{corollary}
\begin{proof}
The torus $\pi^{-1}(T(t)) \subseteq X'$ is a cluster torus for the cluster structure obtained from a complete Demazure weave $\mathfrak{w'}$ on $\beta'$ after adding a disjoint frozen vertex. By Lemma \ref{lem:cluster-localization} this is quasi-cluster equivalent to the cluster structure on $X'$ given by freezing the vertex corresponding to $z_k$ on the Demazure weave $\mathfrak{w}$ on $\beta$ obtained from $\mathfrak{w}'$ by adding a trivalent vertex at the top. Thus, $\pi^{-1}(T(t))$ is a cluster torus for this cluster structure on $X'$, and thus it is also a cluster torus on $X(\beta)$. 
\end{proof}

The previous corollary is of paramount importance in our work, for it allows us to inductively construct cluster tori on braid varieties: if we want to show that an element $z \in X(\beta)$ belongs to a cluster torus, we can try and search for a decomposition $\beta = \beta_1\sigma_i\sigma_i\beta_2$, in such a way that $z_k \neq 0$ and the projection of $z$ to $X(\beta_1\sigma_i\beta_2)$ belongs to a cluster torus. This is the strategy we will follow.

\subsection{Torus actions}\label{subsec:torus-actions} 
Consider the maximal torus $\mathbb{T} \subseteq \GL(n)$ consisting of diagonal matrices. The torus $\mathbb{T}$ acts on the flag variety $\Fl(n)$, with fixed points being precisely the flags $F^{\bullet}_{w}$ with $w \in S_{n}$. Note that the action factors through the quotient $T := \mathbb{T}/\C^{\times}$ of $\mathbb{T}$ by scalar matrices. It follows from the description of the braid variety via flags that, for any braid $\beta \in \mathrm{Br}^{+}_{n}$, the torus $T$ acts diagonally on the braid variety $X(\beta)$.

We would like to find an explicit formula for the action of $T$ on the coordinates $z_1, \dots, z_{\ell}$ of $X(\beta)$ introduced in Section \ref{subsec:braid-varieties}. We follow \cite[Section 2.2]{CGGS1}.\footnote{Note that \cite{CGGS1} uses different conventions for the braid matrices, so the formulas look a little different from ours.}.
Note that we have
\begin{equation}\label{eq: torus action}
\left(\begin{matrix} t_1 & 0 \\ 0 & t_2 \end{matrix}\right)\left(\begin{matrix} z & -1 \\ 1 & 0\end{matrix}\right) = \left(\begin{matrix} \frac{t_1}{t_2}z & -1 \\ 1 & 0 \end{matrix}\right)\left(\begin{matrix} t_2 & 0 \\ 0 & t_1 \end{matrix}\right).
\end{equation}
We write $w_k := s_{i_{k-1}}\cdots s_{i_1}$ for $k \in [2, l + 1]$ and $w_1 := 1$. We obtain the following:
\begin{lemma}\label{lem:torus-action}
Let $\beta = \sigma_{i_1}\cdots \sigma_{i_{\ell}} \in \Br_n^{+}$ be a positive braid, and define an action of $T$ on $\C^{\ell}$ via:
\[
(t_1, \dots, t_{n})\cdot(z_1, \dots, z_{\ell}) = \left( \frac{t_{w_k(i_k)}}{t_{w_k(i_k+1)}}z_{k}\right)_{k = 1}^{\ell}.
\]
This $T$-action preserves $X(\beta) \subseteq \C^{\ell}$. 
\end{lemma}
\begin{proof}
By applying \eqref{eq: torus action} iteratively, we obtain:
\begin{equation}\label{eq: explicit action}
\diag(t_1, \dots, t_n)B_{\beta}(z_1, \dots, z_n) = B_{\beta}\left(\frac{t_{i_1}}{t_{i_1+1}}z_1, \frac{t_{s_{i_1}(i_2)}}{t_{s_{i_1}(i_2+1)}}z_2, \dots, \frac{t_{w_{\ell}(i_{\ell})}}{t_{w_{\ell}(i_{\ell} + 1)}}z_{\ell}\right)\diag(t_{w_{\ell + 1}(1)}, \dots, t_{w_{\ell + 1}(n)}). 
\end{equation}
Now, assume that $z = (z_1, \dots, z_\ell) \in X(\beta) \subseteq \C^n$, so that $\delta(\beta)^{-1}B_{\beta}(z)$ is upper triangular. We want to show that for any $t = (t_1, \dots, t_n) \in T$, the matrix $\delta(\beta)^{-1}B_{\beta}(t\cdot z)$ is upper triangular. From \eqref{eq: explicit action} we obtain:
\[
\begin{array}{rl}
\delta(\beta)^{-1}B_{\beta}(t\cdot z)  = &  \delta(\beta)^{-1}\diag(t_1, \dots, t_n)B_{\beta}(z)\diag(t_{w_{\ell + 1}(1)}, \dots, t_{w_{\ell + 1}(n)})^{-1}\\
= & \diag(t_{\delta(\beta)(1)}, \dots, t_{\delta(\beta)(n)})\delta(\beta)^{-1}B_{\beta}(z)\diag(t_{w_{\ell + 1}(1)}, \dots, t_{w_{\ell + 1}(n)})^{-1}
\end{array}
\]
which is upper triangular since it is the product of an upper triangular matrix with diagonal matrices. 
\end{proof}

Thus, we obtain an action of $T$ on $X(\beta)$ where each variable $z_{k}$ $(k = 1, \dots, \ell)$ is homogeneous of weight
\begin{equation}\label{eq: weight of zk}
\wt(z_{k}) = \be(w_{k}(i_{k})) - \be(w_{k}(i_{k}+1))
\end{equation}
where $\be(1), \dots, \be(n)$ are the standard basis vectors of $\Z^n$ and we identify
\[
\mathfrak{X}(T) = \{(\theta_1, \dots, \theta_n) \in \Z^n \mid \sum_{i = 1}^{n} \theta_i = 0\}.
\] 

The following remark is clear.

\begin{remark} \label{rem:free-action-weights-span-the-weght-lattice}
    Let $z = (z_1, \dots, z_{\ell}) \in X(\beta)$. Then, the torus $T$ acts freely on $z$ if and only if the set $\{\wt(z_{i}) \mid z_{i} \neq 0\}$ spans the weight lattice $\mathfrak{X}(T)$. 
\end{remark}

\begin{remark}\label{rmk:pictures-of-weights}
We can read the weight of a coordinate $z_i$ from a picture of the braid $\beta$, as follows. Consider the strands that are incident on the left to the $i$-th crossing of $\beta$. If the strand incident from the top (resp.~the bottom) to the $i$-th crossing arrives at the $p$-th (resp.~$q$-th) level strand at the leftmost end, then $\wt(z_i) = \be(p) - \be(q)$.  For example, the next figure illustrates that for $\beta = \sigma_1\sigma_2\sigma_2\sigma_1\sigma_2$ we have $\wt(z_3) = \be({\color{red} 3}) - \be({\color{teal} 1})$. 
\begin{center}
    \includegraphics[scale=1]{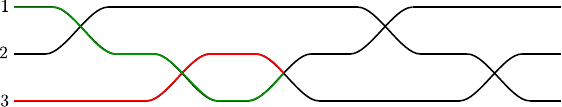}
\end{center}
\end{remark}

The following then generalizes Remark \ref{rem:free-action-weights-span-the-weght-lattice}.

\begin{lemma}\label{lem:stabilizers-of-points}
Let $z = (z_1, \dots, z_{\ell}) \in X(\beta)$. Then,
\[
\Stab_{(\C^{\times})^{n}}(z) = \{(t_1, \dots, t_n) \mid t_i = t_j \; \text{if the} \; i\text{-th and} \; j\text{-th strands cross at the} \; k\text{-th crossing with} \; z_k \neq 0\}.
\]
\end{lemma}
\begin{proof}
If the $i$-th and $j$-th strands cross at the $k$-th crossing, then the action of $(t_1, \dots, t_n)$ on $(z_1, \dots, z_{\ell})$ rescales $z_k$ by $(t_i/t_j)^{\pm 1}$ and the result follows. 
\end{proof}

The following result can be checked directly. 

\begin{lemma}\label{lem:braid-moves-are-equivariant}
Let $\beta$ and $\beta'$ be positive braid words differing by braid moves. Then, the isomorphism $X(\beta) \to X(\beta')$ induced by \eqref{eq:matrix braid relations} is $T$-equivariant. 
\end{lemma}

\begin{lemma}[Section 2.4, \cite{CGGS1}] Let $\mathfrak{w}$ be a weave on $\beta$, and $v$ a trivalent vertex of $\mathfrak{w}$. Then, the $s$-variable $s_{v}$ is homogeneous with respect to the $T$-action. 
\end{lemma}

From Theorem \ref{thm:cluster-structure}(b) we obtain the following result. 

\begin{corollary} \label{cor:torus-acts-by-aut}
The cluster variables associated to a Demazure weave $\mathfrak{w}: \beta \to \delta(\beta)$ are homogeneous with respect to the $T$-action. Thus, $T$ acts on $X(\beta)$ by cluster automorphisms. 
\end{corollary}

We remark that the cluster structure on $X(\beta)$ is of really full rank, cf. Theorem \ref{thm:cluster-structure} (d). Thus, the group of cluster automorphisms is a torus of rank equal to the number of frozen variables in the corresponding cluster structure, cf. Remark \ref{rmk:rk-in-full-rank-case}. Let us say that $j = 1, \dots, |\beta|$ is an \emph{essential crossing} of $\beta$ if $\delta(\sigma_{i_1}\dots\widehat{\sigma_{i_j}}\dots \sigma_{i_{|\beta|}}) < \delta(\beta)$ where the hat means that the corresponding crossing is ommited. Otherwise, we say that $j$ is a non-essential crossing. It is known that $j$ is an essential crossing if and only if $z_j \equiv 0$ on the braid variety $X(\beta)$ \footnote{We thank Linhui Shen for this observation.}, cf. \cite[Lemma 5.11.(1)]{CGGLSS}.

\begin{lemma}\label{lem:dim-ker}
Let $\beta = \sigma_{i_1}\dots \sigma_{i_{\ell}} \in \Br_n^{+}$ be a positive braid, and let $z_1, \dots, z_{\ell}$ be the coordinates on the braid variety from Definition \ref{def:braid-variety}. Let
\[
L:= \langle \wt(z_j) \mid j = 1, \dots, \ell \; \text{is a non-essential crossing}\rangle \subseteq \mathfrak{X}(T), \qquad k := \rank(L)
\]
Then, the map $T \to \Aut(X(\beta))$ has kernel of dimension $n-1-k$. 
\end{lemma}
\begin{proof}
 It follows from \eqref{eq: weight of zk} that the quotient $\mathfrak{X}(T)/L$ is a free abelian group. So we can find a torus $T' \subseteq T$ such that $L = \mathfrak{X}(T/T')$. We claim that the kernel of the map $T \to \Aut(X(\beta))$ is precisely $T'$, from where the result follows.
One inclusion is clear: if $t' \in T'$ then $\wt(z_j)(t') = 1$ for all non-essential crossings $j = 1, \dots, \ell$, so $t'$ acts trivially on $X(\beta)$. Assume then that $t \not\in T'$, so that there exists a non-essential crossing $j$ such that $\wt(z_j)(t) \neq 1$. Since $z_j$ is a cluster monomial, cf. Theorem \ref{thm:cluster-structure}(e), we can find an element in the braid variety $X(\beta)$ with $z_j \neq 0$. Then, $t$ does not act trivially on this element and we are done. 
\end{proof}

\begin{corollary} \label{cor:map-surjective}
Let $\beta = \beta_{i_1}\dots\beta_{i_{\ell}} \in \Br^{+}_n$ be a positive braid. Let $k := \rank\langle \wt(z_j) \mid j = 1, \dots, \ell \; \text{is a non-essential crossing}\rangle$. Then, the map $T \to \Aut(X(\beta))$ is surjective if and only if $X(\beta)$ has $k$ frozen variables. 
\end{corollary}

In general, the cluster structure on $X(\beta)$, $\beta \in \mathrm{Br}^{+}_{n}$ can have more than $n-1 = \rank(T)$ frozen variables. 

\begin{example}\label{ex:T-not-Aut-for-braid-varieties}
Consider the braid $\beta = \sigma_1\sigma_1\sigma_2\sigma_2\sigma_1\sigma_1 \in \Br^{+}_{3}$. It is easy to see that $X(\beta) \cong (\C^{\times})^{3}$,
\[
X(\beta) := \{(z_1, z_1^{-1}, z_2, z_2^{-1}, z_3,z_3^{-1}) \mid (z_1, z_2, z_3) \in (\C^{\times})^{3}\} \cong (\C^{\times})^{3}.
\]
In particular, $X(\beta)$ is a cluster variety with three isolated frozen variables, $z_1, z_2, z_3$. 
All crossings in $\beta$ are non-essential. Under the action of $T = (\C^{\times})^{3}/\C^{\times}$, we have
\[
\wt(z_1) = \wt(z_3) = \be(1) - \be(2), \qquad \wt(z_2) = \be(2) - \be(3).
\]
We see that $L$ is the entire weight lattice  $\mathfrak{X}(T)$ and has rank $2$, which agrees with the rank of $T$. 
So the map $T \to \Aut(\C[X(\beta)])$ is injective by Lemma~\ref{lem:dim-ker}. It is not surjective by Corollary~\ref{cor:map-surjective}, since $X(\beta)$ has $3 > 2$ frozen variables. 
\end{example}

In Section \ref{subsec:double-bs}, we will introduce the class of \emph{double Bott-Samelson varieties}; this is a large class of braid varieties for which the map $T \to \Aut(\C[X(\beta)])$ is surjective.
In Section \ref{subsec:positroid}, we will mention that this also holds for positroid varieties.

\subsection{Richardson varieties} Every element $w \in S_n$ defines a Schubert cell
\[
C_w := \borel(w\borel) \subseteq \Fl(n),
\]
and an opposite Schubert cell
\[
C_w^{-} := \borel_{-}(w\borel) \subseteq \Fl(n).
\]
It is known that $C_w$ and $C_w^{-}$ are affine spaces of dimensions $\ell(w)$ and $\ell(ww_0)$, respectively, where $w_0 \in S_n$ is the longest element. For $u, w \in S_n$ the \emph{Richardson variety} is
\[
R(u, w) := C_u^{-} \cap C_w \subseteq \Fl(n).
\]
It is known that this variety is nonempty if and only if $u \leq w$ in Bruhat order. In this case, $R(u, w)$ is a smooth affine variety of dimension $\ell(w) - \ell(u)$. Note that the torus $T$ from Section \ref{subsec:torus-actions} acts naturally on $\Fl(n)$, and this action preserves $R(u, w)$.

\begin{theorem} \cite{CGGS_positroid}
Let $u, w \in S_n$ with $u \leq w$. Let $\beta(u^{-1}w_0) \in \Br^{+}_{n}$ denote a minimal-length positive lift of the element $u^{-1}w_0 \in S_n$, and similarly for $\beta(w) \in \Br^{+}_{n}$. Then there exists a $T$-equivariant isomorphism
\[
R(u, w) \cong X(\beta(w)\beta(u^{-1}w_0)).
\]
\end{theorem}

The torus $T$ is a subgroup of the cluster automorphism group of a Richardson variety. 
Galashin and Lam~\cite[Corollary 4.8]{GL20} show that, if $u^{-1} w$ is an $n$-cycle, then the action of $T$ on $R(u,w)$ is free.
However, this result falls short of determining when the stabilizer locus is empty in two ways. Firstly, $T$ can be a proper group of $\Aut(R(u,w))$ and, secondly, Galashin and Lam's result is not an if and only if result. Both phenomena can be seen in the Richardson variety $R(s_2, s_1 s_3 s_2 s_1 s_3)$. This variety is isomorphic to $(\mathbb{C}^{\ast})^4$, and its cluster automorphism group is $4$-dimensional, although $T$ is only $3$-dimensional. Moreover, $T$ acts freely on it even though $(s_2)^{-1} (s_1 s_3 s_2 s_1 s_3)$ is not a $4$-cycle. 
We regard determining criteria for $\Aut(R(u,w))$ to act freely, and determining how much larger $\Aut(R(u,w))$ can be than $T$, as interesting open problems.

One case in which we can solve these problems is when $u=e$.
$R(e, w)$ is open in an affine space, as follows. First, we have a parametrization of the Schubert cell $C_w$ by the set
\[
\cC(w) = \{M \in \GL(n) \mid M_{w(i), i} = 1, \qquad M_{p, q} = 0 \; \text{if} \; p > w(q) \; \text{or} \; w^{-1}(p) < q.\}
\]
Now, a matrix in $\cC(w)$ represents a flag belonging to $C^{-}_{e}$ if and only if this matrix belongs to $\borel_{-}\borel$, that is, if and only if all of its principal minors are nonzero. Thus,
\[
R(e, w) = \{M \in \cC(w) \mid \minor_{[1, i], [1,i]}(M) \neq 0 \; \text{for all} \; i = 1, \dots, n-1\}.
\]
where $\minor_{I,J}(M)$ is the minor on the rows $I$ and columns $J$ of a matrix $M$. 

\begin{remark}
$R(e,w)$ is also the reduced double Bruhat cell $G^{e,w}$, and is dense in the Schubert cell $C_w$, which can also be identified with the unipotent cell corresponding to $w$. From these perspectives, cluster structures on $R(e,w)$ can be understood using works of Berenstein, Fomin and Zelevinsky~\cite{BFZ} and Geiss, Leclerc and Schr\"oer~\cite{GLS11}.
\end{remark}

\begin{lemma}\label{lem:richardson-free} 
Let $w \in S_n$. Then, the torus $T$ acts freely on $R(e, w)$ if and only if $w$ is an $n$-cycle. 
\end{lemma}
\begin{proof}
If $w$ is an $n$-cycle, then $T$ acts freely on $R(e,w)$ by~\cite[Lemma 2.10]{CGGS1}, see also Lemma~\ref{lem:stabilizers} below.

Now assume that $w$ is not a single cycle, and let $w = w_1\dots w_k$ be the cycle decomposition of $w$. Let $C_i$ be the support of the cycle $w_i$, and let $c_i := |C_i|$. For every $i$, let $M_i \in \cC(w_i) \subseteq \GL(c_i)$ be such that all its principal minors are nonzero, so that $M_i$ represents an element in $R(e, w_i)$. We remark that, here, we are thinking of the columns of $M_i$ as indexed by $C_i$, and the rows are indexed by $w_i(C_i)$. 

Let us now define $M \in \GL(n)$ by
\[
M_{p, q} = \begin{cases} (M_i)_{p, q} & \text{if there exist} \; i \; \text{such that} \; p, q \in C_i \\ 0 & \text{else}.  \end{cases} 
\]
It is clear that $M \in \cC(w)$. Note also that a principal minor of $M$ is a product of principal minors of the $M_i$'s, so that $M$ represents an element of $R(e,w)$.

Now note that $T$ does not act freely on $M\borel$: if $t = \diag(t_1, \dots, t_n) \in T$ is such that $t_p = t_q$ whenever $p, q \in C_i$ for some $i$, then $t$ stabilizes $M\borel$. The result follows. 
\end{proof}

The varieties $R(e,w)$ are special cases of double Bott-Samelson varieties, that we define next. As we will see in Lemma \ref{lem:surjective-in-bs-case} below, for double Bott-Samelson varieties the map from $T$ to the cluster automorphism group is surjective, so we can characterize the stabilizer locus entirely in terms of the action of $T$. 

\subsection{Double Bott-Samelson varieties}\label{subsec:double-bs}

\begin{definition}\label{def:double-bs}
Let $\beta \in \Br_n^{+}$ be a positive braid word. The \emph{double Bott-Samelson variety} is
\begin{equation*}
\begin{array}{rl}
\BS(\beta)  & =  \{(z_1, \dots, z_{\ell(\beta)}) \mid B_{\beta}(z_1, \dots, z_{\ell(\beta)}) \; \text{admits an} \; LU\text{-decomposition}\}
\\
& = \{(z_1, \dots, z_{\ell(\beta)}) \mid \minor_{[1,i],[1,i]}B_{\beta}(z_1, \dots, z_{\ell(\beta)}) \neq 0 \; \text{for all} \; i = 1, \dots, n-1\}.
\end{array}
\end{equation*}
\end{definition}

The following is \cite[Lemma 3.16]{CGGLSS} (and its proof). There, $\BS(\beta)$ was denoted by $\mbox{Conf}(\beta)$. Let us denote by $\Delta \in \Br^{+}_{n}$ a minimal lift of the longest element $w_0 \in S_{n}$. 

\begin{lemma}\label{lem:bs-as-braid}
    Let $\beta \in \Br_n^{+}$ be a positive braid word. Then, we have an isomorphism
    \[
    \begin{array}{c}
    X(\Delta\beta) \cong \BS(\beta) \\ (z_1, \dots, z_{\binom{n}{2}}, z_{\binom{n}{2}+1}, \dots, z_{\binom{n}{2} + |\beta|}) \mapsto  (z_{\binom{n}{2}+1}, \dots, z_{\binom{n}{2} + |\beta|})
    \end{array}
    \]
\end{lemma}

Cluster structures for double Bott-Samelson varieties were constructed in \cite{shen2021cluster}. In fact, thanks to \cite[Proposition 5.7]{CGGLSS}, the cluster structure on $\BS(\beta)$ constructed in \cite{shen2021cluster} coincides with that using weaves. 

\begin{remark}
By Lemma \ref{lem:torus-action}, there is an action of $T = (\C^{\times})^{n}/\C^{\times}$ on $\BS(\beta) \cong X(\Delta\beta)$. We will use a small modification of this action, as follows: instead of having the torus $T$ act on every coordinate $z_1, \dots, z_{\binom{n}{2}}, z_{\binom{n}{2} + 1}, \dots, z_{\binom{n}{2} + \ell(\beta)}$ corresponding to the crossings of the word $\Delta\beta$, we just let it act on the coordinates corresponding to $\beta$, with the weights given as in \eqref{eq: weight of zk}, that is, for $(z_1, \dots, z_{\ell}) \in \BS(\beta)$,
\[
(t_1, \dots, t_{n})\cdot(z_1, \dots, z_{\ell}) = \left( \frac{t_{w_k(i_k)}}{t_{w_k(i_k+1)}}z_{k}\right)_{k = 1}^{\ell}.
\]
It follows from \eqref{eq: explicit action} and the Cauchy-Binet formula that this does define an action of $T$ on $\BS(\beta)$. Moreover, it follows from \eqref{eq: torus action} that this action coincides with the action given by pulling back the action on $X(\Delta\beta)$ upon the isomorphism $X(\Delta\beta) \cong \BS(\beta)$, up to an automorphism of $(\C^{\times})^{n}$ exchanging the coordinates. When we discuss torus actions on double Bott-Samelson varieties we will always refer to the action defined in this remark, which is more simple-minded than the action on $X(\Delta\beta)$. 
\end{remark}
Note that if the word $\beta$ is the minimal lift of $w \in S_n$ then we have a $T$-equivariant natural isomorphism
\[
\BS(\beta) \cong R(e, w).
\]

\begin{lemma}\label{lem:richardson-as-closed}
Let $\beta \in \Br_n^{+}$ be a positive braid word and let $w = \pi(\beta)$. (Recall that $\pi: \Br_n \to S_n$ is the group homomorphism with $\pi(\sigma_i) = (i\ i+1)$.) Then, the Richardson variety $R(e,w)$ is isomorphic to a $T$-stable closed subvariety of $\BS(\beta)$. 
\end{lemma}
\begin{proof}
If $\beta$ is reduced then in fact $\BS(\beta) = R(e,w)$. Now, if $\beta$ is not reduced then, after possibly applying braid moves, we can assume that $\beta$ has the form $\beta = \beta_1\sigma_i\sigma_i\beta_2$, see Lemma \ref{lem:braid-moves-are-equivariant}, where the first $\sigma_i$ is in position $k$ and the second $\sigma_i$ is in position $k+1$. Then, the locus $z_k = z_{k+1} = 0$ is closed, $T$-stable and isomorphic to $\BS(\beta_1\beta_2)$. Now we proceed recursively until reaching a reduced word.  
\end{proof}

\begin{lemma}\label{lem:stabilizers}
Let $\beta \in \Br^{+}_{n}$ and let $w = \pi(\beta)$. Define the torus $T_{\beta} \subseteq (\C^{\times})^{n}$ by:
\[
T_{\beta} = \{(t_1, \dots, t_n) \in (\C^{\times})^{n} \mid t_i = t_{w(i)} \; \text{for every} \; i = 1, \dots, n\}.
\]
Then, for every $z = (z_1, \dots, z_{\ell}) \in \BS(\beta)$:
\[
\Stab_{(\C^{\times})^{n}}(z) \subseteq T_{\beta}. 
\]
\end{lemma}
\begin{proof}
Let $t = (t_1, \dots, t_n) \in \Stab_{(\C^{\times})^{n}}(z)$, so that by \eqref{eq: torus action} we have the identity:
\[
\diag(t_{1}, \dots, t_{n})B_{\beta}(z) = B_{\beta}(z)\diag(t_{w(1)}, \dots, t_{w(n)}).
\]
Computing principal minors on both sides we obtain that, for every $i$,
\[
t_1\cdots t_i \minor_{[1,i],[1,i]}B_{\beta}(z) = t_{w(1)}\cdots t_{w(i)}\minor_{[1,i], [1,i]}B_{\beta}(z),
\]
and since $z \in \BS(\beta)$ this implies that $t_1\cdots t_i = t_{w(1)}\cdots t_{w(i)}$ for every $i = 1, \dots, n$. Inductively, we obtain $t_i = t_{w(i)}$ for every $i = 1, \dots, n$, so $t \in T_{\beta}$, as needed. 
\end{proof}

\begin{corollary}\label{cor:free-action}
Let $\beta \in \Br^{+}_{n}$ and let $w = \pi(\beta)$. Then, the torus $T$ acts freely on $\BS(\beta)$ if and only if $w \in S_n$ is an $n$-cycle.
\end{corollary}
\begin{proof}
If $w$ is an $n$-cycle, then Lemma \ref{lem:stabilizers} (see also \cite[Lemma 2.10]{CGGS1}) implies that $T$ acts freely on $\BS(\beta)$. The converse follows from Lemmas \ref{lem:richardson-as-closed} and \ref{lem:richardson-free}. 
\end{proof}

Let us now discuss the cluster structure on $\BS(\beta)$ following \cite{shen2021cluster}. The cluster variables are in bijection with the letters of $\beta$. In fact, one cluster is given by
\[
c_j = \minor_{[1,i_j],[1,i_j]}B_{\sigma_{i_1}\dots\sigma_{i_j}}(z_1, \dots, z_j).
\]
A cluster variable $c_j$ is frozen if and only if it corresponds to the rightmost appearance of a braid generator $\sigma_{i_j}$ in $\beta$. It follows that the number of frozen variables corresponds to the number of different braid generators appearing in $\beta$: in particular, there are at most $n-1$ frozen variables. The quiver $Q_{\beta}$ for this particular seed is given by the \emph{amalgamation procedure} of Fock-Goncharov and can be constructed inductively as follows. First, the vertices of $Q$ are colored by $1, \dots, n-1$: a vertex of color $i$ corresponds to an appearance of $\sigma_i$ in $\beta$, and an $i$-colored vertex is frozen if and only if it corresponds to the rightmost appearance of $\sigma_i$. Assume that the quiver $Q_{\beta}$ has been constructed. To construct the quiver $Q_{\beta\sigma_i}$:
\begin{itemize}
    \item[(1)] Thaw the frozen vertex $\bullet$ of $Q_\beta$ corresponding to the rightmost appearance of $\sigma_i$ in $\beta$ (if any).
    \item[(2)] Add a new frozen vertex ${\color{blue} \blacksquare}$ and an arrow $\bullet \to {\color{blue} \blacksquare}$. This new frozen vertex will correspond to the rightmost $\sigma_i$ in $\beta\sigma_i$, and it is colored by $i$. 
    \item[(3)] For $j \in \{i-1, i+1\}$ let ${\color{red} \blacksquare}$ be the $j$-colored frozen vertex in $Q(\beta)$ corresponding to the rightmost appearance of $\sigma_j$ in $\beta$. 
    \begin{itemize}
        \item[(3a)] If ${\color{red} \blacksquare}$ already has an arrow to a vertex of color $i$ in $Q(\beta)$, do nothing.
        \item[(3b)] Else, add an arrow ${\color{red} \blacksquare} \to \bullet$, where $\bullet$ is the $i$-colored newly thawed vertex from Step (1). 
    \end{itemize}
\end{itemize}

\begin{example}\label{ex:type-d5}
(a) For the word $\beta = \sigma_1^{n}$, the quiver $Q_{\beta}$ is an equioriented $A_{n-1}$ quiver with a single frozen target. \\

(b)For the word $\beta = \sigma_1^{3}\sigma_2\sigma_1^{2}\sigma_2$ we have that $Q_{\beta}$ is the quiver:

\begin{center}
\begin{tikzpicture}
\node at (0,0) {$\circ$};
\draw[->] (0.2,0)--(0.8,0);
\node at (1,0) {$\circ$};
\draw [->] (1.2, 0) -- (1.8,0);
\node at (2,0) {$\circ$};
\draw [->] (2.2, 0) -- (2.8,0);
\node at (3,0) {$\circ$};
\draw [->] (3.2, 0) -- (3.8,0);
\node at (4,0) {$\square$};

\node at (3,-1) {$\circ$};
\draw [->] (3.2, -1) -- (3.8,-1);
\node at (4,-1) {$\square$};

\draw [->] (4, -0.2) -- (3.1, -0.9);
\draw[->] (2.9, -0.9) -- (2.1, -0.2);
\end{tikzpicture}
\end{center}
Which is a quiver of type $D_5$ with two frozen vertices. More generally, for $\beta = \sigma_1^{n-2}\sigma_2\sigma_1^{2}\sigma_2$, $Q_{\beta}$ is a quiver of type $D_n$ with two frozen vertices, see e.g. \cite[Definition 1.2]{GSW_positive}. \\

(c) For the word $\beta = (\sigma_1\sigma_2)^{4}$ we have that $Q_{\beta}$ is the following quiver of type $E_6$:
\begin{center}
\begin{tikzpicture}
\node at (0,0) {$\circ$};
\draw[->] (0.2,0)--(0.8,0);
\node at (1,0) {$\circ$};
\draw [->] (1.2, 0) -- (1.8,0);
\node at (2,0) {$\circ$};
\draw [->] (2.2, 0) -- (2.8,0);
\node at (3,0) {$\square$};

\node at (0,-1) {$\circ$};
\draw[->] (0.2,-1)--(0.8,-1);
\node at (1,-1) {$\circ$};
\draw [->] (1.2, -1) -- (1.8,-1);
\node at (2,-1) {$\circ$};
\draw [->] (2.2, -1) -- (2.8,-1);
\node at (3,-1) {$\square$};

\draw[->] (0, -0.9) -- (0, -0.1);
\draw[->] (1, -0.9) -- (1, -0.1);
\draw[->] (2, -0.9) -- (2, -0.1);

\draw[->] (0.8, -0.2)--(0.2, -0.8);
\draw[->] (1.8, -0.2)--(1.2, -0.8);
\draw[->] (2.8, -0.2)--(2.2, -0.8);
\end{tikzpicture}
\end{center}

(d) For the word $\beta = (\sigma_1\sigma_2)^{4}\sigma_1$, $Q_{\beta}$ is a quiver of type $E_7$: 
\begin{center}
\begin{tikzpicture}
\node at (0,0) {$\circ$};
\draw[->] (0.2,0)--(0.8,0);
\node at (1,0) {$\circ$};
\draw [->] (1.2, 0) -- (1.8,0);
\node at (2,0) {$\circ$};
\draw [->] (2.2, 0) -- (2.8,0);
\node at (3,0) {$\circ$};
\draw[->] (3.2, 0) -- (3.8, 0);
\node at (4,0) {$\square$};

\node at (0,-1) {$\circ$};
\draw[->] (0.2,-1)--(0.8,-1);
\node at (1,-1) {$\circ$};
\draw [->] (1.2, -1) -- (1.8,-1);
\node at (2,-1) {$\circ$};
\draw [->] (2.2, -1) -- (2.8,-1);
\node at (3,-1) {$\square$};

\draw[->] (0, -0.9) -- (0, -0.1);
\draw[->] (1, -0.9) -- (1, -0.1);
\draw[->] (2, -0.9) -- (2, -0.1);
\draw[->] (3, -0.8) -- (3, -0.1);

\draw[->] (0.8, -0.2)--(0.2, -0.8);
\draw[->] (1.8, -0.2)--(1.2, -0.8);
\draw[->] (2.8, -0.2)--(2.2, -0.8);
\end{tikzpicture}
\end{center}

(e) For the word $\beta = (\sigma_1\sigma_2)^{5}$, $Q_{\beta}$ is a quiver of type $E_8$:
\begin{center}
\begin{tikzpicture}
\node at (0,0) {$\circ$};
\draw[->] (0.2,0)--(0.8,0);
\node at (1,0) {$\circ$};
\draw [->] (1.2, 0) -- (1.8,0);
\node at (2,0) {$\circ$};
\draw [->] (2.2, 0) -- (2.8,0);
\node at (3,0) {$\circ$};
\draw[->] (3.2, 0) -- (3.8, 0);
\node at (4,0) {$\square$};

\node at (0,-1) {$\circ$};
\draw[->] (0.2,-1)--(0.8,-1);
\node at (1,-1) {$\circ$};
\draw [->] (1.2, -1) -- (1.8,-1);
\node at (2,-1) {$\circ$};
\draw [->] (2.2, -1) -- (2.8,-1);
\node at (3,-1) {$\circ$};
\draw[->] (3.2, -1) -- (3.8, -1);
\node at (4,-1) {$\square$};

\draw[->] (0, -0.9) -- (0, -0.1);
\draw[->] (1, -0.9) -- (1, -0.1);
\draw[->] (2, -0.9) -- (2, -0.1);
\draw[->] (3, -0.9) -- (3, -0.1);

\draw[->] (0.8, -0.2)--(0.2, -0.8);
\draw[->] (1.8, -0.2)--(1.2, -0.8);
\draw[->] (2.8, -0.2)--(2.2, -0.8);
\draw[->] (3.8, -0.2)--(3.2, -0.8);
\end{tikzpicture}
\end{center}

\end{example}

\subsection{Conjecture \ref{conj:main} in the case of double Bott-Samelson varieties} In this section, we make Conjecture \ref{conj:main} more explicit for double Bott-Samelson varieties. Recall that, if $T = (\C^{\times})^{n}/\C^{\times}$ is the maximal torus in $\mathrm{PGL}(n)$, then there is a map $T \to \Aut(\BS(\beta))$. Note also that every crossing belonging to $\beta$ in the word $\Delta\beta$ is non-essential. Thus, if $\beta = \sigma_{i_1}\cdots \sigma_{i_{\ell}}$ then the kernel of the map $T \to \Aut(\BS(\beta))$ is the torus $T' \subseteq T$ such that
\[
\mathfrak{X}(T/T') = \langle \wt(z_i) \mid i = 1, \dots, \ell \rangle. 
\]

We are interested in finding $\langle \wt(z_i) \mid i = 1, \dots, \ell\rangle$. Let us say that a crossing $\sigma_{i_j}$ of $\beta$ is \emph{special} if $i_j \neq i_{j'}$ for $j' > j$. Thus, special crossings correspond to the rightmost appearances of the braid generators in $\beta$. Let us define
\[
L_{\beta} = \langle \wt(z_i)| i = 1, \dots \ell\rangle, \qquad L'_{\beta} = \langle \wt(z_i) \mid i \; \text{is a special crossing of} \;  \beta\rangle.
\]

It is clear that $L'_{\beta} \subseteq L_{\beta} \subseteq \mathfrak{X}(T)$. Let us study the group $L'_{\beta}$. For this, let us define a graph $\Gamma_{\beta}$ as follows: the vertices are $1, \dots, n$, and there is an edge between $i$ and $j$ if the $i$-th and $j$-th strands of $\beta$ cross at a special crossing. The reason why we are interested in the graph $\Gamma_{\beta}$ is that $\be(i) - \be(j) \in L'_{\beta}$ if $i$ and $j$ belong to the same connected component of $\Gamma_{\beta}$. In particular, if $\Gamma_{\beta}$ is connected then $L'_{\beta} = L_{\beta} = \mathfrak{X}(T)$. Thus, we are interested in finding the number of connected components $h_0(\Gamma_{\beta})$.

\begin{lemma}\label{lem:graph}
\begin{enumerate}
\item If the braid generator $\sigma_k$ does not appear in $\beta$, then
\[
h_0(\Gamma_{\beta\sigma_k}) = h_0(\Gamma_{\sigma_k\beta}) = h_0(\Gamma_{\beta}) - 1. 
\]

\item If the braid generator $\sigma_k$ appears in $\beta$ then
\[
h_0(\Gamma_{\beta\sigma_k}) = h_0(\Gamma_{\sigma_k\beta}) = h_0(\Gamma_{\beta}). 
\]

\item Let $s$ be the number of braid generators which do \emph{not} appear in $\beta$. Then $h_0(\Gamma_{\beta}) = s+1$. 
\end{enumerate}
\end{lemma}
\begin{proof}
(3) follows from (1) and (2) after observing that, if $\beta$ is the empty word, then $h_0(\Gamma_{\beta}) = n$. Let us show (1). Note that $\sigma_k$ is a special crossing of both $\sigma_k\beta$ and $\beta\sigma_k$, and that $\Gamma_{\sigma_k\beta}$ is obtained from $\Gamma_{\beta}$ by replacing every edge $(k) - (i)$ with an edge $(k+1)-(i)$ and vice versa, plus adding a new edge $(k) - (k+1)$. Thus, the connected components of $\Gamma_{\sigma_k\beta}$ are obtained from the connected components of $\Gamma_{\beta}$ by merging the components containing $k$ and $k+1$. For $\Gamma_{\beta\sigma_k}$, if $i_1$ and $i_2$ are the strands incident to the crossing $\sigma_k$, then $\Gamma_{\beta\sigma_k}$ is obtained from $\Gamma_{\beta}$ by adding a new edge $(i_1) - (i_2)$, which merges the connected components of $\Gamma_{\beta}$ containing $i_1$ and $i_2$.  In both cases, the result follows. 

Let us now show (2). Note that $\Gamma_{\sigma_k\beta}$ and $\Gamma_{\beta}$ are isomorphic via the isomorphism $k \leftrightarrow k+1$, so $h_0(\Gamma_{\sigma_k\beta}) = h_0(\Gamma_{\beta})$, as wanted. For $\beta\sigma_k$, let $m(k)$ be such that $i_{m(k)} = k$ and $i_{j} \neq k$ for $m(k) < j \leq \ell$. Let $\beta'$ be the subword $\sigma_{m(k)+1}\cdots \sigma_{\ell}$, so that by (1) and the equality $h_0(\Gamma_{\sigma_k\beta}) = h_0(\Gamma_{\beta})$ when $\sigma_k$ appears in $\beta$, we have
\[
h_0(\Gamma_{\sigma_k\beta'}) = h_0(\Gamma_{\beta'\sigma_{k}}) = h_0(\Gamma_{\sigma_k\beta'\sigma_k}).
\]
Now we complete $\sigma_k\beta'$ and $\sigma_k\beta'\sigma_k$ to $\beta$ and $\beta\sigma_k$, respectively, by adding letters one by one on the left. At each step, this will have the same effect on $h_0$. The result follows. 
\end{proof}

\begin{lemma}\label{lem:surjective-in-bs-case}
The map $T \to \Aut(\BS(\beta))$ is surjective, and it is injective if and only if every braid generator appears in $\beta$.
\end{lemma}
\begin{proof}
The torus $\Aut(\BS(\beta))$ has dimension the number of frozen variables in the cluster structure of $\BS(\beta)$, which is precisely the number of distinct braid generators appearing in $\beta$. If $\beta$ does not contain all braid generators, then we can find a parabolic subgroup $\Br_{n_1} \times \cdots \times \Br_{n_k} \subseteq \Br_n$ such that  $\beta \in \Br_{n_1} \times \cdots \times \Br_{n_k}$, and $\beta$ contains all generators $\sigma_i$ of the parabolic subgroup $\Br_{n_1} \times \cdots \times \Br_{n_k}$. If $\beta = \beta_1\cdots \beta_k$ with $\beta_i \in \Br^{+}_{n_i}$, note that
\begin{equation} \label{eq:BS-decomposition}
\BS(\beta) \cong \BS(\beta_1) \times \cdots \times \BS(\beta_k).
\end{equation}
Note also that we have subtori $T_{n_i}$ in $T$, so that $T_{n_i}$ acts on $\BS(\beta_i)$. The codimension of $T_{n_1} \times \cdots \times T_{n_k}$ in $T$ is $k$, which is precisely the number of missing braid generators in $\beta$. Thus, it is enough to show that if $\beta$ contains all braid generators $\sigma_i$, then the map $T \to \Aut(\BS(\beta))$ is an isomorphism.

For this, since every braid generator appears in $\beta$ it follows from Lemma \ref{lem:graph} (c) that $\Gamma_{\beta}$ is a single connected component, and thus that $\be(i) - \be(j) \in L'_{\beta} \subseteq L_{\beta}$ for every $i \neq j \in \{1, \dots, n-1\}$. Thus in this case $L_{\beta} = L'_{\beta} = \mathfrak{X}(T)$, and the result follows from Lemma \ref{lem:dim-ker}. 
\end{proof}

\begin{remark}
    Note that it also follows that, in general, $L_{\beta} = L'_{\beta}$. 
\end{remark}

Note that, if $\beta = \beta_1\cdots \beta_k \in \Br^{+}_{n_1} \times \cdots \times \Br^{+}_{n_k}$ contains all the braid generators of the parabolic subgroup $\Br^{+}_{n_1} \times \cdots \times \Br^{+}_{n_k}$ then Lemma \ref{lem:surjective-in-bs-case} actually shows that
\[
\Aut(\BS(\beta)) \cong T_{n_1} \times \cdots \times T_{n_k},
\]
while Corollary \ref{cor:free-action} asserts that $T_{n_i}$ acts freely on $\BS(\beta_i)$ if and only if the Coxeter projection of $\beta_i$ is a single cycle. Thus, we obtain the following.

\begin{corollary}\label{cor:free-action-doublebs}
Let $\beta \in \Br_{n}^{+}$ be a positive braid word. Then, the stabilizer locus $\cS(\BS(\beta)) = \emptyset$ if and only if any two components of the braid closure of $\beta$ are unlinked. In particular, if every braid generator appears in $\beta$ then $\cS(\BS(\beta)) = \emptyset$ if and only if the Coxeter projection of $\beta$ is an $n$-cycle. 
\end{corollary}

\begin{corollary} \label{cor:GCDCriterion}
Let $\A(k,n) \subseteq \Gr(k,n)$ be the maximal positroid stratum. Then $\cS(\A(k,n)) = \emptyset$ if and only if $\gcd(k,n) = 1$. In particular, if $\gcd(k,n) > 1$ then $\per(\A(k,n)) \neq \emptyset$. 
\end{corollary}
\begin{proof}
    The maximal positroid stratum in $\Gr(k,n)$ is isomorphic to the Richardson variety $R(e, w)$ where $w \in S_n$ is the maximal $k$-Grassmannian permutation, i.e. in one-line notation $w = [n-k+1, \dots, n, 1, \dots, n-k]$. So $\A(k,n) \cong \BS(\beta)$ where $\beta \in \Br_n^+$ is a minimal lift of $w$. It remains to verify that $w$ is an $n$-cycle if and only if $k$ and $n$ are coprime, which is straightforward. 
\end{proof}

Assuming the validity of Conjecture \ref{conj:main} we actually have that $\per(\A(k,n)) = \emptyset$ if and only if $\gcd(k,n) = 1$. More generally, the following is then a special case of Conjecture \ref{conj:main} in the case of double Bott-Samelson varieties.

\begin{conjecture}\label{conj:conj-in-bs-case}
Let $\beta \in \Br^{+}_{n}$ be a positive braid word. The following conditions are equivalent.
\begin{itemize}
    \item[(a)] $\per(\BS(\beta)) = \emptyset$. 
    \item[(b)] Any two components of the braid closure of $\beta$ are unlinked.
\end{itemize}
\end{conjecture}

Using Proposition \ref{prop:product}, this conjecture is equivalent to showing that, whenever $\beta \in \Br^{+}_{n}$ is a braid such that every braid generator of $\Br^{+}_{n}$ appears in $\beta$, then $\per(\BS(\beta)) = \emptyset$ if and only if the braid closure of $\beta$ consists of a single component. 

We also note that Conjecture \ref{conj:conj-in-bs-case} can be further reformulated as a conjecture concerning deep loci of augmentation varieties of Legendrian rainbow closures of positive braids, as these are isomorphic to double Bott-Samelson varieties, see  \cite{GSW20}.

\subsection{Positroid varieties} \label{subsec:positroid}

Apart from double Bott-Samelson varieties, we note that there is another class of braid varieties for which the morphism $T \to \Aut(A)$ is surjective. These are \emph{open positroid varieties} in the Grassmannian $\Gr(k,n)$, which in fact stratify the Grassmannian~\cite{KLS13}. The simplest description of an open positroid variety is to think of a point of $\Gr(k,n)$ as a $k \times n$ matrix and fix the rank of the submatrices in columns $\{ i, i+1, \ldots, j \}$ for each $(i,j)$,  wrapping around modulo $n$ if $j<i$.
Positroid varieties are both braid varieties for certain $n$-stranded braids and also torus bundles over braid varieties for certain $k$-stranded braids (see~\cite[Theorem 1.3]{CGGS_positroid}), and they have cluster structures which can be described either using the cluster structures on braid varieties, or using Postnikov's theory of plabic graphs~\cite{GL_positroid, P06}. Connections between these approaches were established in \cite{CLSBW23}. Also, for every open positroid variety $\Pi^{\circ}$ in $\Gr(k,n)$, there is a Richardson variety (often, many of them) $R(u,w)$ in the flag variety such that $R(u,w)$ projects isomorphically to $\Pi^{\circ}$.
The permutation $u^{-1} w$, together with some extra data, is called the decorated permutation of $\Pi^{\circ}$. Decorated permutations are in bijection with another combinatorial object called Grassman necklaces, which we will use below; we refer to~\cite{KLS13} for details.

We now cite results of Galashin and Lam concerning torus actions on positroids. 
Let $T$ be the $(n-1)$-dimensional torus of automorphisms of $\Gr(k,n)$, which acts by rescaling coordinates in $\CC^n$.
$T$ acts on $\Pi^{\circ}$ by cluster automorphisms, and this gives a surjection $T \to \Aut(\Pi^{\circ})$;
see \cite[Comment (5) on Conjecture 8.8]{GL22}.
We will now discuss the kernel of the map $T \to \Aut(\Pi^{\circ})$, and when the action of $\Aut(\Pi^{\circ})$ on $\Pi^{\circ}$ is free. 
See~\cite[Proposition 1.6]{GL_positroid} for a closely related discussion. Before we introduce material which is specific to positroids, we discuss stabilizers for the $T$-action on $\Gr(k,n)$.

Let $x$ be a point of the Grassmannian $\Gr(k,n)$. The theory of matroids assigns a matroid of rank $k$ on the ground set $[n]$ to $x$, which has a decomposition $[n] = \bigsqcup E_p$ into connected components.
We give two descriptions of these connected components.
First, for $a$ and $b$ distinct elements of $[n]$, the elements $a$ and $b$ are in the same block $E_p$ if and only if there is some $k-1$ element set $S$ of $[n] \setminus \{ a,b \}$ such that the Pl\"ucker coordinates $\Delta_{S \cup \{ a \}}(x)$ and $\Delta_{S \cup \{ b \}}(x)$ are both nonzero.
Second, consider a point $x$ of the Grassmannian $\Gr(k,n)$ as the row span of a $k \times n$ matrix, with columns $v_1$, $v_2$, \dots, $v_n$.
Then the connected components are the finest set partition of $[n]$ such that there is a vector space decomposition $\CC^k = \bigoplus V_j$, with $\{ v_i : i \in E_j \} \subseteq V_j$.
The following lemma is well-known to those who think about Grasssmannians and matroids, but we give a short proof to remove any mystery from it:
\begin{lemma} \label{lem:stabilizer-point-positroid}
    With the above notation, the stabilizer of $x$ in $T$ is those $(t_1, t_2, \ldots, t_n)$ which obey $t_a=t_b$ whenever $a$ and $b$ are in the same connected component of the matroid.
\end{lemma}

\begin{proof}
    First, let $(t_1, t_2, \ldots, t_n)$ stabilize $x$ and let $a$ and $b$ be in the same connected component of the matroid of $x$.
    Then there are nonzero Pl\"ucker coordinates $\Delta_{S \cup \{ a \}}(x)$ and $\Delta_{S \cup \{ b \}}(x)$, so $h:=\Delta_{S \cup \{ a \}}/\Delta_{S \cup \{ b \}}$ is a rational function on $\Gr(k,n)$ which is well-defined and nonzero at $x$. We have $h((t_1, t_2, \ldots, t_n)\cdot x) = (t_a/t_b) h(x)$ so we deduce that $t_a=t_b$.

    Now, let $(t_1, t_2, \ldots, t_n)$ obey the condition on connected components, and take a vector space decomposition of  $\CC^k$ into  $\bigoplus V_j$ as above. The action of $(t_1, \ldots, t_n)$ rescales the vector $v_j \in \CC^k$ by $t_j$. This is the same as the action of $\CC^k$ induced by the $k \times k$ diagonal matrix which rescales the $V_j$ summand by the common value of $t_a$ for all $a \in E_j$. Thus, the action of $T$ on $x$ corresponds to left-multiplying our $k \times n$ matrix by an invertible matrix, and thus gives the same point of $\Gr(k,n)$.
\end{proof}

We now relate this material to postiroid varieties.
Fix a decorated permutation $f : [n] \to [n]$ and let $\Pi^{\circ} = \Pi^{\circ}_{f}$ be the corresponding open positroid variety.
Let $C_1$, $C_2$, \dots, $C_r$ be the orbits of $f$ on $[n]$, so $\{ C_1, C_2, \ldots, C_r \}$ is a set partition of $[n]$.
Let $\{ D_1, D_2, \ldots, D_s \}$ be the finest non-crossing partition of $[n]$ which coarsens $\{ C_1, C_2, \ldots, C_r \}$.
If $x$ is a generic point of $\Pi^{\circ}$, then the $D_i$ are the connected components of the matroid of $x$, see~\cite[Corollary 7.9]{ARW}.
So the kernel of $T \to \Aut(\Pi^{\circ})$ has dimension $s-1$, where $s$ is the number of blocks $\{ D_1, D_2, \ldots, D_s \}$, and thus the dimension of $\Aut(\Pi^{\circ})$ is $n-s$. 

We now attack the problem of understanding the stabilizer of some point $x \in \Pi^{\circ}$ which is not generic. Let $E_1$, $E_2$, \dots, $E_t$ be the connected components of the matroid of some $x \in \Pi^{\circ}$.

\begin{proposition} \label{prop:MatroidStabilizerCombinatorics}
    With the above notation, the set partition $\{ E_1, E_2, \ldots, E_t \}$ refines the set partition $\{ D_1, D_2, \ldots, D_s \}$ and coarsens  the set partition $\{ C_1, C_2, \ldots, C_r \}$.
\end{proposition}

\begin{proof}
    We first show that $\{ E_1, E_2, \ldots, E_t \}$ refines $\{ D_1, D_2, \ldots, D_s \}$. Let $v_1$, $v_2$, \dots, $v_n$ be vectors in $\CC^k$ representing the point $x$ as above. 
    For a generic point of $\Pi^{\circ}$, we have a decomposition $\CC^k = \bigoplus V_j$, with $\{ v_i : i \in D_j \} \subseteq V_j$ so, for any point of $\Pi^{\circ}$, such a decomposition still exists. So the $E_i$, which index the finest such decomposition of $\CC^k$, refine the $D_j$.

    We now show that $\{ E_1, E_2, \ldots, E_t \}$ coarsens $\{ C_1, C_2, \ldots, C_s \}$.  
    Let $I_1$, $I_2$, \dots, $I_n$ be the Grassmann necklace corresponding to $f$; for each index $i$, we have $I_i = (I_i \cap I_{i+1}) \cup \{ i \}$ and $I_{i+1} = (I_i \cap I_{i+1}) \cup \{ f(i) \}$ (here subscripts are periodic modulo $n$). The ratios $\Delta_{I_{i+1}}/\Delta_{I_i}$ are the frozen variables of the cluster structure on $\Pi^{\circ}$, so they are nonzero at $x$, so we see that $i$ and $f(i)$ are in the connected component of the matroid of $x$. Thus, $\{ E_1, E_2, \ldots, E_t \}$ coarsens $\{ C_1, C_2, \ldots, C_s \}$.  
\end{proof}

We thus deduce the following analogue of Corollary \ref{cor:free-action-doublebs}. 

\begin{corollary}
    The action of $\Aut(\Pi^{\circ})$ on $\Pi^{\circ}$ is free if and only the restriction of $f$ to each $D_j$ is a single cycle, in other words, if and only if $\{ C_{\bullet} \}$ and $\{ D_{\bullet} \}$ are the same set partition of $[n]$.
\end{corollary}

\begin{proof}[Proof sketch]
First, suppose that the restriction of $f$ to each $D_j$ is a single cycle. So the set partitions $\{ C_{\bullet} \}$ and $\{ D_{\bullet} \}$ are the same, and all the points of $\Pi^{\circ}$ have the same stabilizer.

Now, suppose that the set partition $\{ C_{\bullet} \}$ properly refines $\{ D_{\bullet} \}$. For each $j$, choose a point $x_j$ in the Grassmannian $\Gr(k_j, \CC^{C_j})$ realizing the decorated permutation $f|_{C_j}$. Then the point $x := x_1 \oplus x_2 \oplus \cdots \oplus x_r$ of $\Gr(k,n)$ has decorated permutation $f$ and its matroid has connected components $C_1$, $C_2$, \dots, $C_r$, so $x$ lies in $\Pi^{\circ}$, and $x$ is stabilized by the torus corresponding to $\{ C_{\bullet} \}$, which is larger than the torus corresponding to $\{ D_{\bullet} \}$.
\end{proof}

Thus, a hypothetical mysterious point of $\Pi^{\circ}$ would correspond to a point $x \in \Pi^{\circ}$ whose matroid has connected components $\{ D_{\bullet} \}$, but which does not lie in any cluster torus. 
The authors have found this useful in working out small examples, although for our proofs we find weaves to be more effective in constructing clusters than techniques for creating clusters in positroid varieties, such as plabic graphs. We close with an example in which there are no mysterious points, but there are points with trivial stabilizer that do not lie in any Pl\"ucker cluster torus. 

\begin{example} \label{eg:NonPlabic}
    Let $f$ be the decorated permutation $(135)(264)$ in $S_6$. This corresponds to the locus in $\Gr(3,6)$ where the triples $\{ v_1, v_2, v_3 \}$, $\{ v_3, v_4, v_5 \}$ and $\{ v_1, v_5, v_6 \}$ are coplanar, but the quadruples $\{ v_i, v_{i+1}, v_{i+2}, v_{i+3} \}$ span $\CC^3$ and the pairs $\{ v_i, v_{i+1} \}$ are linearly independent for all $1 \leq i \leq 6$, where subscripts are periodic modulo $6$.

    The orbits of $f$ are $C_1 = \{ 1,3,5 \}$ and $C_2 = \{ 2,4,6 \}$, so we have a single set $D_1 = \{ 1,2,3,4,5,6 \}$. 
    Points of $\Pi^{\circ}$ not lying in the stabilizer locus correspond to matroids with this decorated permutation and a single connected component. 
By brute force, one can check that the list of such matroids is:
    \begin{enumerate}
    \item The matroid corresponding to $v_1 = e_1$, $v_2= e_1+e_2$, $v_3 = e_2$, $v_4 = e_2+e_3$, $v_5=e_3$, $v_6 = e_1+e_3$. This is the generic case.
    \item The matroid where $v_1$, $v_3$ and $v_5$ are proportional nonzero vectors, and the other vectors are chosen generically.
        \item The matroid corresponding to $v_1 = e_1$, $v_2= e_1+e_2$, $v_3 = e_2$, $v_4 = e_2+e_3$, $v_5=e_3$, $v_6 = e_1-e_3$. This is like the previous case, except that $v_2$, $v_4$, $v_6$ are coplanar.
    \end{enumerate}

    The cluster structure is of type $A_1$; the cluster variables are the Pl\"ucker coordinate $\Delta_{246}$ and a non-Pl\"ucker cluster variable 
    $Y:=(\Delta_{124} \Delta_{346} \Delta_{256} + \Delta_{234} \Delta_{456} \Delta_{126}) \Delta_{234}^{-1}$. (There are two standard cluster structures  on $\Pi^{\circ}$ which differ by a quasi-automorphism, source-labeling and target-labeling~\cite{FSB}. We have computed $Y$ using target-labeling conventions; if we used source-labeling, this would multiply $Y$ by a monomial in the frozen variables, which would not effect any of the statements that we make about $Y$ in the following.)
    On the whole Grassmannian, $Y$ vanishes when the planes $\text{Span}(v_1, v_2)$, $\text{Span}(v_3, v_4)$ and $\text{Span}(v_5, v_6)$ have a nontrivial intersection. 

    The cluster torus where $\Delta_{246} \neq 0$ covers the first two matroids above. The cluster torus $Y \neq 0$ covers the first and third matroids. We thus see that there are no mysterious points in this example. However, it was important to use non-Pl\"ucker cluster variables in this computation, which suggests that techniques based on plabic graphs will find it difficult to prove the absence of mysterious points.
\end{example}

\section{Deep loci for three-stranded braids}
\label{sec:empty-deep-loci}

\subsection{Two-stranded braids} Here, we show the following result.

\begin{proposition}\label{prop:2-strands}
Let $\beta = \sigma^{\ell} \in \mathrm{Br}^{+}_{2}$.
\begin{itemize}
\item[(a)] If $\ell$ is even, then $\per(X(\beta))$ is empty.
\item[(b)] If $\ell$ is odd, then $\per(X(\beta))$ consists of a single point.
\end{itemize}
\end{proposition}
\begin{proof}
We claim that if $z = (z_1, \dots, z_{\ell}) \in X(\beta)$ is $z \neq (0, 0, \dots, 0)$, then there exists a cluster torus containing $z$.

We do this by induction on $\ell$. The case $\ell = 2$ is clear, since $X(\sigma^{2}) = \{(z, z^{-1}) \mid z \in \C^{\times}\}$. Now let $k \in \{1, \dots, \ell\}$ be maximal so that $z_{k} \neq 0$. We remark that $k > 1$: otherwise, $z = (z_1, 0, 0, \dots, 0)$ and thus
\begin{equation}\label{eq: power 2 strand}
\delta B_{\sigma^{\ell}}(z) = \delta B_{\sigma}(z_1)B_{\sigma^{\ell -1 }}(0) = \left(\begin{matrix} 0 & 1 \\ 1 & 0\end{matrix}\right)\left(\begin{matrix}z_1 & -1 \\ 1 & 0\end{matrix}\right)\left(\begin{matrix} 0 & -1 \\ 1 & 0\end{matrix}\right)^{\ell -1}
\end{equation}
and the matrix \eqref{eq: power 2 strand} has $\pm 1$ (if $\ell$ is even) or $\pm z_1$ (if $\ell$ is odd) in its $(2,1)$-entry, contradicting the fact that $z \in X(\beta)$.

So $k > 1$ and thus we can draw a trivalent vertex with $z_k$ being the label of its right arm. If $k < \ell$ then, thanks to Lemma \ref{lem:easy3valent} we have that the output in this weave is
\[
(z_1, \dots, z_{k-2}, z_{k-1} - z_{k}^{-1}, -z_k, 0, \dots, 0)
\]
and we have a point in $X(\sigma^{\ell - 1})$ with a nonzero coordinate, so we can apply induction. We arrive to the same conclusion if there exists $i \leq k-2$ such that $z_i \neq 0$, or if $z_{k-1} \neq z_{k}^{-1}$.

So assume that $k = \ell$, $z_{k-1} = z_{k}^{-1}$ and $z_{i} = 0$ for every $i \leq k-2$. In this case, we can draw a trivalent vertex whose right arm is labeled by $z_{k-1}$, and we get as an output
\[
(0, \dots, 0, -z_{k-1}^{-1}, 0)
\]
which is again a point with a nonzero coordinate, and we can again apply induction.

It remains to see that $(0, \dots, 0) \in X(\beta)$ if and only if $\ell$ is odd, which is checked directly. 
\end{proof}

\begin{corollary}
 The cluster varieties $X(\sigma^{\ell})$ have no mysterious points.
\end{corollary}
\begin{proof}
Note that a lift of the maximal element of $S_2$ to $\Br^{+}_{2}$ is precisely $\sigma$. Thus, $X(\sigma^{\ell}) \cong \BS(\sigma^{\ell - 1})$, and the torus of cluster automorphisms is the one-dimensional maximal torus in $\mathrm{PGL}(2)$. It remains to see that this torus acts freely on $z = (z_1, \dots, z_{\ell}) \in X(\beta)$ if and only if at least one coordinate $z_i$ is nonzero, which is clear from \eqref{eq: explicit action}. Thus,
\[
\per(X(\sigma^{\ell})) = \cS(X(\sigma^{\ell})) = \begin{cases} \emptyset & \text{if} \; \ell \; \text{is even}, \\
\{(0, \dots, 0)\} & \text{if} \; \ell \; \text{is odd.}\end{cases}
\]
\end{proof}

We remark that $X(\sigma^{\ell}) = \BS(\sigma^{\ell - 1})$ is a cluster variety of type $A_{\ell - 2}$ with one frozen:
\[
Q_{\sigma^{\ell - 1}} = 1 \to 2 \to \cdots \to \ell - 2 \to {\color{blue} \ell - 1}
\]
so from Corollary \ref{cor:independence-of-coefficients} we obtain the following result.

\begin{corollary}\label{cor:type-a}
    Let $\A$ be a cluster variety of type $A_{n}$. Then, $\A$ has no mysterious points. Moreover, if $\A$ has really full rank and $m$ frozen vertices then
    \[
    \per(\A) \cong \begin{cases} \emptyset & \text{if} \; n \; \text{is even,} \\
    (\C^{\times})^{m-1} & \text{if} \; n \; \text{is odd}. \end{cases}
    \]
\end{corollary}

Note that a cluster variety of type $A_{2n}$ with no frozens is already really full rank. On the other hand, a cluster variety of type $A_{2n+1}$ without frozens is not full rank. 

\begin{remark} \label{rem:G2Cluster}
    The maximal open positroid variety in $\Gr(2, \ell+1)$ is a torus bundle over $X(\sigma^{\ell})$, and this map is a cluster quasi-homomorphism. We could also describe the results of this section in terms of positroids. 
    Viewing a point of $\Gr(2, \ell+1)$ as an $2 \times (\ell+1)$ matrix, with columns $v_1$, $v_2$, \dots, $v_{\ell+1}$, this positroid variety is the locus where $v_i$ and $v_{i+1}$ are linearly independent, including that $v_1$ and $v_{\ell+1}$ are linearly independent.
    The corresponding decorated permutation is $i \mapsto i+2 \bmod \ell+1$. If $\ell$ is even, this is a single $\ell+1$ cycle so the stabilizer locus is empty, and our result shows that the deep locus is also empty. If $\ell$ is odd, then this is two $\tfrac{\ell+1}{2}$ cycles, and the stabilizer locus has one component, where $v_1$, $v_3$, $v_5$, \dots, $v_{\ell}$ are proportional, as are $v_2$, $v_4$, $v_6$, \dots, $v_{\ell+1}$; our result shows that this is also the deep locus.
\end{remark}

\subsection{Main result on three-stranded braids} For $a, b \geq 0$ consider the braid
\[
\beta(a,b) = \sigma_1^{a}(\sigma_2\sigma_1)^{b} \in \mathrm{Br}_{3}^{+},
\]
and let us denote $X(a,b) := X(\beta(a,b))$ the corresponding braid variety. For example, the maximal positroid stratum in $\Gr(3,n)$ is (up to a torus factor) $X(1, n-2)$.

\begin{remark}
If $b > 1$, then some minimal lift of the longest element $w_0 \in S_{n}$, i.e. either $\sigma_1\sigma_2\sigma_1$ or $\sigma_2\sigma_1\sigma_2$, is a consecutive subword of $\beta$, so that $X(a,b)$ is a double Bott-Samelson variety. The same is true if $a, b \geq 1$. On the other hand, if $b = 0$ then $\beta(a,0)$ belongs to the parabolic subgroup $\Br_{2}$, so we are really in the case of the previous subsection. Thus, we see that we are really working with double Bott-Samelson varieties. 
\end{remark}

\begin{theorem}\label{thm:main}
Let $z = (z_1, \dots, z_{a + 2b}) \in X(a,b)$. If $T$ acts freely on $z$, then there exists a seed $t \in \seeds{\C[X(a,b)]}$ such that $z \in T(t)$. 
\end{theorem}

\begin{corollary} \label{cor:X(a,b)-no-mysterious-points}
The cluster varieties $X(a,b)$ have no mysterious points.
\end{corollary}

Thanks to Corollary \ref{cor:reduction-to-full-rank}, the statement of Corollary~\ref{cor:X(a,b)-no-mysterious-points} remains correct if we vary the frozen part of quivers defining cluster structures on $X(a,b)$, and so we get the following.

\begin{corollary}
The cluster varieties for quivers whose principal part is mutation equivalent to quivers $Q_{k,l}$ of the form

\begin{center}
\begin{tikzpicture}
\node at (0,0) {$\circ$};
\draw[->] (0.2,0)--(0.8,0);
\node at (1,0) {$\circ$};
\draw [->] (1.2, 0) -- (1.6,0);
\node at (2,0) {$\ldots$};
\draw [->] (2.4, 0) -- (2.6,0);
\node at (3,0) {$\circ$};
\draw[->] (3.4, 0) -- (3.8, 0);
\node at (4,0) {$\circ$};
\draw[->] (4.2, 0) -- (4.6, 0);
\node at (5,0) {$\ldots$};
\draw[->] (5.4, 0) -- (5.8, 0);
\node at (6,0) {$\circ$};

\node at (0,-1) {$\circ$};
\draw[->] (0.2,-1)--(0.8,-1);
\node at (1,-1) {$\circ$};
\draw [->] (1.2, -1) -- (1.6,-1);
\node at (2,-1) {$\ldots$};
\draw [->] (2.4, -1) -- (2.8,-1);
\node at (3,-1) {$\circ$};

\draw[->] (0, -0.9) -- (0, -0.1);
\draw[->] (1, -0.9) -- (1, -0.1);
\draw[->] (3, -0.8) -- (3, -0.1);

\draw[->] (0.8, -0.2)--(0.2, -0.8);
\draw[->] (1.8, -0.2)--(1.2, -0.8);
\draw[->] (2.8, -0.2)--(2.2, -0.8);

\draw[decoration={brace,raise=5pt},decorate]
  (-0.2, 0) -- node[above=5pt] {$k$} (3.2, 0);
\draw[decoration={brace,raise=5pt},decorate]
  (3.8, 0) -- node[above=5pt] {$l$} (6.2, 0);
\end{tikzpicture}
\end{center}

for arbitrary $k, l \geq 0$ have no mysterious points. 
\end{corollary}

For given $k, l$, the quiver defining the cluster structure on the variety $X(l+1, k+2)$ has the above form, with two frozen vertices.

\begin{corollary} \label{cor:main:Section5}
The following braid varieties have empty deep loci and are unions of finitely many cluster charts:

\begin{enumerate}
    \item Maximal open positroid cells in $\Gr(2, n)$ for $n$ odd;
    \item Maximal open positroid cells in $\Gr(3, n)$ for $n$ not divisible by $3$;
    \item Cluster varieties of types $A_{2n}, E_{6}$ and $E_{8}$ and really full rank. The corresponding braids were given in \cite{casals2022lagrangian}.
\end{enumerate}
\end{corollary}

In items (1) and (2), the cluster structure in question goes back to Scott \cite{Scott}. The principal part of the quivers in (1) corresponds to $k= 0, l = n-3$. The principal part of the quivers in (2) corresponds to $k= n - 4, l = 0$. 

The proof of Theorem \ref{thm:main} is by induction on $|\beta(a,b)| = a + 2b$. First, we make certain reductions.

\begin{enumerate}
    \item We may assume that $b > 3$. Indeed, the cluster variety $X(a,0) \cong X(a,1)$ is of type $A_{a-2}$ with one frozen, $X(a,2)$ is of type $A_{a-1}$ with two frozens, and $X(a,3)$ is of type $A_{a+1}$ with two frozens. All these cases have already been covered by Proposition \ref{prop:2-strands}. 
    \item We may assume $a \geq 1$. This is because, thanks to cyclic rotation, \cite[Section 5.4]{CGGLSS}, $X(0,b)$ is quasi-cluster equivalent to the braid variety for the braid $\sigma_{2}(\sigma_{2}\sigma_1)^{b-1}\sigma_{2}$, which is isomorphic to the braid variety $X(2, b-1)$. 
\end{enumerate}

With these reductions at hand, let us discuss the strategy towards the proof of Theorem \ref{thm:main}. We will do an inductive argument using Corollary \ref{cor:necessary-for-induction}. First, note that if the $T$-action on $z$ is free, then at least one of the coordinates $z_1, \dots, z_{\ell}$ of $z$ must be nonzero. We will see that we can always find such a coordinate that will index the right arm of a trivalent vertex in a Demazure weave from $\beta(a,b)$ to a braid word $\beta'$ cyclically equivalent to $\beta(a', b')$ with $a' + 2b' < a + 2b$. We have the birational map $X(a,b) \dashrightarrow X(\beta')$, and $z$ is in the domain of this map so, by Corollary \ref{cor:necessary-for-induction} and an inductive argument, it is enough to see that the $T$-action on the image of $z$ in $X(\beta')$ is free. This is the technical heart of the proof: the torus $T$ is \emph{not} going to act freely on the image of $z$ for \emph{all} weaves, but we will see that we can always find one that works. We will do this, essentially, by \lq\lq trial and error\rq\rq. If one weave does not work, we will find a replacement that does work. In one sentence, the replacement is a \lq\lq translation of the original weave to the left\rq\rq. So part of our strategy involves reading the braid word $\beta(a,b)$ from right to left. 

\subsection{Weaving} We start implementing the strategy above. In this section, we will see that if $z \in X(a,b)$ has a free $T$-stabilizer, then it is possible to draw a weave from $X(a,b)$ to $X(a',b')$, with $a'+2b' < a+2b$, so that $z$ is in the domain of the corresponding map $X(a,b) \dashrightarrow X(a',b')$.\\

Consider the word $\beta(a,b)$. Note that every letter in $\beta(a,b)$ except for the first, $(a+1)$-st and $(a+2)$-nd can potentially become the right incoming arm of a $3$-valent vertex. The next result, then, ensures that we are always able to start building a weave.

\begin{lemma}
    Let $z = (z_1, \dots, z_{a}, z_{a+1}, z_{a+2}, \dots, z_{a+2b}) \in X(a,b)$. If the $T$-action on $z$ is free, then one of $\{z_1, \dots, z_{a+2b}\} \setminus \{z_1, z_{a+1}, z_{a+2}\}$ is nonzero. 
\end{lemma}
\begin{proof}
Assume on the contrary that $z_{2}, \dots, z_{a}, z_{a+3}, \dots, z_{a+2b}$ are zero. Let us assume first that $a$ is odd. Then the matrix $B_{\sigma_1^{a}\sigma_2\sigma_1}(z_1, 0, \dots, 0, z_{a+1}, z_{a+2})$ is, up to signs:
\[
\left(\begin{matrix} z_1z_{a+2} - z_{a+1} & -z_1 & 1 \\ z_{a+2} & -1 & 0 \\ 1 & 0 & 0  \end{matrix}\right) 
\]
Multiplying on the right by $B_{(\sigma_2\sigma_1)^{b-1}}(0, 0, \dots, 0)$ will, up to a sign, simply permute the columns. Thus we see that the end result will be a matrix representing the antistandard flag if and only if $z_1 = z_{a+1} = z_{a+2} = 0$, which contradicts the assumption that $T$ acts freely on $z$. 

Now assume that $a$ is even. Then  the matrix $B_{\sigma_1^{a}\sigma_2\sigma_1}(z_1, 0, \dots, 0, z_{a+1}, z_{a+2})$ is, up to signs:
\[
\left(\begin{matrix} -z_1z_{a+1} - z_{a+2} & 1 & z_1 \\ -z_{a+1} & 0 & 1 \\ 1 & 0 & 0  \end{matrix}\right) 
\]
and we arrive at the same conclusion as the odd $a$ case. 
\end{proof}

In order to apply induction we want to make sure that, after we build a trivalent vertex on $\beta(a,b)$, we end up with a braid equivalent, up to cyclic rotation, to a braid of the form $\beta(a',b')$ with $a' + 2b' < a+2b$. It turns out that this is \emph{not} always the case but, as we will see in Lemma \ref{lem:notbadcases} below, we can always avoid the bad cases.

Let us start with the easiest case. If we have a trivalent vertex whose right arm is $z_{i}$ with $1 < i \leq a$, then the resulting braid is simply $\beta(a-1, b)$. 

Now assume that we build the weave in Figure \ref{fig:aplus2weave}.

\begin{figure}[h!]
    \centering
    \includegraphics[scale=1.5]{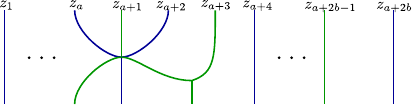}
    \caption{We can form this weave if $z_{a+3} \neq 0$. Note also that we can form a similar weave when $z_{a+3+3k} \neq 0$ for $k \geq 0$.}
    \label{fig:aplus2weave}
\end{figure}

Note that the braid word on the bottom is $\beta(a-1, b)$. Note also that we can build a similar trivalent vertex whose upper strand is the $a+6, a+9, \dots$ strand, and the end result is, again, $\beta(a-1, b)$. Thus, if $z_{a+3k} \neq 0$, we can build a weave from $\beta(a,b)$ to $\beta(a-1,b)$. 

Let us now examine the weave in Figure \ref{fig:aplus4weave}, that we can build provided $z_{a+4} \neq 0$. Note that the braid word in the bottom is $\beta(a+1, b-1)$. Note also that we can build a similar trivalent vertex whose upper strand is the $a+7, a+10, \dots$ strand end the end result is, again $\beta(a+1, b-1)$. Thus, if $z_{a+3k+1} \neq 0$, we can build a weave from $\beta(a,b)$ to $\beta(a+1,b-1)$. 

\begin{figure}[h!]
    \centering
    \includegraphics[scale=1.5]{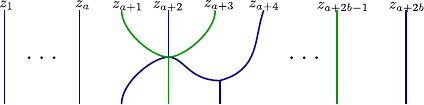}
    \caption{We can form this weave if $z_{a+4} \neq 0$. We can form a similar weave if $z_{a+4+3k} \neq 0$ for $k \geq 0$.}
    \label{fig:aplus4weave}
\end{figure}

On the opposite extreme, let us examine the weave in Figure \ref{fig:aplus2bweave}. The braid word in the bottom is $\sigma_1^{a}(\sigma_2\sigma_1)^{b-1}\sigma_2$, which is cyclically equivalent to $\beta(a+1, b-1)$. Similarly, to the above, we can conclude that if $z_{a+2b-3k} \neq 0$, then we can build a weave from $\beta(a,b)$ to a braid which is cyclically equivalent to $\beta(a+1, b-1)$.
\begin{figure}[h!]
    \centering
    \includegraphics[scale=1.5]{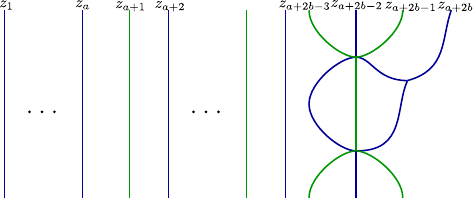}
    \caption{We can form this weave if $z_{a+2b} \neq 0$. We can form a similar weave whenever $z_{a+2b-3k} \neq 0$, $k \geq 0$, $2b > 3k$.}
    \label{fig:aplus2bweave}
\end{figure}

Finally, observe the weave in Figure \ref{fig:aplusbminus2weave}. The braid word in the bottom is $\sigma_1^{a}(\sigma_2\sigma_1)^{b-2}\sigma_2\sigma_2\sigma_1$. Using that $a \geq 1$, this is cyclically equivalent to $\sigma_1^{a-1}(\sigma_2\sigma_1)^{b-2}\sigma_2\sigma_2\sigma_1\sigma_2$, that is equivalent to $\beta(a-1, b)$. Thus, if $z_{a+2b-2-3k} \neq 0$, we can build a weave from $\beta(a, b)$ to a word cyclically equivalent to $\beta(a-1, b)$.

\begin{figure}[h!]
    \centering
    \includegraphics[scale=1.5]{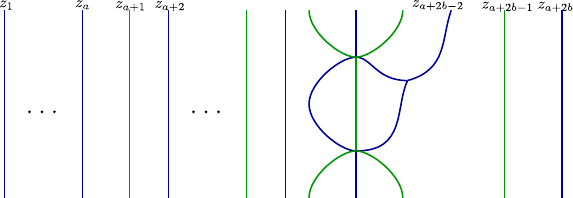}
    \caption{We can form this weave whenever $z_{a+2b-2} \neq 0$, and we can form a similar weave when $z_{a+2b-2-3k} \neq 0$, with $a < a+2b-2-3k \leq a+2b-2$.}
    \label{fig:aplusbminus2weave}
\end{figure}

The only $z$-coordinates where we do \emph{not} know whether them being nonzero implies that we can build a weave from $\beta(a,b)$ to a braid word cyclically equivalent to $\beta(a', b')$ are those of the form $z_{i}$ where:
\begin{itemize}
    \item $i = a + 5 + 3k$, $k \geq 0$, or
    \item $i = a + 2b - 1 - 3k$, $k \geq 0$.
\end{itemize}

Thus, there will be some coordinates whose nonvanishing does \emph{not} imply that we can build a weave from $\beta(a,b)$ to $\beta(a', b')$ exactly when $a + 5 \equiv a + 2b - 1 \mod 3$, that is, when $b \equiv 0 \mod 3$. In this case, the non-vanishing of $z_{a+5}, z_{a+8}, \cdots, z_{a+2b-1}$ does not help us build a weave to $\beta(a', b')$ with $a + 2b < a' + 2b'$. The next lemma says that we can avoid these cases.

\begin{lemma}\label{lem:notbadcases}
    Assume that $b \equiv 0 \mod 3$. If $T$ acts freely on $z = (z_1, \dots, z_{a+2b}) \in X(a,b)$ then one of $\{z_1, \dots, z_{a+2b}\} \setminus \{z_1, z_{a+1}, z_{a+2}, z_{a+5}, z_{a+8}, \dots, z_{a+2b-1}\}$ is nonzero. 
\end{lemma}
\begin{proof}
    Assume the contrary and let us reach a contradiction. We will work by induction on $b$, and separate the cases when $a$ is even or odd. Recall that we always assume $a \geq 1$ and $b > 3$, hence the permutation matrix for $\delta(\beta)^{-1}$ is $\left(\begin{matrix} 0 & 0 & 1 \\ 0 & 1 & 0 \\ 1 & 0 &0  \end{matrix}\right)$. 

    Let us assume first that $a$ is odd. Since $B_{\sigma_1^{2s}}(0, \dots, 0)$ is a diagonal matrix with $\pm 1$ on the diagonal, we may assume that $a = 1$. Note that:
    \[
    B_{\sigma_1\sigma_2\sigma_1}(z_1, z_{a+1}, z_{a+2}) = \left(\begin{matrix}z_{1}z_{a+2}-z_{a+1} & -z_{1} & 1 \\ z_{a+2} & -1 & 0 \\ 1 & 0 & 0 \end{matrix}\right).
    \]
Now we need to multiply by $B_{212}(0,0,z_{a+5})$ which yields:
\[
B_{121212}(z_1, z_{a+1}, z_{a+2}, 0, 0, z_{a+5}) = \left(\begin{matrix} 1 & -z_{1}z_{a+2}z_{a+5} + z_{a+1}z_{a+5} + z_{1} & z_{1}z_{a+2} - z_{a+1} \\ 0 & -z_{a+2}z_{a+5} + 1 & z_{a+2} \\ 0 & -z_{a+5} & 1  \end{matrix}\right).
\]
And, in order to get $B_{\beta(a,b)}(z)$, we have to multiply by matrices of the form of $B_{121212}(0, 0, w, 0, 0, w')$ and then finally multiply by $B_{1}(0)$. Note, however, that 
\begin{equation}\label{eq:121212}
B_{121212}(0,0,w,0,0,w) = \left(\begin{matrix} 1 & 0 & 0 \\ 0 & -ww' + 1 & w \\ 0 & -w' & 1 \end{matrix}\right).
\end{equation}

So the matrix $B_{\sigma_1(\sigma_2\sigma_1)^{b-1}\sigma_2}(z_1, z_{a+1}, z_{a+2}, 0, 0, z_{a+5}, 0, 0, z_{a+8}, \dots, z_{a+2b-4}, 0, 0, z_{a+2b -1})$ has a $(1,0,0)^T$ as its first column. Multiplying by $B_{1}(0)$ will, up to a sign, interchange the first and second columns. So $B_{\beta(a+2b)}(z_1, z_{a+1}, z_{a+2}, 0, 0, z_{a+5}, 0, 0, z_{a+8}, \dots, z_{a+2b-1}, 0)$ has $(-1,0,0)^T$ in its second column. The multiplication on the left by $\delta(\beta)^{-1}$ gives a matrix with $(0,0,-1)^T$ in its second column. It cannot be upper-triangular, and so we conclude that the point $z$ cannot be in the braid variety.

Let us now assume that $a$ is even, so similarly to the odd case we can assume $a = 2$. First we compute:
\[
B_{1121}(z_1, 0, z_{a+1}, z_{a+2}) = \left(\begin{matrix} -z_{1}z_{a+1} - z_{a+2} & 1 & z_{1} \\ -z_{a+1} & 0 & 1 \\ 1 & 0 & 0 \end{matrix}\right)
\]
and
\[
B_{1121212}(z_1, 0, z_{a+1}, z_{a+2}, 0, 0, z_{a+5}) = \left(\begin{matrix} z_{1} & z_{1}z_{a+1}z_{a+5} + z_{a+2}z_{a+5} - 1 & -z_{1}z_{a+1}-z_{a+2} \\ 1 & z_{a+2}z_{a+5} & -z_{a+1} \\ 0 & -z_{a+5} & 1\end{matrix}\right).
\]
Just as in the previous case, multiplying by a power of \eqref{eq:121212} followed by $B_{1}(0)$ will yield a matrix whose second column is $(-z_1, -1, 0)^T$, and further multiplication on the left by $\delta(\beta)^{-1}$ will yield a matrix with second column $(0, -1, -z_1)^T$. Since $z \in X(a,b)$, this forces $z_1 = 0$. 

Now we claim that the matrix:
\[
B_{\sigma_1^{2}(\sigma_1\sigma_2)^{b}}(0,0,z_{a+1}, z_{a+2}, 0, 0, z_{a+5}, 0, 0, z_{a+8}, \dots, 0, 0, z_{a+2b-1}, 0)
\]
has the form
\[
\left(\begin{matrix} \varepsilon & 0 & \beta \\ \alpha z_{a+1} & - 1 & z_{a+1}\gamma \\ -\alpha & 0 & - \gamma\end{matrix} \right)
\]
for some polynomials $\alpha, \beta, \gamma$. We will prove this claim below, let us now assume the matrix does have such form. The multiplication of such a matrix on the left by $\delta(\beta)^{-1}$ gives a matrix with first column $(-\alpha, \alpha z_{a+1}, \varepsilon)^T$ which has to be an upper-triangular matrix in $\GL_3(\mathbb{C})$ if $z \in X(a,b)$. This implies that $\alpha \neq 0$ and $\alpha z_{a+1} = 0$, so $z_{a+1} = 0$. It follows that the only nonzero $z$'s are $z_{a+2}, z_{a+5}, z_{a+8}, \dots, z_{a+2b-1}$. But these all have, up to a sign, the same $T$-weight. By Remark~\ref{rem:free-action-weights-span-the-weght-lattice}, this means that the action of $T$ on $z$ is not free, and so we arrived to a contradiction. 

We now prove our claim by induction. Note first that:
\[
B_{11212121}(0,0,z_{a+1}, z_{a+2}, 0,0, z_{a+5}, 0) = \left(\begin{matrix} z_{a+1}z_{a+5} - 1 & 0 & -z_{a+2} \\ z_{a+1}z_{a+5} & - 1 & -z_{a+1} \\ -z_{a+5} & 0 & 1 \end{matrix}\right). 
\]
So the result is certainly true when $b = 3$. Since $b \neq 0 \mod 3$, we now need to multiply by matrices of the form:
\[
B_{212121}(0,w, 0, 0, w', 0) = \left(\begin{matrix} -ww' + 1 & 0 & w  \\ 0 & 1 & 0 \\ -w' & 0 & 1 \end{matrix}\right).
\]
Thus, arguing by induction we have to multiply:
\[
\left(\begin{matrix} \varepsilon & 0 & \beta \\ \alpha z_{a+1} & 1 & z_{a+1}\gamma \\ -\alpha & 0 & - \gamma\end{matrix} \right)\left(\begin{matrix} -ww' + 1 & 0 & w  \\ 0 & 1 & 0 \\ -w' & 0 & 1 \end{matrix}\right) = \left(\begin{matrix} \varepsilon(-ww'+1) - w'\beta & 0 & \varepsilon w + \beta \\
z_{a+1}(\alpha(-ww'+1) - \gamma w') & 1 & z_{a+1}(w\alpha + \gamma) \\ -\alpha(-ww'+1)+\gamma w' & 0 & -w\alpha - \gamma \end{matrix}\right)
\]
and we conclude.
\end{proof}

We summarize the results of this section in the following proposition.

\begin{proposition}
Let $z \in X(a,b)$, and assume that the torus $T$ acts freely on $z$. Then, there exist a braid $\beta'$ cyclically equivalent to $\beta(a',b')$ for $a'+2b' < a + 2b$ such that $z$ belongs to the domain of a birational map $X(a,b) \to X(\beta')$ obtained as the composition of maps of the form \eqref{eqn:opening crossings} and isomorphisms given by braid moves. 
\end{proposition}

To be able to apply Corollary \ref{cor:necessary-for-induction}, we need to verify that the image of $z$ in $X(\beta')$ does not have a free $T$-stabilizer. We will achieve this in Section \ref{subsec:completion}. Before, we will study the weights of the components $z_1, \dots, z_{a + 2b}$ of $z$. 

\subsection{Weights} Let us examine the weight of $z_{i}$, $i = 1, \dots, a+2b$. It suffices to note the following easy lemma.

\begin{lemma} \label{lem:weights-in-X(a,b)}
Let $\beta = \beta(a,b)$, and $z = (z_1, \dots, z_{a}, z_{a+1}, \dots, z_{a+2b}) \in X(\beta)$. Then:
\begin{enumerate}
\item $z_1, \dots, z_a$ all have, up to a sign, the same weight.
\item For $i \geq a$,
$\wt(z_i) = -\wt(z_{i+3})$.
\item $\wt(z_a), \wt(z_{a+1}), \wt(z_{a+2})$ are different (even up to a sign), and any two of them generate the (rank 2) weight lattice of $T$.
\item More generally, if $\wt(z_i) \neq \pm \wt(z_j)$, then $\wt(z_{i})$ and $\wt(z_{j})$ already generate the (rank $2$) weight lattice of $T$. If such $z_i, z_j \neq 0$, then $T$ acts freely on $z$.
\end{enumerate}
\end{lemma}

\begin{lemma}\label{lem:notsame}
Consider any of the weaves in Figures \ref{fig:aplus4weave}--\ref{fig:aplusbminus2weave} (or the weaves obtained by shifting these weaves modulo $3$). Let $z_{i}$ be the right incoming arm of the trivalent vertex. If all nonzero variables in the bottom have weight $\pm\wt(z_{i})$, then the same is true for all nonzero variables in the top. In particular, in this case the action of $T$ on $z$ is not free. 
\end{lemma}
\begin{proof}
It is clear that, to the left of the trivalent vertex, every top variable with weight $\neq \pm \wt(z_{i})$ will be zero. For the variables on the right, we use Lemma \ref{lem:hard3valent}, see in particular Figure \ref{fig:hard3valent}. In the context of that Figure, we have $w=z_{i} \neq 0$, and $w^{-1}z_{1}$, $wz_{2} - z_{1} = 0$. Then, $z_1 = z_2 = 0$. All the variables to the right are just multiplied by a multiple of the (nonzero) variable $z_{i}$, from where the result follows. 
\end{proof}

We conclude that, if the action of $T$ on $z$ is free, then there must be at least one nonzero bottom variable with weight different from $\pm \wt(z_{i})$.

\subsection{Completion of the proof} \label{subsec:completion}
We are now ready to complete the proof of Theorem \ref{thm:main}. We go by cases, which are obtained by reading the braid word from right to left.

{\bf Case 1.} Assume first that $z_{a+2b} \neq 0$, and draw a weave like the one in Figure \ref{fig:aplus2bweave}. By Lemma \ref{lem:notsame}, not all the nonzero variables at the bottom can have the same weight as $\pm\wt(z_{a+2b})$. If the $T$-action on the bottom variables is not free, then all the nonzero variables must have weights concentrated in a $1$-dimensional sublattice of the weight lattice of $T$ that, in particular, cannot be the one spanned by $\wt(z_{a+2b})$. This implies that:
\begin{enumerate}
    \item $z_{a+2b-3} - z_{a+2b}^{-1} = 0$ (this is the leftmost outgoing variable of the bottom hexavalent vertex), so $z_{a+2b-3} \neq 0$, and
    \item $z_{a+2b-3-3k} = 0$, in particular, $z_{a+2b-6} = 0$ (here we are using that $b > 3$). 
\end{enumerate}

Then draw the weave in Figure \ref{fig:aplus2bminus3weave} instead of that in Figure \ref{fig:aplus2bweave}. 

\begin{figure}[h!]
    \centering
    \includegraphics[scale=1.5]{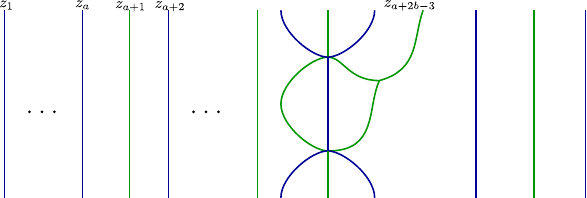}
    \caption{Weave to draw in case all the elements of weight $\wt(z_{a+2b})$ in the bottom of that in Figure \ref{fig:aplus2bweave} are zero. Note that the outgoing variable in the trivalent vertex is $z_{a+2b-6}-z_{a+2b-3}^{-1} = -z_{a+2b-3}^{-1}  \neq 0$.}
    \label{fig:aplus2bminus3weave}
\end{figure}

For the weave in Figure \ref{fig:aplus2bminus3weave} there is a nonzero variable in the bottom with the same weight as $-\wt(z_{a+2b-3})$, namely $-z_{a+2b-3}^{-1}$ (this is the leftmost outgoing variable of the bottom hexavalent vertex). By Lemma \ref{lem:notsame}, there must be another nonzero variable in the bottom with weight not in the span of $\wt(z_{a+2b-3})$. It follows from Lemma~\ref{lem:weights-in-X(a,b)}.(4) that $T$ acts freely on the bottom variables. \\

{\bf Case 2.} Now assume that $z_{a+2b} = 0$, and $z_{a+2b-1} \neq 0$. Assume that $b \not\equiv 0 \mod 3$.
Then form a weave like that in Figure \ref{fig:aplus2weave} (if $b \equiv 2$ mod 3) or \ref{fig:aplus4weave} (if $b \equiv 1$ mod 3),  with the right incoming variable of the trivalent vertex being $z_{a+2b-1}$. If the $T$-action on the variables on the bottom is not free, then, similarly to the previous case, we can draw another weave (shifting our weave to the left by $3$) where we can make sure there is a nonzero element in the bottom of weight $-\wt(z_{a+2b-4})$ and conclude the action on the bottom variables must be free from Lemma \ref{lem:notsame}.  A similar process is repeated assuming $z_{a+2b} = 0, z_{a+2b-1} = 0$ and $z_{a+2b-2} \neq 0$. 
 
 If $b \equiv 0 \mod 3$, we may skip the consideration of  $z_{a+2b-1}$ and assume that $z_{a+2b} = 0$ and $z_{a+2b-2} \neq 0$, cf. Lemma \ref{lem:notbadcases}. Again, in this case we then repeat process similar to the one described in the previous paragraph.\\

{\bf Case 3.} Now assume $z_{a+2b} = z_{a+2b-1} = z_{a+2b-2} = 0$ and $z_{a+2b-3} \neq 0$. Then draw a weave like that in Figure \ref{fig:aplus2bminus3weave}. By Lemma \ref{lem:hard3valent}, the variable $z_{a+2b} = 0$ transforms into $z_{a+2b}z_{a+2b-3}^{2} - z_{a+2b-3} = -z_{a+2b-3} \neq 0$, and we can conclude from Lemma \ref{lem:notsame} that the $T$-action on the bottom variables must be free. \\

We keep reading the word right-to-left, always assuming we have a variable $z_{i} \neq 0$ with $z_{i+1} = \cdots = z_{a+2b} = 0$ (some variables must be ommitted in case $b \equiv 0$ mod 3, note that Lemma \ref{lem:notbadcases} ensures we can do this) and using Lemmas \ref{lem:notsame} and \ref{lem:hard3valent} (when $i \geq a$) or Lemmas \ref{lem:notbadcases} and \ref{lem:easy3valent} (when $i < a$) to conclude that, if the action of $T$ on $z$ is free, then we can always find a weave with one trivalent vertex whose end braid is (cyclically equivalent to) a braid of the form $\beta(a', b')$ and where the action of $T$ on the bottom variables is free. By an inductive argument, the element of $X(a', b')$ corresponding to the bottom variables belongs to a cluster torus. Using now Corollary \ref{cor:necessary-for-induction} we conclude that the initial element $z \in X(a,b)$ belongs to a cluster torus. 

\begin{corollary}\label{cor:conj-x(a,b)}
The varieties $X(a,b)$ have no mysterious points, and thus Conjecture \ref{conj:main} is valid for them.
\end{corollary}
\begin{proof}
First, we may assume that $b > 3$, for otherwise we are dealing with a cluster variety of type $A$ and can appeal to Corollary \ref{cor:type-a}. Note that, upon a cyclic rotation 
\[
\beta(a,b) = \sigma_1^{a}(\sigma_2\sigma_1)^{b-2}(\sigma_2)(\sigma_2\sigma_1\sigma_2) \simeq \Delta \sigma_1^{a}(\sigma_2\sigma_1)^{b-2}\sigma_2.
\]
Thus, $X(a,b) \cong \BS(\sigma_1^{a}(\sigma_2\sigma_1)^{b-2}\sigma_2)$. It follows from Lemma \ref{lem:surjective-in-bs-case} that the torus of cluster automorphisms $\Aut(X(a,b))$ is precisely the maximal torus in $\mathrm{PGL}(2)$. By Theorem \ref{thm:main}, if this torus acts freely on $z \in X(a,b)$ then $z \not\in \per(X(a,b))$, i.e., $\per(X(a,b)) \subseteq \cS(X(a,b))$, which is the nontrivial part of Conjecture \ref{conj:main}. 
\end{proof}

\begin{corollary}\label{cor:x(a,b)-empty-periphery}
Assume $b > 3$. Then, $\per(X(a,b)) = \emptyset$ if and only if $a$ is odd and $b \equiv 0$ or $b \equiv 2$ $\mod 3$. 
\end{corollary}
\begin{proof}
Since $X(a,b) \cong \BS(\sigma_1^{a}(\sigma_2\sigma_1)^{b-2}\sigma_2)$, it follows from Corollary \ref{cor:conj-x(a,b)} that Conjecture \ref{conj:conj-in-bs-case} is valid for $X(a,b)$. So we need to show that $s_1^{a}(s_2s_1)^{b-2}s_2 \in S_3$ is a $3$-cycle if and only if $a$ is odd and $b -2 \equiv 0$ or $b-2 \equiv 1$ $\mod 3$, which is straightforward to do. 
\end{proof}

In Section \ref{sec:geometry-of-deep-loci}, we will refine this theorem by examining the geometry of the deep locus in the cases where it is nonempty.

\begin{remark} \label{rmk:cluster-x}
It is shown in \cite[Theorem 8.6]{CGGLSS} that all braid varieties and, in particular, all varieties $X(a, b)$ also admit \emph{cluster Poisson structures} in the sense of Fock--Goncharov \cite{FG09}, and the canonical $p$-maps are unimodular isomorphisms in this case. While we do not give any definitions, a precise statement is that the union of cluster tori $\bigcup_{t \in \seeds{X(a, b)}} T(t)$ is unimodularly isomorphic to the union of cluster $\mathcal{X}$-charts known as a \emph{cluster $\mathcal{X}$-variety}, or \emph{cluster Poisson variety}. This implies that the affinizations of these isomorphic schemes are also isomorphic. Corollary \ref{cor:x(a,b)-empty-periphery} then gives the following.
\end{remark}

\begin{corollary} \label{cor:x(a,b)-as-x-empty-periphery}
Assume $b > 3$. Then $X(a, b)$ is covered by the union of cluster $\mathcal{X}$-tori constructed in \cite{CGGLSS} if and only if $a$ is odd and $b \equiv 0$ or $b \equiv 2$ $\mod 3$. 
\end{corollary}

\subsection{Varieties of finite cluster type} Another consequence of Theorem \ref{thm:main} is the following result.

\begin{theorem} \label{thm:finite-type-really-full-rank}
Let $\A$ be a cluster variety of simply-laced finite cluster type and really full rank. Then, $\A$ has no mysterious points. Moreover, if the exchange matrix of $\A$ is connected then $\per(\A) = \emptyset$ if and only if $\A$ is of type $A_{2n}, E_6$ or $E_8$.
\end{theorem}

\begin{proof}
If $\A$ is of type $A$, the result follows from Corollary \ref{cor:type-a}. For all the other types, we will show that one cluster variety of the corresponding type and really full rank can be realized as $X(a,b)$ for appropriate $(a,b)$. The result then will follow from Corollaries \ref{cor:independence-of-coefficients} and \ref{cor:conj-x(a,b)}.

Now, if $b = 4$, then
\[
\begin{array}{rl}
\beta(a,4) = &  \sigma_1^{a}(\sigma_2\sigma_1)^{4} \\
 = & \sigma_1^{a}\sigma_2\sigma_1\sigma_2\sigma_1\sigma_2\Delta \\
 \simeq & \Delta\sigma_1^{a}\sigma_2\sigma_1\sigma_2\sigma_1\sigma_2 \\
 = & \Delta \sigma_1^{a+1}\sigma_2\sigma_1^{2}\sigma_2
\end{array}
\]
so that $X(a,4) \cong \BS(\sigma_1^{a+1}\sigma_2\sigma_1^{2}\sigma_2)$. It is known, see e.g. Example \ref{ex:type-d5} (b), that $\BS(\sigma_1^{a+1}\sigma_2\sigma_1^{2}\sigma_2)$ is a cluster variety of type $D_{a+3}$. By Corollary \ref{cor:x(a,b)-empty-periphery}, $\per(X(a,4)) \neq \emptyset$, regardless of the value of $a \geq 1$. So a cluster variety of really full rank and type $D_n$, $n \geq 4$, always has nonempty deep locus. 

If we look at $\beta(1,5)$ we have
\[
\begin{array}{rl}
\beta(1,5) = & \sigma_1(\sigma_2\sigma_1)^{3}\sigma_2\Delta \\
\simeq & \Delta (\sigma_1\sigma_2)^{4}.
\end{array}
\]
So that $X(1,5) \cong \BS((\sigma_1\sigma_2)^{4})$ By Example \ref{ex:type-d5}(c), $\BS((\sigma_1\sigma_2)^{4})$ is a cluster variety of type $E_6$. It follows from Corollary \ref{cor:x(a,b)-empty-periphery} that $\per(X(1,5)) = \emptyset$. So a cluster variety of type $E_6$ and really full rank always has empty deep locus. The case $E_8$ follows similarly, after observing that $X(1,6) \cong \BS((\sigma_1\sigma_2)^{5})$ which thanks to Example \ref{ex:type-d5}(e)is a cluster variety of type $E_8$.

For type $E_7$, consider $\beta(2,5)$:
\[
\begin{array}{rl}
\beta(2,5) = & \sigma_1^{2}(\sigma_2\sigma_1)^{5} \\
\simeq & \sigma_1(\sigma_2\sigma_1)^{5}\sigma_2 \\
= & (\sigma_1\sigma_2)^{4}\sigma_1\Delta \\
\simeq & \Delta((\sigma_1\sigma_2)^{4}\sigma_1)
\end{array}
\]
so that $X(2,5) \cong \BS((\sigma_1\sigma_2)^{4}\sigma_1)$, which is a cluster variety of type $E_7$ by Example \ref{ex:type-d5}(d). By Corollary \ref{cor:x(a,b)-empty-periphery}, $\per(X(2,5)) \neq \emptyset$, and this finishes the proof. 
\end{proof}

\begin{corollary}
Let $\A$ be a cluster variety of simply-laced finite cluster type with an arbitrary choice of frozens. Then, $\A$ has no mysterious points. 
\end{corollary}

\begin{proof}
This follows from Theorem \ref{thm:finite-type-really-full-rank} by Corollary \ref{cor:reduction-to-full-rank}.
\end{proof}

\subsection{\texorpdfstring{$X(a,b)$}{X(a,b)} as a positroid variety} 
Just as we can understand our two strand results using positroid varieties (Remark~\ref{rem:G2Cluster}), we can also understand our three strand results using positroid varieties.
If $b = 0$ then $X(a,b)$ is a cluster variety of type $A$ and really full rank, which is up to a torus factor a maximal positroid in $\Gr(2,a+1)$. Now assume $a, b \geq 1$. We claim that the variety $X(a,b)$ is, up to a torus factor, isomorphic to a positroid variety in $\Gr(3, a+b+1)$. In order to see this, we use \cite[Theorem 1.3]{CGGS_positroid}. It suffices to construct a $3$-bounded affine permutation in the extended affine symmetric group $f \in \tilde{S}_{a+b+1}$ whose juggling braid $J_3(f)$ is precisely $\sigma_1^{a}(\sigma_2\sigma_1)^{b}$. We refer the reader to \cite[Sections 2.1.3 and 2.3]{CGGS_positroid} for the definition of a bounded affine permutation (which is equivalent data to a decorated permutation) and its juggling braid. See also \cite{KLS13}.

Let $n:= a+b+1$ and define $f: \Z \to \Z$ as follows: $f(x+n) = f(x) + n$ and for $1 \leq x \leq n$:
\begin{equation}\label{eq:positroid-data}
f(x) = \begin{cases} x + 2 & \text{if} \;1 \leq  x \leq a-1, \\
x + 3 & \text{if} \;  a \leq x \leq n -1,\\
x + a + 2 & \text{if} \; x = n. \end{cases} 
\end{equation}
Note that $f\{1, 2, \dots, n\} = \{3, \dots, a, a+1, a+3, \dots, n, n+1, n+1, n+a+2\}$ and these are pairwise distinct modulo $n$, so $f$ indeed defines an element $f \in \tilde{S}_n$. Note also that $\{x \in \{1, \dots, n\} \mid f(x) > n\} = \{n-2, n-1, n\}$, so that $f$ is a $3$-bounded affine permutation. It is straighforward to verify that the juggling braid $J_3(f)$ coincides with $\sigma_1^{a}(\sigma_2\sigma_1)^{b}$, see Figure \ref{fig:juggling-braid}. We summarize this in the following result.

\begin{proposition}\label{prop:x-is-a-positroid}
The variety $X(a,b)$ is, up to a torus factor, isomorphic to the open positroid variety $\Pi^{\circ}_{f} \subseteq \Gr(3,a+b+1)$, where $f \in \tilde{S}_{a+b+1}$ is given as in \eqref{eq:positroid-data}. 
\end{proposition}

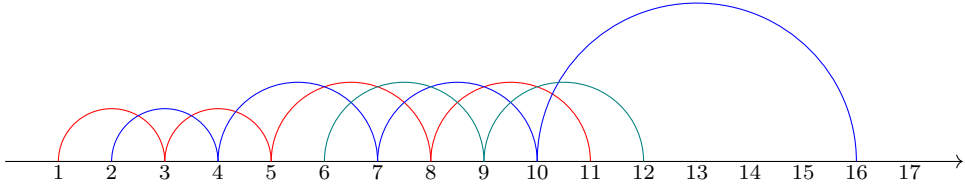
\begin{figure}
\begin{tikzpicture}[scale=0.7]
    \draw[->](0,0) -- (18, 0);
    \draw (1,-0.2) node {\scriptsize{1}};
    \draw (2,-0.2) node {\scriptsize{2}};
    \draw (3,-0.2) node {\scriptsize{3}};
    \draw (4,-0.2) node {\scriptsize{4}};
    \draw (5,-0.2) node {\scriptsize{5}};
    \draw (6,-0.2) node {\scriptsize{6}};
    \draw (7,-0.2) node {\scriptsize{7}};
    \draw (8,-0.2) node {\scriptsize{8}};
    \draw (9,-0.2) node {\scriptsize{9}};
    \draw (10,-0.2) node {\scriptsize{10}};
    \draw (11,-0.2) node {\scriptsize{11}};
    \draw (12,-0.2) node {\scriptsize{12}};
    \draw (13,-0.2) node {\scriptsize{13}};
    \draw (14,-0.2) node {\scriptsize{14}};
    \draw (15,-0.2) node {\scriptsize{15}};
    \draw (16,-0.2) node {\scriptsize{16}};
    \draw (17, -0.2) node{\scriptsize{17}};

    \draw[color=red] (3,0) arc(0:180:1);
    \draw[color=red] (5,0) arc(0:180:1);
    \draw[color=red] (8,0) arc(0:180:1.5);
    \draw[color=red] (11,0) arc(0:180:1.5);

    \draw[color=blue] (4,0) arc(0:180:1);
    \draw[color=blue] (7,0) arc(0:180:1.5);
     \draw[color=blue] (10,0) arc(0:180:1.5);
      \draw[color=blue] (16,0) arc(0:180:3);

      \draw[color=teal] (9,0) arc(0:180:1.5);
      \draw[color=teal] (12,0) arc(0:180:1.5);
\end{tikzpicture}
\caption{The juggling diagram for the bounded affine permutation $f$ given by \eqref{eq:positroid-data} when $a = 4, b = 5$. It is obtained by joining $x$ to $f(x)$ for $1 \leq x \leq n = a+b+1$. We have colored the arcs according to the strand of the juggling braid they belong to. Note that this braid is $\sigma_1^{4}(\sigma_2\sigma_1)^{5}$.}
\label{fig:juggling-braid}
\end{figure}

In terms of the Richardson datum $u, w \in S_{a+b+1}$, $u \leq w$, defining the positroid $\Pi^{\circ}_{f}$, note that we have, in one-line notation, $w = [n-2, n-1, n, 1, 2, \dots, n-3]$ and $u = [1, 2, a+2, 3, \dots, a+1, a+3, \dots, n]$ (if $a = 1$ then $u$ is the identity). Note in particular that $w$ is always the maximal $3$-Grassmannian permutation in $S_{n}$. 

\section{Geometry of the stabilizer and deep loci}
\label{sec:geometry-of-deep-loci}

When the deep or the stabilizer locus  is nonempty, it is an interesting question to study its geometry. We will describe irreducible components of the  $T$-stabilizer loci of double Bott-Samelson and more general braid varieties in terms of smaller double Bott-Samelson, resp. braid varieties. For the varieties $X(a,b)$, this provides a description of the deep locus thanks to Corollary~\ref{cor:conj-x(a,b)}. Further, for more general varieties of the same cluster types of really full rank, results of Section \ref{sec:deep-locus} allow us to establish some results on the geometry of their deep loci.

\subsection{The stabilizer locus for double Bott-Samelson varieties}\label{subsec:stabilizer-double-bs} We would like to describe the stabilizer locus for double Bott-Samelson varieties. For simplicity, we will assume that $\beta \in \Br^{+}_{n}$ is a braid word such that every braid generator appears in $\beta$: this guarantees that $T = (\C^{\times})^{n}/\C^{\times} \cong \Aut(\BS(\beta))$, cf. Lemma \ref{lem:surjective-in-bs-case}. In the case where not every braid generator appears in $\beta$, we have a decomposition $\BS(\beta) \cong \BS(\beta_1) \times \BS(\beta_2)$, so we do not lose any generality by assuming this. We will use the description of the double Bott-Samelson variety and its torus action in terms of flags, cf. Section \ref{subsec:torus-actions}.

Let us first describe a shuffle product on flags. Let $\{1, \dots, n\} = I_1 \sqcup I_2$ be a partition in two disjoint sets, with $|I_1| = a, |I_2| = b$. The shuffle product is an operation $\shuffle_{w}: \Fl(a) \times \Fl(b) \to \Fl(n)$ depending on an element $w \in S_n$ and described as follows. First, we think of $\C^{a}$ as $\C^{I_1}$, with ordered basis $\{e_{i} \mid i \in I_1\}$, and similarly we think of $\C^{b}$ as $\C^{I_2}$, so that we get a direct sum decomposition $\C^{n} = \C^{a} \oplus \C^{b}$. Now for two flags $F^{\bullet} \in \Fl(a)$ and $G^{\bullet} \in \Fl(b)$, and $j = 1, \dots, n$ we define
\[
(F^{\bullet} \shuffle_{w} G^{\bullet})^{j} = F^{j_1} \oplus G^{j_2}, \; \text{where} \; j_1 = |I_1\cap w([j])|, j_2 = |I_2\cap w([j])|. 
\]

\begin{remark}\label{rmk:coset}
Note that $\shuffle_{w'w} = \shuffle_{w}$ if $w' \in S_{I_1} \times S_{I_2} \subseteq S_{n}$. 
\end{remark}

\begin{lemma}\label{lem:lu-decomposition}
Using the same notation as above, for any $w \in S_n$ we have that $F^{\bullet} \shuffle_w G^{\bullet} \in \borel_{-}\borel/\borel$ if and only if $F^{\bullet} \in \borel_{a, -}\borel_{a}/\borel_{a}$ and $G^{\bullet} \in \borel_{b, -}\borel_{b}/\borel_{b}$. 
\end{lemma}
\begin{proof}
Thanks to Remark \ref{rmk:coset}, we may assume that $w$ is a minimal length coset representative in $(S_{I_1} \times S_{I_2})\backslash S_n$, that is, we may assume that
\begin{equation}\label{eq:minimal-length}
w^{-1}(i) < w^{-1}(j) \; \text{if both} \; i \; \text{and} \; j \; \text{belong to one of} \; I_1, I_2. 
\end{equation}
Now, if $(M_{i,j})_{i,j \in I_1}$ and $(N_{i,j})_{i, j \in I_2}$ are matrices representing the flags $F^{\bullet}$ and $G^{\bullet}$, respectively, then a matrix representing the flag $F^{\bullet} \shuffle_{w} G^{\bullet}$ is given by $(P_{ij})$ where
\[
P_{ij} = \begin{cases} M_{w(i), w(j)} & \text{if} \; w(i), w(j) \in I_1, \\ N_{w(i), w(j)} & \text{if} \; w(i), w(j) \in I_2, \\ 0 & \text{else}.  \end{cases} 
\]
It follows from \eqref{eq:minimal-length} that a principal minor of $P$ is, up to a sign, the product of a principal minor of $M$ and a principal minor of $N$. The result follows. 
\end{proof}

\begin{lemma}\label{lem:shuffle-vs-stabilizers}
As above, let $\{1, \dots, n\} = I_1 \sqcup I_2$ be a partition with $|I_1| = a$ and $|I_2| = b$. Define the torus
\[
T_{I_1, I_2} = \{(t_1, \dots, t_n) \in (\C^{\times})^{n} \mid t_i = t_j \; \text{whenever both} \; i \; \text{and} \; j \; \text{belong to one of} \; I_1, I_2\}. 
\]
Then, a flag $F^{\bullet} \in \Fl(n)$ is stabilized by $T_{I_1, I_2}$ if and only if there exist flags $F_1^{\bullet} \in \Fl(a)$, $F_2^{\bullet} \in \Fl(b)$ and $w \in S_n$ such that
\[
F^{\bullet} = F_1^{\bullet} \shuffle_{w} F_2^{\bullet}.
\]
\end{lemma}
\begin{proof}
A flag $F^{\bullet}$ is stabilized by $T_{I_1, I_2}$ if and only if for every $1 \leq i \leq n$ we can find a basis of $F^{i}$ consisting of elements that belong to $\C^{I_1} \oplus 0$ or $0 \oplus \C^{I_2}$. This creates flags $F_{1}^{\bullet} \in \Fl(a), F_2^{\bullet} \in \Fl(b)$ and an element $w \in S_n$ such that $F^{\bullet} = F_{1}^{\bullet} \shuffle_{w} F_2^{\bullet}$. 
\end{proof}

\begin{theorem}\label{thm:components-stabilizer-double-bs}
Let $\beta \in \Br_n^{+}$, and assume that every braid generator appears in $\beta$. Let $C$ be the set of connected components of the braid closure $\overline{\beta}$, and let $C = C_1 \sqcup C_2$ be a partition of $C$. Let $\{1, \dots, n\} = I_1 \sqcup I_2$ be the partition given by $k \in I_j$ if the $k$-th strand of the braid belongs to the component $C_j$. Then, the set of elements in $\BS(\beta)$ that are stabilized by $T_{I_1, I_2}$ is isomorphic to $\BS(\beta_1) \times \BS(\beta_2)$, where $\beta_1 \in \Br^{+}_{I_1}$ and $\beta_2 = \Br^{+}_{I_2}$ are the braid words obtained by considering only the crossings between components belonging to $C_1$ and $C_2$, respectively. 
\end{theorem}
See Figure \ref{fig:stabilizer-components} for an example of the concepts introduced in Theorem \ref{thm:components-stabilizer-double-bs}.
\begin{proof}
As above, let $a = |I_1|$, $b = |I_2|$. Let $\beta = \sigma_{i_1}\cdots \sigma_{i_{\ell}}$, and let
\[
F = (F_0^{\bullet}, \dots, F_{\ell}^{\bullet}) \in \BS(\beta)
\]
be such that $T_{I_1, I_2}$ stabilizes $F$. By Lemma \ref{lem:shuffle-vs-stabilizers} for each $j = 1, \dots, \ell$ there exist flags $G_j^{\bullet} \in \Fl(a)$, $H_j^{\bullet} \in \Fl(b)$ and $v_j \in S_n$ such that
\[
F_j^{\bullet} = G_j^{\bullet} \shuffle_{v_j} H_j^{\bullet}. 
\]
Note that we may assume $G_0^{\bullet}$ and $H_0^{\bullet}$ are the standard flags, while $v_0$ can be any permutation in $S_{I_1} \times S_{I_2}$. We would like to express the relationship between $(G_j^{\bullet}, H_j^{\bullet}, v_j)$ and $(G_{j+1}^{\bullet}, H_{j+1}^{\bullet}, v_{j+1})$ in terms of the crossing $\sigma_{i_j}$. Let us first deal with the permutations $v_0, \dots, v_{\ell}$. For this, we define a (a priori, different) sequence of permutations $v'_0, \dots, v'_{\ell} \in S_n$ as follows. Right before the $j$-th crossing of $\beta$, scan the braid top-to-bottom and let $v'_{j-1}(i) \in I_1$ if the $i$-th strand belongs  a component in $C_1$, and $v'_{j-1}(i) \in I_2$ if the $i$-th strand belongs to a component in $C_2$. Thanks to Remark \eqref{rmk:coset} we may assume that $(v'_{j-1})^{-1}(i_1) < (v'_{j-1})^{-1}(i_2)$ if both $i_1, i_2$ belong to one of $I_1, I_2$, and this determines the permutation $v'_{j-1}$ uniquely, see Figure \ref{fig:stabilizer-components} for an example. Note that $v'_{0} = e$, the identity.

Now we will inductively show the following:
\begin{enumerate}
\item $F_{j}^{\bullet} = G^{\bullet}_{j} \shuffle_{v'_j} H^{\bullet}_{j}$.
\item $G^{\bullet}_{j} = G^{\bullet}_{j+1}$ and $H^{\bullet}_{j} = H^{\bullet}_{j+1}$ if the crossing $\sigma_{i_{j+1}}$ is between strands belonging to one component  in $C_1$ and another in $C_2$.
\item $H_{j}^{\bullet} = H_{j+1}^{\bullet}$ and $G_j^{\bullet}$ is in position $k$ with respect to $G_{j+1}^{\bullet}$, if the crossing $\sigma_{i_{j+1}}$ is between strands corresponding to components of $C_1$, that occupy the $k$-th and $k+1$-st positions among such strands right before the crossing $\sigma_{i_{j+1}}$.
\item $G_{j}^{\bullet} = G_{j+1}^{\bullet}$ and $H_j^{\bullet}$ is in position $k$ with respect to $H_{j+1}^{\bullet}$, if the crossing $\sigma_{i_{j+1}}$ is between strands corresponding to components of $C_2$, that occupy the $k$-th and $k+1$-st positions among such strands right before the crossing $\sigma_{i_{j+1}}$.
\end{enumerate}

Note that (1) in the case $j = 0$ is clear. Now, assume that we know (1) for $j$, and let us show it for $j+1$, along the way showing that (2), (3) and (4) are also true. We have three cases. \\

{\it Case 1. The crossing $\sigma_{i_{j+1}}$ is between strands belonging to one component in $C_1$ and another in $C_2$}. Note that in this case, $|I_1 \cap v'_j[i_j]| = |I_1 \cap v'_j[i_j+1]|$ while $|I_2 \cap v'_j[i_j]| + 1 =|I_2 \cap v'_j[i_j+1]|$, and $|I_1 \cap v'_{j+1}[i_j]| + 1 = |I_1 \cap v'_{j+1}[i_j+1]|$ while $|I_2 \cap v'_{j+1}[i_j]| =|I_2 \cap v'_j[i_j+1]|$, or a similar condition with $1, 2$ interchanged. So $F_{j}^{i_{j+1}+1}/F_{j}^{i_{j+1}-1}$ has a basis consisting of a vector $x \in \C^{I_1}$ and a vector $y \in \C^{I_2}$ and $F_{j}^{i_{j+1}}$ is the preimage of $\langle x\rangle$ (or $\langle y \rangle$). Since $F_j^{i_j}$ must differ from $F_{j+1}^{i_j}$ and the latter must be stable under the $T_{I_1, I_2}$-action, the only option is that $F_{j+1}^{i_{j}}$ is the preimage of $\langle y \rangle$ (or $\langle x \rangle$, respectively). But then it follows that $F^{\bullet}_{j+1} = G^{\bullet}_{j} \shuffle_{v'_{j+1}} H^{\bullet}_{j}$, as needed. \\

{\it Case 2. The crossing $\sigma_{i_{j+1}}$ is between strands belonging to components in $C_1$, that occupy the $k$-th and $k+1$-st positions among such strands right before the crossing $\sigma_{i_{j+1}}$}. In this case, $F_{j}^{i_{j+1}+1}/F_{j}^{i_{j+1}-1} = G_{j}^{k+1}/G_{j}^{k-1}$. So we must have $F_{j+1}^{\bullet} = G_{j+1}^{\bullet} \shuffle_{v'_j} H_{j}^{\bullet}$, where $G_j$ and $G_{j+1}$ differ in the $k$-th subspace. Finally, since $\sigma_{i_{j+1}}$ is a crossing between two components of $C_1$, we have $v'_{j+1}(v'_{j})^{-1} \in S_{I_1} \subseteq S_{I_1} \times S_{I_2}$. So by Remark \ref{rmk:coset} we obtain $G_{j+1}^{\bullet} \shuffle_{v'_j} H_{j}^{\bullet} = G_{j+1}^{\bullet} \shuffle_{v'_{j+1}} H_{j}^{\bullet}$. \\

{\it Case 3. The crossing $\sigma_{i_{j+1}}$ is between strands belonging to components in $C_2$, that occupy the $k$-th and $k+1$-st positions among such strands right before the crossing $\sigma_{i_{j+1}}$}. This is similar to Case 2. \\

So, given an element $F \in \BS(\beta)$ that is stabilized by $T_{I_1, I_2}$ we obtain a pair of sequences $G, H$. Now, $G \in \BS(\beta_1)$ by (3) above and Lemma \ref{lem:lu-decomposition}, while $H \in \BS(\beta_2)$ by (4) above and Lemma \ref{lem:lu-decomposition}. Conversely, given $(G, H) \in \BS(\beta_1) \times \BS(\beta_2)$, (1)--(4) above give us a recipe for creating an element $F \in \BS(\beta)$, and it is clear that these two assignments are inverse of each other.
\end{proof}

\begin{figure}
    \centering
    \includegraphics[scale=0.7]{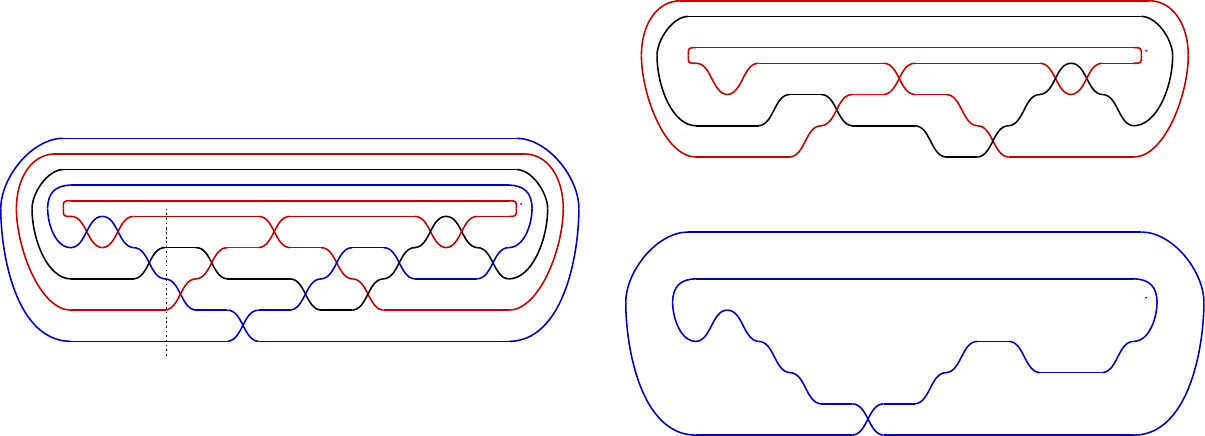}
    \caption{The closure of the braid $\beta = \sigma_1^{2}\sigma_2\sigma_3\sigma_2\sigma_4\sigma_1\sigma_3\sigma_2\sigma_3\sigma_2\sigma_1^{2}\sigma_2$ has three connected components, $C = \{$\textcolor{red}{$c_1$}, \textcolor{blue}{$c_2$}, $c_3\}.$ Choosing the partition $C = \{$\textcolor{red}{$c_1$},$ c_3\} \sqcup \{$\textcolor{blue}{$c_2$}$\}$ we obtain $I_1 = \{$\textcolor{red}{$1$}, $3$, \textcolor{red}{$4$}$\}$ and $I_2 = \{$\textcolor{blue}{$2, 5$}$\}$. The component of the stabilizer locus associated to this partition is isomorphic to $\BS(\sigma_2\sigma_1\sigma_2\sigma_1^{2}) \times \BS(\sigma_1)$. Looking at the slice of the braid given by the dotted line, we can see that $v'_3 = [1, 3, 2, 4, 5]$. }
    \label{fig:stabilizer-components}
\end{figure}

\begin{definition}
Let $\beta \in \Br^{+}_{n}$, and assume that every braid generator appears in $\beta$. Let $C$ be the set of connected components of the braid closure $\overline{\beta}$ and let $C = C_1 \sqcup C_2$ be a partition of $C$ into two non-empty subsets. We denote by $\mathcal{C}(C_1, C_2)$ the set of elements $F \in \BS(\beta)$ stabilized by $T_{I_1, I_2}$, where $I_1$ and $I_2$ are as in Theorem \ref{thm:components-stabilizer-double-bs}. 
\end{definition}

\begin{corollary}\label{cor:components-stabilizer-double-bs}
Let $\beta \in \Br^{+}_{n}$ and assume that every braid generator appears in $\beta$. The irreducible components of $\mathcal{S}(\BS(\beta))$ are precisely the varieties of the form $\mathcal{C}(C_1, C_2)$, where $(C_1, C_2)$ runs over the set of partitions of the set of components $C$ of the braid closure $\overline{\beta}$.
\end{corollary}
\begin{proof}
Since every braid generator appears in $\beta$ we know from Lemma \ref{lem:surjective-in-bs-case} that $\Aut(\BS(\beta)) \cong (\C^{\times})^{n}/\C^{\times}$. Moreover, we know from Lemma \ref{lem:stabilizers} that the stabilizer in $(\C^{\times})^{n}$ of any element of the Bott-Samelson variety $\BS(\beta)$ is contained in $T_{\beta} = \{(t_1, \dots, t_n) \mid t_{i} = t_{w(i)}\}$ for $w = \pi(\beta)$. Note that the set of tori $S$ that satisfy 
\[
\{(t, \dots, t) \mid t \in \C^{\times}\} \subseteq S \subseteq T_{\beta}
\]
and such that $S/\C^{\times}$ is a rank $1$ torus are precisely the tori of the form $T_{I_1, I_2}$ where $I_1, I_2$ are obtained from a partition $C = C_1 \sqcup C_2$ as in Theorem \ref{thm:components-stabilizer-double-bs}, for $C_1, C_2 \neq \emptyset$. Thanks to Lemma~\ref{lem:stabilizers-of-points}, it follows then that every point in $\mathcal{S}(\BS(\beta))$ belongs to $\mathcal{C}(C_1, C_2)$ for some partition of $C$. In terms of the coordinates $(z_1, \dots, z_{\ell})$ on $\BS(\beta)$, the locus $\mathcal{C}(C_1, C_2)$ consists of those points for which $z_j = 0$ whenever $\sigma_{i_j}$ is a crossing between a component in $C_1$ and a component in $C_2$, so $\mathcal{C}(C_1, C_2)$ is closed. Finally, by Theorem~ \ref{thm:components-stabilizer-double-bs} and isomorphism \eqref{eq:BS-decomposition}, $\mathcal{C}(C_1, C_2)$ is isomorphic to a double Bott-Samelson variety, and it is thus irreducible. 
\end{proof}

Iterating the arguments in the proof of Theorem \ref{thm:components-stabilizer-double-bs} and Corollary \ref{cor:components-stabilizer-double-bs} we obtain the following result.

\begin{corollary}\label{cor:intersections-components}
Let $\beta \in \Br_n^{+}$ be a positive braid. Any set partition $I_1 \sqcup I_2 \sqcup \dots \sqcup I_k$ of $\{1, \dots, n\}$ where each $I_j$ is a union of orbits of $\pi(\beta)$ (or equivalently, a set partition $C = C_1 \sqcup \dots \sqcup C_k$ on the set of connected components of the braid closure $\overline{\beta}$ gives rise to a subgroup $T_{I_1, \dots, I_{k}} \subseteq (\C^{\times})^{n}$. For $j = 1, \dots, k$ let $\beta_j$ be the braid obtained by considering only the strands indexed by $I_{j}$. Then, the set of elements of $\BS(\beta)$ which are stabilized by $T_{I_1, \dots, I_k}$ is isomorphic to $\BS(\beta_1) \times \cdots \times \BS(\beta_k)$. In particular, the intersection of the sets of points stabilized by two subtori of $(\C^{\times})^{n}$ is isomorphic to the product of double Bott Samelson varieties associated with the coarsest common refinement of the corresponding partitions of the set $C$. 
\end{corollary}

Note that there may be two different set partitions of $C$ giving rise to the same locus of $\BS(\beta)$ in Corollary \ref{cor:intersections-components}. For example, if 
$\beta = \sigma_1\sigma_2\sigma_2\sigma_1 \in \Br_{3}^{+}$ we have that $\overline{\beta}$ consists of three components, $C = \{C_1, C_2, C_3\}$, so that $C_2$ and $C_3$ do not intersect. If we set $I_1 = \{C_1\}, I_2 = \{C_2, C_3\}$ and $I'_1 = \{C_1\}$, $I'_2 = \{C_2\}$, $I'_3 = \{C_3\}$, we can see that these two set partitions give rise to the same locus in $\BS(\beta)$. 

\subsection{The deep locus of \texorpdfstring{$X(a,b)$}{X(a,b)}} The varieties $X(a,b)$ are double Bott-Samelson varieties which have no mysterious points by Corollary~\ref{cor:conj-x(a,b)}, so Theorem \ref{thm:components-stabilizer-double-bs} provides a description of the irreducible components of their deep locus. 

\begin{theorem}\label{thm:geometry-deep-locus}
Assume throughout that $b > 3$ and $a \geq 1$. Then,
\begin{itemize}
\item[(a)] If $a$ is odd and $b \equiv 1 \mod 3$, then $\per(X(a,b))$ has three components $C_1, C_2, C_3$. Moreover,
\begin{itemize}
    \item $C_1$ and $C_2$ are both cluster varieties of type $A_{\frac{2}{3}(b-1) - 1}$ with a single frozen.
    \item $C_3$ is a cluster variety of type $A_{a + \frac{1}{3}(2b - 5) -1}$ with a single frozen.
    \item $C_1\cap C_2 = C_2\cap C_3 = C_1\cap C_3 = \{\mathrm{pt}\}$. 
\end{itemize}
In this case $\per(X(a,b))$ is never smooth, and it is equidimensional if and only if $a = 1$. 
\item[(b)] If $a$ is even and $b \equiv 0 \mod 3$, then $\per(X(a,b))$ is smooth, irreducible, and isomorphic to a cluster variety of type $A_{\frac{2}{3}(b-3)}$ with a single frozen.
\item[(c)] If $a$ is even and $b \equiv 1 \mod 3$, then $\per(X(a,b))$ is smooth, irreducible and isomorphic to a cluster variety of type $A_{a + \frac{2}{3}(b-4)}$ with a single frozen. 
\item[(d)] If $a$ is even and $b \equiv 2 \mod 3$, then $\per(X(a,b))$ is smooth, irreducible, and isomorphic to a cluster variety of type $A_{a+\frac{2}{3}(b-2)}$ with a single frozen. 
\end{itemize}
\end{theorem}

See Figure \ref{fig:deep-components} for an interpretation of Theorem \ref{thm:geometry-deep-locus} in terms of link components. 

\begin{figure}
    \centering
    \includegraphics{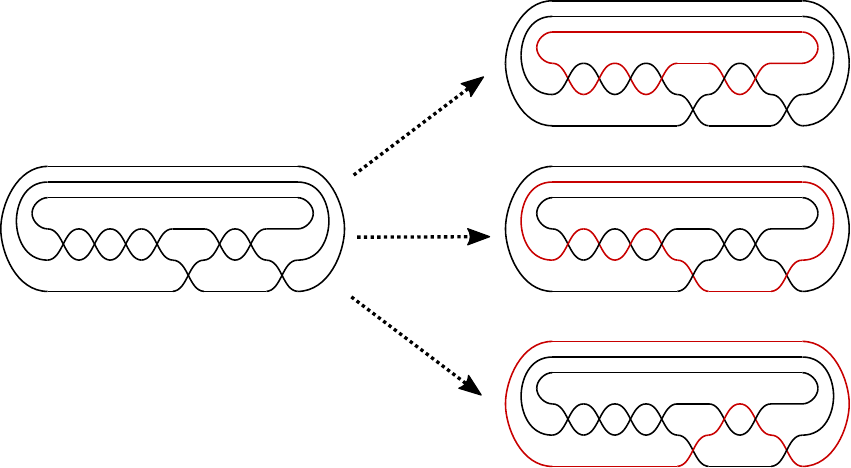}
    \caption{The components of the deep locus of $\BS(\sigma_1^{4}\sigma_2\sigma_1^{2}\sigma_2) \cong X(3,4)$. Each component corresponds to a partition of the set of components into two disjoint parts, which are marked by the colors red and black. Since the red component is always an unknot, the corresponding component is the double Bott-Samelson cell of the black components.}
    \label{fig:deep-components}
\end{figure}

Recall that $X(l+1, k+2)$ is a cluster variety of type $Q_{k, l}$ with two frozens, where $Q_{k,l}$ is the quiver of the following form:

\begin{center}
\begin{tikzpicture}
\node at (0,0) {$\circ$};
\draw[->] (0.2,0)--(0.8,0);
\node at (1,0) {$\circ$};
\draw [->] (1.2, 0) -- (1.6,0);
\node at (2,0) {$\ldots$};
\draw [->] (2.4, 0) -- (2.6,0);
\node at (3,0) {$\circ$};
\draw[->] (3.4, 0) -- (3.8, 0);
\node at (4,0) {$\circ$};
\draw[->] (4.2, 0) -- (4.6, 0);
\node at (5,0) {$\ldots$};
\draw[->] (5.4, 0) -- (5.8, 0);
\node at (6,0) {$\circ$};

\node at (0,-1) {$\circ$};
\draw[->] (0.2,-1)--(0.8,-1);
\node at (1,-1) {$\circ$};
\draw [->] (1.2, -1) -- (1.6,-1);
\node at (2,-1) {$\ldots$};
\draw [->] (2.4, -1) -- (2.8,-1);
\node at (3,-1) {$\circ$};

\draw[->] (0, -0.9) -- (0, -0.1);
\draw[->] (1, -0.9) -- (1, -0.1);
\draw[->] (3, -0.8) -- (3, -0.1);

\draw[->] (0.8, -0.2)--(0.2, -0.8);
\draw[->] (1.8, -0.2)--(1.2, -0.8);
\draw[->] (2.8, -0.2)--(2.2, -0.8);

\draw[decoration={brace,raise=5pt},decorate]
  (-0.2, 0) -- node[above=5pt] {$k$} (3.2, 0);
\draw[decoration={brace,raise=5pt},decorate]
  (3.8, 0) -- node[above=5pt] {$l$} (6.2, 0);
\end{tikzpicture}
\end{center}
Thanks to Corollaries \ref{cor:coincidence-up-to-tori} and \ref{cor:properties-of-deep-loci-up-to-tori}, we immediately get the following corollaries.

\begin{corollary}\label{cor:Q_{k,l}-geometry-deep-locus}
Let $\A$ be a cluster variety of type $Q_{k,l}$ and really full rank. 

\begin{itemize}
\item[(a)] $\per(\A) = \emptyset$ if and only if $l$ is even and $k \equiv 0$ or $k\equiv 1 \mod 3$.
\item[(b)] $\per(\A)$ is smooth and irreducible if and only if $l$ is odd.
\item[(c)] $\per(\A)$ is not smooth and not irreducible if and only if $l$ is even and $k \equiv 2 \mod 3$. Moreover, it is equidimensional if and only if $l = 0$. 
\end{itemize}
\end{corollary}

\begin{corollary}\label{cor:finite-type-geometry-deep-locus}
Let $\A$ be a cluster variety of finite cluster type and really full rank. Assume that the exchange matrix of $\A$ is connected.

\begin{itemize}
\item[(a)] $\per(\A) = \emptyset$ if and only if $\A$ is of type $A_{2n}, E_{6}$ or $E_{8}$.
\item[(b)] $\per(\A)$ is smooth and irreducible if and only if $\A$ is of type $A_{2n+1}, D_{2n+1}$ or $E_{7}$.
\item[(c)] $\per(\A)$ is not smooth and not irreducible if and only if $\A$ is of type $D_{2n}$. Moreover, it is equidimensional if and only if $2n = 4$. 
\end{itemize}

In other words, let $C$ be the Cartan matrix of a type $ADE$ root system, and let $B = C - 2I$. If $\A$ denotes a cluster variety of finite cluster type, really full rank and corresponds to the type of the Cartan matrix $C$, then
\begin{itemize}
\item[(a)] $\per(\A) = \emptyset$ if and only if $B$ has really full rank.
\item[(b)] $\per(\A)$ is smooth and irreducible if and only if the cokernel of $B$ is a rank $1$ free abelian group.
\item[(c)] $\per(\A)$ is not smooth and not irreducible if and only if the cokernel of $B$ is a rank $2$ free abelian group.
\end{itemize}
\end{corollary}

\begin{proof}[Proof of Theorem \ref{thm:geometry-deep-locus}] By Corollary \ref{cor:conj-x(a,b)} we have that $\per(X(a,b)) = \cS(X(a,b))$, so we will really want to study the stabilizer locus of the $T$-action on $X(a,b)$. For this, we will take the Bott-Samelson realization of $X(a,b)$,
\begin{equation}\label{eq: x(a,b) as bs}
X(a,b) \cong \BS(\sigma_1^{a}(\sigma_2\sigma_1)^{b-2}\sigma_2).
\end{equation}
so we use Theorem \ref{thm:components-stabilizer-double-bs}. Let $\beta_{a,b} = \sigma_1^{a}(\sigma_2\sigma_1)^{b-2}\sigma_2$. Since $\beta_{a,b}$ is a three-stranded braid, the link $\overline{\beta_{a,b}}$ has at most three components, and the cases in the statement of Theorem \ref{thm:geometry-deep-locus} are the ones where there is more than one component (otherwise the deep locus is empty). 
Now Case (a) in Theorem \ref{thm:geometry-deep-locus} is the case when $\overline{\beta_{a,b}}$ has three connected components, $C = \{c_1, c_2, c_3\}$. So there are three ways to partition $C$ into two disjoint open sets, and they correspond to the components $C_1, C_2, C_3$ described in the statement of the Theorem. See Figure \ref{fig:deep-components} for an example. The Cases (b), (c) and (d) are those for which $\overline{\beta_{a,b}}$ has two components, $C = \{c_1, c_2\}$, so the only nontrivial partition of $C$ is $C = \{c_1\} \sqcup \{c_2\}$. In this case, we may assume that $c_1$ is an unknot, so $c_2$ is the closure $\overline{\beta'}$ of a braid in $2$-strands and $\per(\BS(\beta_{a,b})) \cong \BS(\beta')$, which is a type $A$ cluster variety with a single frozen. 
\end{proof}

\begin{remark}
In all cases of Theorem \ref{thm:geometry-deep-locus}, all components of the deep locus are again cluster varieties. Moreover, by Theorem \ref{thm:components-stabilizer-double-bs}, assuming Conjecture \ref{conj:main} we obtain that every components in the deep locus of a double Bott-Samelson cell is again a double Bott-Samelson cell and thus a cluster variety. We now give an example of a cluster variety where the components of the deep locus are \emph{not} again cluster varieties. Consider the cluster variety of type $D_4$ without frozen vertices, whose exchange matrix is
\[
\tilde{B} = \left(\begin{matrix} 0 & 0 & 1 & 0 \\ 0 & 0 & 1 & 0 \\ -1 & -1 & 0 & -1 \\ 0 & 0 & 1 & 0 \end{matrix} \right).
\]
It is known, see \cite[Corollary 1.21]{BFZ}, that the cluster variety is
\[
\A = \{(x_1, x'_1, x_2, x'_2, x_3, x'_3, x_4, x'_4) \in \C^{8} \mid x_1x_1' = x_2x_2' = x_4x_4' = 1+x_3, x_3x_3' = 1+x_1x_2x_4\},
\]
and the automorphism group is $\Aut(A) = \{(t_1, t_2, t_3, t_4) \in (\C^{\times})^{4} \mid t_3 = t_1t_2t_4 = 1\}$, which is a $2$-dimensional torus. A point is in the deep(=stabilizer) locus if and only if there exist $i, j \in \{1, 2, 4\}$ with $i \neq j$ such that $x_i = x'_i = x_j = x'_j = 0$. This implies that $x_3 = x'_3 = -1$, and $x_kx'_k = 0$, where $\{i, j, k\} = \{1, 2, 4\}$. Fixing $i$ and $j$, this is the union of the two coordinate axes in $\C^2$ with coordinates $x_k, x_k'$. Thus, the deep locus has six components, each one isomorphic to the affine line $\C$, which is not a cluster variety. Note that all six components intersect at the point $(0, 0, 0, 0, -1, -1, 0, 0)$, which is the unique point that is stabilized by $\Aut(A)$. 
\end{remark}

\subsection{The \texorpdfstring{$T$}{T}-stabilizer locus of a braid variety} Similarly to the proof of Theorem \ref{thm:components-stabilizer-double-bs}, we can completely describe the components of the $T$-stabilizer locus of any braid variety $X(\beta)$. We remark, however, that in general the map $T \to \Aut(X(\beta))$ is not surjective, so this does not a priori give a description of the stabilizer locus of $X(\beta)$, see Example \ref{ex:T-not-Aut-for-braid-varieties}. We will assume that the Demazure product $\delta(\beta) = w_0$ is the longest element in $S_n$; by \cite[Lemma 3.4]{CGGLSS} we do not lose any generality in doing this. The following result is clear.

\begin{lemma}\label{lem:shuffle-standard}
Let $\{1, \dots, n\} = I_1 \sqcup I_2$ be a partition, with $|I_1| = a$ and $|I_2| = b$. Let $F_{1}^{\bullet} \in \Fl(a)$ and $F_{2}^{\bullet} \in \Fl(b)$ be the respective standard flags. Then $F_1^{\bullet} \shuffle_{e} F_2^{\bullet} \in \Fl(n)$ is the standard flag. 
\end{lemma}

\begin{lemma}\label{lem:shuffle-antistandard}
Let $w \in S_n$, and let $\{1, \dots, n\} = I_1 \sqcup I_2$ be a partition so that $I_1$ is a union of $ww_0$-orbits (and thus so is $I_2$). Let $G_1^{\bullet} \in \Fl(a)$ and $G_2^{\bullet} \in \Fl(b)$ be the corresponding anti-standard flags. Then, $G_1^{\bullet} \shuffle_{w} G_2^{\bullet} \in \Fl(n)$ is the anti-standard flag. 
\end{lemma}
\begin{proof}
    First, note that if $G_1^{\bullet}$ and $G_2^{\bullet}$ are anti-standard flags, then $G_1^{\bullet} \shuffle_{w_0}G_2^{\bullet}$ is again an anti-standard flag. By Remark \ref{rmk:coset}, $G_1^{\bullet} \shuffle_{w_0} G_2^{\bullet} = G_1^{\bullet}\shuffle _{w'w_0} G_2^{\bullet}$ for any $w' \in S_{I_1} \times S_{I_2}$. Since $I_1$ and $I_2$ are unions of $ww_0$-orbits, $ww_0 \in S_{I_1} \times S_{I_2}$. Taking $w' = ww_0$ concludes the proof.
\end{proof}

\begin{remark}
    Note that the converses of Lemmas \ref{lem:shuffle-standard} and \ref{lem:shuffle-antistandard} are also valid, that is, if $G_1^{\bullet}$ and $G_{2}^{\bullet}$ are such that $G_1^{\bullet} \shuffle_{w} G_2^{\bullet}$ is the (anti-)standard flag for some permutation $w \in S_   n$, then both $G_1^{\bullet}$ and $G_2^{\bullet}$ are (anti-)standard flags. 
\end{remark}

We have then the following analogue of Theorem \ref{thm:components-stabilizer-double-bs}.

\begin{theorem}\label{thm:components-Tstabilizer-braid-varieties}
Let $\beta \in \Br_{n}^{+}$ be a positive braid with $\delta(\beta) = w_0$. Let $w = \pi(\beta) \in S_n$ be the Coxeter projection. Decompose $\{1, \dots, n\} = I_1 \sqcup I_2$ in such a way that both $I_1$ and $I_2$ are unions of $ww_0$-orbits, with $|I_1| = a, |I_2| = b$. Let $\beta_1 \in \Br_{a}^{+}, \beta_2 \in \Br_{b}^{+}$ be the braids obtained by considering only the strands indexed by $I_1$ and $I_2$, respectively. Then:
\begin{enumerate}
    \item If $\delta(\beta_1) \in S_{I_1}$ and $\delta(\beta_2) \in S_{I_2}$ are the longest elements in their respective symmetric groups, then the set of element of $X(\beta)$ which are stabilized by $T_{I_1, I_2}$ is isomorphic to $X(\beta_1) \times X(\beta_2)$.
    \item Else, the set of elements in $X(\beta)$ which are stabilized by $T_{I_1, I_2}$ is empty. 
\end{enumerate}
\end{theorem}
\begin{proof}
In Case (1), the proof follows step by step that of Theorem \ref{thm:components-stabilizer-double-bs}, using Lemma \ref{lem:shuffle-antistandard} instead of Lemma \ref{lem:lu-decomposition} to conclude. In Case (2), assume that $\delta(\beta_1)$ is not the longest element of $S_{I_1}$. Similarly to the proof of Theorem \ref{thm:components-stabilizer-double-bs}, any element of $X(\beta)$ stabilized by $T_{I_1, I_2}$ must be of the form $(F^{\bullet} = G_j^{\bullet} \shuffle_{v_j} H_j^{\bullet})_{j}$ and these flags satisfy (1)--(4) of the proof of Theorem \ref{thm:components-stabilizer-double-bs}. But if $\delta(\beta_1)$ is not the longest element of $S_{I_1}$ then $G_{\ell}^{\bullet} \in \Fl(a)$ is not the anti-standard flag, where $\ell$ is the length of $\beta$. It follows that $G^{\bullet}_{\ell} \shuffle_{v_{\ell}} H_{\ell}^{\bullet}$ cannot be the antistandard flag, and thus we cannot have a point in $X(\beta)$ stabilized by $T_{I_1, I_2}$. 
\end{proof}

\begin{remark}
Note that $\BS(\beta) \cong X(\Delta\beta)$, cf. Lemma \ref{lem:bs-as-braid} and any two strands cross in a braid diagram for $\Delta\beta$. Thus, Case (2) of Theorem \ref{thm:components-Tstabilizer-braid-varieties} cannot happen for double Bott-Samelson varieties and hence the statement of Theorem \ref{thm:components-stabilizer-double-bs} is easier than that of Theorem \ref{thm:components-Tstabilizer-braid-varieties}. 
\end{remark}

\begin{remark}
In general, Case (2) of Theorem \ref{thm:components-Tstabilizer-braid-varieties} can happen. For example, take $\beta = \sigma_1\sigma_1\sigma_2\sigma_2\sigma_1\sigma_1$, cf. Example \ref{ex:T-not-Aut-for-braid-varieties}. We have $w = \pi(\beta) = e$ and $ww_0 = w_0 = [3,2,1]$. Thus, the only non-trivial decomposition of $\{1, 2, 3\}$ into unions of $ww_0$-orbits is $I_1 = \{1,3\}$, $I_2 = \{2\}$. But the first and third strands do not cross in a braid diagram for $\beta$, so $\delta(\beta_1) = e$
and we are in Case (2) of Theorem \ref{thm:components-Tstabilizer-braid-varieties}. Note that this is consistent with Example \ref{ex:T-not-Aut-for-braid-varieties}, which implies that the $T$-stabilizer locus of $X(\beta)$ is empty. 
\end{remark}

Note that we obtain the following corollary of Theorem \ref{thm:components-Tstabilizer-braid-varieties}, which generalizes Corollary \ref{cor:free-action}.

\begin{corollary}
    Let $\beta \in \Br^{+}_{n}$ be a braid such that any two strands cross in its braid diagram. Then $T$ acts freely on $X(\beta)$ if and only if the braid closure $\overline{\beta\Delta}$ is a knot, that is, $\pi(\beta)w_0$ is an $n$-cycle. 
\end{corollary}

We also have an analogue of Corollary \ref{cor:components-stabilizer-double-bs}, for which we need the following lemma.

\begin{lemma}
Let $\beta \in \Br_n^{+}$ be a positive braid with $\delta(\beta) = w_0$ and Coxeter projection $\pi(\beta) = w$. Then, for any $z \in X(\beta)$,
\[
\Stab_{(\C^{\times})^{n}}(z) \subseteq \{(t_1, \dots, t_n) \mid t_{ww_0(i)} = t_{i} \; \text{for every} \; i = 1, \dots, n\}.
\]
\end{lemma}
\begin{proof}
Let $t = (t_1, \dots, t_n) \in \Stab_{(\C^{\times})^{n}}(z)$, so that (cf. \eqref{eq: explicit action})
\[\diag(t_1, \dots, t_n)B_{\beta}(z) = B_{\beta}(z)\diag(t_{w(1)}, \dots, t_{w(n)}),\]
and
\[
\diag(t_{w_0(1)}, \dots, t_{w_0(n)})w_0B_{\beta}(z) = w_0B_{\beta}(z)\diag(t_{w(1), \dots, w(n)}).
\]
Since by assumption $z \in X(\beta)$, $w_0B_{\beta}(z)$ is upper triangular, 
we obtain
that
$t_{w_0(i)} = t_{w(i)}$ for every $i$, or equivalently $t_{i} = t_{w_0w^{-1}(i)}$ for every $i$. Taking inverses, the result follows.
\end{proof}

The proof of the following result now is identical to that of Corollary \ref{cor:components-stabilizer-double-bs}.

\begin{corollary}\label{cor:components-Tstabilizer-braid-varieties}
Let $\beta \in \Br^{+}_{n}$ be a positive braid, with $\delta(\beta) = w_0$. Then, the irreducible components of the $T$-stabilizer locus in $X(\beta)$ are given  by partitions $\{1, \dots,n\} = I_1 \sqcup I_2$, where $I_1$ and $I_2$ are subsets closed under the action of $\pi(\beta)w_0$ for which (1) of Theorem \ref{thm:components-Tstabilizer-braid-varieties} occurs. The component corresponding to $I_1, I_2$ is isomorphic to the braid variety $X(\beta_1) \times X(\beta_2)$, where $\beta_1, \beta_2$ are as described in Theorem \ref{thm:components-Tstabilizer-braid-varieties}.
\end{corollary}

Note that two distinct partitions $\{1, \dots, n\}$ as in Corollary \ref{cor:components-Tstabilizer-braid-varieties} can give the same irreducible component of the $T$-stabilizer locus. For example, if $\beta = \Delta$ then $X(\Delta)$ is a point (which is its own stabilizer locus), $\pi(\Delta)w_0 = e$, and any partition of $\{1, \dots, n\}$ gives $X(\Delta)$.

In Corollary \ref{cor:components-Tstabilizer-braid-varieties} $X(\beta_1) \times X(\beta_2)$ is indeed isomorphic to a braid variety $X(\beta')$ for $\beta'$ on $n$ strands obtained by permuting the indices of strands in $\beta_1 \beta_2$ by a permutation which sends $I_1$ to $\left\{1,\ldots,a\right\}$ and $I_2$ to $\left\{a+1,\ldots,a+b\right\}$. By multiplying $\beta'$ by a suitable reduced word as in \cite[Lemma 3.4]{CGGLSS}, we can find $\beta''$ with $\delta(\beta'') = w_0$ and 
\begin{equation} \label{eq:isomorphism-product-braid-varieties}
X(\beta'') \cong X(\beta_1) \times X(\beta_2).
\end{equation}
This is an analogue of the isomorphism \eqref{eq:BS-decomposition} for general braid varieties.

Similarly to Corollary \ref{cor:intersections-components}, iterating arguments in the proof of Theorem~\ref{thm:components-Tstabilizer-braid-varieties} and Corollary~\ref{cor:components-Tstabilizer-braid-varieties}, we obtain the following.

\begin{corollary} \label{cor:intersections-of-components}
Let $\beta \in \Br^{+}_{n}$ be a positive braid, with $\delta(\beta) = w_0$. Any set partition $I_1 \sqcup I_2 \sqcup \ldots \sqcup I_k$ of $\{1, \dots, n\}$ into sets which are stable under the action of $\pi(\beta)w_0$
(equivalently, of the set of components of the closure of $\beta\Delta^{\pm 1}$) gives rise to a subgroup $T_{I_1, I_2,\ldots,I_k} \subset (\mathbb{C}^\times)^n$. Let $\beta_j$ be the braid obtained by considering only the strands indexed by $I_j$, for $j = 1,\ldots,k$. Then, the set of elements in $X(\beta)$ which are stabilized by $T_{I_1, I_2,\ldots,I_k}$ is either:
\begin{enumerate}
    \item Isomorphic to $X(\beta_1) \times X(\beta_2) \times \cdots \times X(\beta_k)$, if, for all $j$, $\delta(\beta_j)$ is the longest element in $S_{I_j}$. \\
    \item Empty, otherwise.
\end{enumerate}

In particular, the intersection of sets of points stabilized by two subtori of  $(\mathbb{C}^\times)^n$ is either empty or isomorphic to the product of braid varieties associated with the coarsest common refinement of the corresponding partitions of orbits of $ww_0$.
\end{corollary}

\begin{remark} \label{rem:T-stab-positroids}
If the map $T \to \Aut(X(\beta))$ is surjective, then Corollaries \ref{cor:components-Tstabilizer-braid-varieties} and \ref{cor:intersections-of-components} describe the components of the stabilizer locus in $X(\beta)$ and their intersections.  In particular, these describe the components of the stabilizer locus of positroid varieties and their intersections, which in particular are again cluster varieties when nonempty. 
\end{remark}

\subsection{Link splitting} Let $\beta \in \Br_n^{+}$ and assume that $\{1, \dots, n\} = I_1 \sqcup I_2$ is a decomposition into $\pi(\beta)w_0$ stable subsets such that we are in Case (1) of Theorem \ref{thm:components-Tstabilizer-braid-varieties}. By construction, the inclusion of a subvariety isomorphic to $X(\beta'') \cong X(\beta_1) \times X(\beta_2)$ into $X(\beta)$ is $T$-equivariant. This induces a map in $T$-equivariant cohomology 
\begin{equation} \label{eq:equivariant-cohomology}
H^*_T(X(\beta)) \to H^*_T(X(\beta_1) \times X(\beta_2)) \cong  H^*_T(X(\beta'')). 
\end{equation}
It is known that the $T$-equivariant cohomology of the braid variety for $\beta$ with $\delta(\beta) = w_0$, considered with the weight filtration, gives a graded piece of Khovanov-Rozansky (KR) homology of (the smooth link given by the braid closure) of $\beta\Delta$ and of $\beta\Delta^{-1}$:
\begin{equation} \label{eq:link-homology}
H^*_T(X(\beta)) \cong  \mbox{HHH}^{a=n} (\beta \Delta) \cong \mbox{HHH}^{a=0} (\beta \Delta^{-1}),
\end{equation}
see \cite[Corollary 4]{T21}, \cite[Theorem 4.20]{CGGS_positroid}, and references therein for notation and details.\footnote{Here $a$ denotes the Hochschild degree.} Taking subsets of $I_1, I_2$ closed under the action of $ww_0$ corresponds to taking a partition of the set of connected components of either of these links. We note that in the context of KR homology and their deformations, one has \lq\lq splitting" maps from the homology of a link $L$ to the homology of the disjoint union $L_1 \sqcup L_2$, where $L = L_1 \cup L_2$ and $L_1, L_2$ each are unions of subsets of connected components of $L$ naturally corresponding to $I_1, I_2$, see e.g. \cite[Theorem 1.14, Section 4]{GH22} for the result for the ``$y$-ified'' deformation of KR homology. There, the disjoint union was obtained by taking the braid closure of a braid $\widetilde{\beta}\Delta^{\pm 1}$, where $\widetilde{\beta}$ is a braid with negative crossings, obtained from $\beta$ by changing the sign of exactly one half of (the even number of) crossings between strands indexed by $I_1$ and strands indexed by $I_2$. It is smoothly isotopic to the closure of the braid $\beta''\Delta^{\pm 1}$, for $\beta''$ from the isomorphism \eqref{eq:isomorphism-product-braid-varieties}, and also to the disjoint union of braid closures of $\beta_1\Delta_{S_{|I_1|}}^{\pm 1}$ and $\beta_2\Delta_{S_{|I_2|}}^{\pm 1}$.
Based on all this, we strongly suspect that the maps \eqref{eq:equivariant-cohomology} composed with isomorphisms \eqref{eq:link-homology} (and K\"unneth isomorphisms if we consider invariants of $\beta_1\Delta_{S_{|I_1|}}^{\pm 1}$ and $\beta_2\Delta_{S_{|I_2|}}^{\pm 1}$) agree with the restrictions of such splitting maps to appropriate graded pieces. In this sense, we expect that Theorem~\ref{thm:components-Tstabilizer-braid-varieties} can be interpreted as a geometric version of splitting maps in link homology theories via braid varieties.

\section{Applications to mirror symmetry}
\label{sec:mirror-symmetry}

This section explains how the previous results on deep loci of cluster varieties
fit in the context of Homological Mirror Symmetry (HMS). We use them to formulate
a conjectural criterion for when the Fukaya category of $G/P$ is generated by tori
(Conjecture \ref{conj:generation-by-tori}), and verify it for an infinite class of Grassmannians
(Theorem \ref{thm:fuk-gr(2,n)}); in such cases this also verifies that HMS holds when choosing as mirror Landau-Ginzburg model the one proposed by Rietsch for $G/P$ \cite{R} (see also \cite{MR} for the special case of Grassmannians). For the sake of readability, we treat the Fukaya category as a blackbox, and give references where the technical definitions and main foundational results used in the proof are established.

\subsection{Fukaya category of Fano manifolds}

A Fano manifold is a regular complex projective variety $X$ whose anti-canonical bundle
$K_X^{-1}$ is ample. Such manifolds have a natural symplectic structure, obtained by
pulling back the Fubini-Study form on projective space along an anti-canonical embedding. Following
\cite{Sh}, for any $\lambda\in\CC$ one can construct the Fukaya category
$\Fuk_\lambda(X)$, which is an $A_\infty$-category whose objects $(L,\xi)$ are
Lagrangian submanifolds $L\subset X$ endowed with a rank one $\CC$-linear
local system $\xi$, and whose Floer-theoretic curvature is equal to $\lambda$.
The curvature can be roughly described as a count of holomorphic disks in $X$ with boundary
on $L$, weighted by the holonomy of $\xi$. $\Fuk_\lambda(X)$ is known to be trivial
unless $\lambda\in\CC$ is an eigenvalue of multiplication by $c_1(TX)$ in the small quantum
cohomology ring $\QH(X)$: see \cite[Proposition 6.8]{Au} and \cite[Corollary 2.10]{Sh}.
It can be thought of as categorifying the subring
$\QH_\lambda(X)\subset\QH(X)$ of generalized eigenvectors with eigenvalue $\lambda$, in the
sense that its Hochschild cohomology is expected to be $\HH^*(\Fuk_\lambda(X))=\QH_\lambda(X)$
\cite{San, Gan}.
There is a notion of split-closed derived category of an $A_\infty$-category \cite[Chapter I]{Sei}
(it is an honest triangulated category),
and a key problem in symplectic topology is to describe a finite collection of objects that generates the split-closed derived Fukaya category $\Der\Fuk_\lambda(X)$ for specific Fano manifolds $X$.
Generators have been constructed for Fano hypersurfaces \cite{Sh}, Fano toric varieties \cite{EL}, and some Grassmannians \cite{Cas20, Cas23}.

\subsection{Mirror symmetry: closed-string and open-string}

A general aim of mirror symmetry is to compute symplectic invariants using algebraic geometry techniques. In the
Fano case, establishing closed-string mirror symmetry typicially consists in finding an affine complex algebraic variety $V$ and a regular
function $W\in\cO(V)$ such that $\QH(X)$ is recovered as the ring of regular functions
of the critical locus $\Crit(W)\subset V$, i.e. the locus where $dW=0$. When this happens,
one can ask if $\QH(X)\cong\cO(\Crit(W))$ upgrades to a statement about the
Fukaya category. Open-string mirror symmetry, or Homological Mirror Symmetry (HMS), holds if one can establish an equivalence of triangulated categories
$\Der\Fuk_\lambda(X)\simeq \Der\Sing(W^{-1}(\lambda))$ for all $\lambda\in\CC$, where the right hand
side is the derived category of singularities \cite{Or}. The latter measures
which coherent sheaves on $W^{-1}(\lambda)$ lack finite resolutions
by locally free sheaves.

\subsection{The case of \texorpdfstring{$G/P$}{G/P}}

A large class of Fano manifolds is given by homogeneous varieties $G/P$, where $G$ is a
simple simply connected complex algebraic group and $P\subset G$ is a parabolic subgroup.
In Lie theory, Langlands duality singles out dual groups $P^\vee\subset G^\vee$,
and thus a dual homogeneous variety $G^\vee/P^\vee$. In \cite{R} a mirror partner for $G/P$
was proposed, given by the top projected open Richardson stratum
$V^\vee\subset G^\vee/P^\vee$ \cite{KLS} together with an explicit function $W\in\cO(V^\vee)$, and the closed-string mirror symmetry $\QH(G/P)=\cO(\Crit(W))$ was verified (see also \cite{Ch22, C23}). The top projected Richardson stratum $V^\vee$ is smooth and
affine, and $W$ is known to have $0$-dimensional critical locus. It is thus natural to ask if HMS holds in this case.

\subsection{Conjecture on generation by tori} \label{subsec:conj-on-generation}

Top projected Richardson strata are isomorphic to varieties of the form $R(e, w)$ for certain $w$ (see e.g. \cite{S23} and references therein). Therefore, they are locally acyclic affine cluster varieties of really full rank. In particular, away from the deep locus they are covered by cluster charts. Observe that
the space of rank one $\CC$-linear local systems on a Lagrangian torus in $G/P$
is $(\CC^\times)^n$, where $n=\dim(G/P)$ is the complex dimension of the
homogeneous variety. Since $\dim(G/P)=\dim(V^\vee)$, it is natural to speculate that each
cluster chart of $T\subset V^\vee$ can be realized as space of local systems on some Lagrangian torus
$L_T\subset G/P$. If this happens, then the values of the cluster variables of $T$
at any critical point $p\in T$ of $W_{|T}$ are nonzero, and can be taken as holonomies of a local
system $\xi_p$ on $L_T$. A result in Floer theory guarantees that, if $L\subset X$ is a
Lagrangian torus, the nonzero objects
of the Fukaya category supported on $L$ are in one-to-one correspondence with
critical points of the disk potential $W_L$, which is a generating function counting holomorphic
disks in $X$ with boundary on $L$: see \cite[Proposition 6.9]{Au} and
\cite[Proposition 4.2]{Sh}. In our
setting above, it is natural to speculate that $W_{|T}=W_{L_T}$, and thus guess the following.

\begin{conjecture}\label{conj:generation-by-tori}
If $V^\vee$ has empty deep locus, then for each $\lambda\in\CC$ the category $\Der\Fuk_\lambda(G/P)$ is generated by objects supported on Lagrangian tori.
\end{conjecture}

A possible strategy to prove this conjecture consists of the following steps:
\begin{enumerate}
	\item to construct a Lagrangian torus $L_T\subset G/P$ from the cluster chart $T\subset V^\vee$;
	\item to explain why the disk potential of $L_T$ is the restriction $W_{L_T}=W_{|T}$;
	\item to prove that, for all $\lambda\in\CC$, $\Der\Fuk_\lambda(G/P)$ is generated
	by the objects $(L_T,\xi_p)$, with $T\subset V$ running in a collection of cluster
	charts covering the critical points of $W^{-1}(\lambda)$, and $p\in T$ in the set
	of critical points of $W_{|T}$ with critical value $\lambda$.
\end{enumerate}

If the deep locus of $V^\vee$ is non-empty, it can happen that some critical points of $W$
are contained in it. This is the case when $G/P=\Gr(2,4)$ and $V^\vee$ is the
top positroid stratum: in this case $4$ of the $6$ critical points of $W$ are in the
intersection of the two cluster charts, while the remaining two are in the deep locus
and have critical value $0$. In this example there seems to be no obvious mirror symmetry interpretation for the critical points in the deep locus as holonomies of a local system on a Lagrangian torus, so we don't know how to achieve step (1) above. This is why we do not conjecture that $\Der\Fuk_\lambda(G/P)$ is generated by objects supported on Lagrangian tori when the deep locus of $V^\vee$ is non-empty.

\subsection{Cluster duality and mirror symmetry} 

In the cluster algebra context, mirror symmetry usually manifests in form of \emph{Fock--Goncharov cluster duality}. There, two sides are given by cluster $\mathcal{A}$-varieties (i.e. unions of cluster charts) and cluster $\mathcal{X}$-varieties  for Langlands dual seeds; see e.g. \cite{HK18} for a survey. The general philosophy is than that partial compactifications of the variety on one side correspond to certain potentials on the other side which we denote $W_{\mathcal{A}}$ or $W_{\mathcal{X}}$, depending on the choice of sides. These are regular functions on cluster $\mathcal{A}$- and $\mathcal{X}$-varieties, and so can be also seen as regular functions on their affinizations, which is the perspective we are willing to follow. On the level of derived categories of singularities, the distinction between a cluster $\mathcal{A}$- (resp. $\mathcal{X}$-)variety and its affinization disappears as long as there are no critical points in the deep locus. Of course, this happens in particular whenever the deep locus is empty.

As discussed above, the top projected open Richardson strata $V \subset G/P$ and $V^\vee \subset G^\vee / P^\vee$ are affine cluster varieties. Further, it is shown in \cite{CGGLSS} that they are also affinizations of cluster $\mathcal{X}$-varieties, see Remark \ref{rmk:cluster-x}. It follows from \cite[Section 6.6, Section 8]{CGGLSS} that seeds on $V$ and on $V^\vee$ are cluster dual to each other. The Lie-theoretical potential $W$ matches the potential $W_{\mathcal{X}}$ for the $\mathcal{X}$-cluster structure on $V^\vee$ in the case of top open positroid strata in Grassmannians \cite{MR}. It also can be related to $W_{\mathcal{A}}$ for the $\mathcal{A}$-cluster structure on $V^\vee$ via the unimodular cluster ensemble automorphism, see \cite[Section 7, Remark 7.12]{BCMNC23}. That the same might be proved to be the case for all pairs $(V, V^\vee)$ is strongly suggested by recent work \cite{LRYZ24}.

\subsection{Results for Grassmannians}

The papers \cite{Cas20, Cas23} prove generation of the Fukaya category by Lagrangian tori for certain classes of Grassmannians. In such cases, the steps of the strategy outlined in Section \ref{subsec:conj-on-generation} are implemented as follows:

\begin{enumerate}
	\item If $T\subset V$ corresponds to a Pl\"{u}cker seed, \cite{RW}
	describes how to construct a degeneration of $\Gr(k,n)$ to a toric variety $X(\Sigma_T)$,
	whose polyhedral fan $\Sigma_T$ is the normal fan of an Okounkov body $\Delta_T(D_{FZ})$
	for the frozen divisor of $\Gr(k,n)$. One can think of $\Delta_T(D_{FZ})$ as the moment polytope
	of a densely defined Hamiltonian torus action on $\Gr(k,n)$, and take $L_T\subset\Gr(k,n)$
	to be a Lagrangian torus moment fiber.
	\item If the toric variety $X(\Sigma_T)$ has a small resolution of singularities,
	the fact that $\Delta_T(D_{FZ})$ is polar dual to the Newton polytope $\Newt(W_{|T})$
	and a result \cite{NNU} on the behavior of holomorphic disks under
	toric degeneration guarantee that $W_{|T}=W_{L_T}$. 
	\item If $\dim\QH_\lambda(\Gr(k,n))=1$, the generation criterion \cite[Corollary 2.19]{Sh}
	guarantees that any $(L_T,\xi_p)$ with $p\in T$ critical point with
	$W(p)=\lambda$ generates $\Der\Fuk_\lambda(\Gr(k,n))$.
\end{enumerate}

The characterization of the periphery of the top positroid stratum in the mirror Landau-Ginzburg model $V^\vee \subset\Gr(n-2,n)$ for $\Gr(2,n)$ obtained in Corollary \ref{cor:main:Section5} allows to extend the results of \cite{Cas20, Cas23} without modifying the main strategy and get the following.

\begin{theorem}\label{thm:fuk-gr(2,n)}
If $n$ is odd, then $\Der\Fuk_\lambda(\Gr(2,n))$ is generated by objects supported on Lagrangian tori and $\Der\Fuk_\lambda(\Gr(2,n))\simeq \Der\Sing(W^{-1}(\lambda))$ for all $\lambda\in\CC$.
\end{theorem}

\begin{proof}
All cluster charts $T\subset V^\vee$ are Pl\"{u}cker charts in this case, so the approach
to (1) described above produces a Lagrangian torus $L_T\subset\Gr(2,n)$ for each of them.
It is known \cite[Theorem 1.5]{NU} that the toric degenerations $X(\Sigma_T)$
all have small toric resolutions in this case, hence the approach to (2) described above guarantees
that the Laurent polynomials $W_{|T}$ obtained by restriction of $W\in\cO(V)$ to a cluster
chart $T\subset V^\vee$ match the disk potentials $W_{L_T}$ of the corresponding Lagrangian
tori $L_T\subset\Gr(2,n)$. For general Grassmannians, the critical values of $W\in\cO(V^\vee)$ are
the same as the eigenvalues of multiplication by
$c_1$ in $\QH(\Gr(k,n))$ \cite[Proposition 1.12]{Cas20}, and when $n$ is odd $\dim\QH_\lambda(\Gr(2,n))=1$ for
all eigenvalues $\lambda\in\CC$ \cite[Lemma 4.6]{Cas23}. This means that the approach to (3) described above
reduces the generation of any $\Der\Fuk_\lambda(\Gr(2,n))$ to showing the existence,
for any critical point $p$ of $W\in\cO(V^\vee)$ with value $\lambda\in\CC$, of a cluster
chart $T\subset V^\vee$ such that $p\in T$. The cluster charts cover $V^\vee$ away from
the deep locus, and since the deep locus in empty thanks to Corollary \ref{cor:main:Section5} such a chart $T$ exists. The equivalence $\Der\Fuk_\lambda(\Gr(2,n))\simeq \Der\Sing(W^{-1}(\lambda))$ holds because both sides are equivalent to the derived category $\Der(\operatorname{Cl}_d)$ of finitely generated modules over
the Clifford algebra $\operatorname{Cl}_d$ of the quadratic form or rank $d=2(n-2)$ on $\CC^d$. For the left hand side, \cite[Proposition 4.2-4.3, Corollary 6.5]{Sh} explains that the endomorphism algebra of a Lagrangian torus $(L_T,\xi_p)\in\Der\Fuk_\lambda(\Gr(2,n))$ generating the Fukaya category is a Clifford algebra, whose quadratic form is the Hessian of the disk potential $W_{L_T}$ at the local system with holonomy $p$. By closed-string mirror symmetry \cite[Proposition 9.2]{MR} we know that $\QH_\lambda(\Gr(2,n))\cong \cO(\Crit(W)\cap W^{-1}(\lambda))$, so when $n$ is odd the fact that $\dim\QH_\lambda(\Gr(2,n))=1$ implies
that there the critical point $p$ of the Landau-Ginzburg potential $W$ with critical value $W(p)=\lambda$ must be unique. Thanks to \cite[Proposition 1.14]{Or}, one can then localize the calculation of the category on singularities on the right hand side to the chosen cluster chart $T$ containing the only critical point $p$, and get $\Der\Sing(W^{-1}(\lambda))\simeq \Der\Sing(W^{-1}(\lambda)\cap T)$. Invoking \cite[Theorem 4.11, Corollay 2.7]{Dyc} the latter category is generated by the sky-scraper sheaf at the critical point $p$, whose endomorphism algebra is again a Clifford algebra (see Sheridan \cite[Corollary 6.4 and 6.5]{Sh} for how intrinsic formality of Clifford algebras is used to deduce the equivalence after matching the endomorphism algebras of generators).
\end{proof}

\bibliographystyle{plain}
\bibliography{bibliography.bib}

\end{document}